\documentclass[12pt]{amsart}
\usepackage{mathpazo}
\usepackage{hyperref}
\usepackage{bm}
\usepackage{verbatim}
\usepackage{slashbox}
\usepackage[margin=1in]{geometry}
\usepackage{array,multirow}
\usepackage{graphicx}
\usepackage{tikz}
\usepackage{tikz-cd}
\usepackage[all]{xy}
\usepackage{array}
\usepackage{cleveref}

\usepackage{mathrsfs}

\usepackage{amssymb,amsmath,latexsym,amsthm}%or any usual package
\newtheorem{theorem}{Theorem}[section]
\newtheorem{lemma}[theorem]{Lemma}
\newtheorem{proposition}[theorem]{Proposition}

\newtheorem{assumption}[theorem]{Assumption}

\newtheorem{corollary}[theorem]{Corollary}
\newtheorem{definition}[theorem]{Definition}
\newtheorem{preremark}[theorem]{Remark}
\newenvironment{remark}{\begin{preremark}\rm}{\end{preremark}}
\newtheorem{notation}[theorem]{Notation}
\newtheorem{preexample}[theorem]{Example}
\newenvironment{example}{\begin{preexample}\rm}{\end{preexample}}
\newtheorem{preclaim}[theorem]{Claim}

\newtheorem{prequestion}[theorem]{Question}

\newtheorem{preapplication}[theorem]{Application}

\numberwithin{equation}{section}

\newcommand \ZZ {{\mathbb Z}}
\newcommand \QQ {{\mathbb Q}}
\newcommand \NN {{\mathbb N}}
\newcommand  \FF {{\mathbb F}}

\newcommand \bC {{\mathbb C}}
\newcommand \bQ {{\mathbb Q}}

\newcommand \cO {{\mathcal O}}

\newcommand \CC {{\mathbb C}}

\newcommand \RR {{\mathbb R}}

\newcommand \cf {{\mathfrak f}}

\newcommand \bP {{\mathbb P}}

\newcommand \bH {{\mathbb H}}

\newcommand \bD {{\mathbb D}}
\newcommand{\fa}{\mathfrak{a}}

\newcommand{\cP}{\mathcal{P}}
\newcommand{\cQ}{\mathcal{Q}}

\newcommand{\Chi}{X}

\newcommand \fp {{\mathfrak p}}
\newcommand \hpi {{h}}

\newcommand \cl {\mathrm{cl}}

\global\let\ker\undefined
\DeclareMathOperator{\ker}{Ker}

\DeclareMathOperator{\End}{End}
\DeclareMathOperator{\Gal}{Gal}

\DeclareMathOperator{\GL}{GL}
\DeclareMathOperator{\SL}{SL}
\DeclareMathOperator{\SO}{SO}
\DeclareMathOperator{\GU}{GU}

\DeclareMathOperator{\Tr}{tr}

\DeclareMathOperator{\Id}{Id}
\DeclareMathOperator{\Lie}{Lie}

\DeclareMathOperator{\Res}{Res}

\newcommand \cA {{\mathcal A}}
\newcommand \CG {{\mathcal G}}

\newcommand \CH {{\mathcal H}}

\newcommand \Sh {\mathrm{Sh}}
\newcommand \Shg {\mathrm{S}}

\newcommand \pol {\mathrm{pol}}

\newcommand{\cU}{\mathcal{U}}
\newcommand{\bb}{v}
\newcommand{\jj}{j}
\newcommand{\kk}{k}

\newcommand{\U}{\mathcal{U}}

\newcommand{\bZ}{\mathbb{Z}}

\newcommand{\bF}{\mathbb{F}}
\newcommand{\Aut}{{\mathrm{Aut}}}

\newcommand{\bcP}{\overline{\cP}}

\newcommand{\bcQo}{\overline{\cQ^\circ_\lambda}}

\usepackage{tikz-cd}

\title{Infinitely many primes of basic reduction for some abelian fourfolds}
\date{}

\author{Wanlin Li}
\address{Department of Mathematics,
Vanderbilt University,
Nashville, TN 37235, USA}
\email{wanlin.li@vanderbilt.edu}

\author{Elena Mantovan}
\address{Department of Mathematics, California Institute of Technology, Pasadena, CA 91125, USA}
\email{mantovan@caltech.edu}

\author{Rachel Pries}
\address{Department of Mathematics, 
Colorado State University, 
Fort Collins, CO 80523, USA}
\email{pries@colostate.edu}

\author{Yunqing Tang}
\address{Department of Mathematics,
University of California at Berkeley,
Berkeley, CA 94720, USA}
\email{yunqing.tang@berkeley.edu}

\begin{document}

\thanks{
We would like to thank the American Institute of Mathematics for their support through the Square program.
Li was partially supported by NSF grant DMS-23-02511.
Mantovan was partially supported by NSF grant DMS-22-00694.
Pries was partially supported by NSF grants DMS-19-01819 
and DMS-22-00418. 
Tang was partially supported by NSF grants 
DMS-18-01237 and DMS-22-31958 and a Sloan Research Fellowship.
We would like to thank  
John Voight (for information about class groups and triangle groups), Liang Xiao (for suggesting this problem to us),
and Tonghai Yang (for discussion on complex Shimura curves and Heegner cycles).
We would also like to thank Eran Assaf, Francesc Castella, K\c{e}stutis \v{C}esnavi\v{c}ius, Ofer Gabber, Tom Graber, Emma Knight, Yuan Liu, and Peter Sarnak for helpful comments and/or discussions.
Some of the work was done when Tang was at CNRS and Universit\'e Paris-Saclay from February 2020 to June 2021 and at Princeton University from January 2019 to January 2020 and from July 2021 to June 2022.}

\begin{abstract}
If $E$ is an elliptic curve, defined over $\QQ$ or a number field having at least one real embedding,  
then Elkies proved that $E$ has supersingular reduction at infinitely many primes $p$.
Baba and Granath extended this result to certain curves $C$ of genus $2$ with field of moduli $\QQ$,
under a condition on the endomorphism ring of the Jacobian.
In this paper, we extend these results to certain curves of genus $4$ having an automorphism of order $5$, 
proving that the Jacobians of these curves have basic reduction (as defined by Kottwitz) for infinitely many primes $p$.

To do this, we study the complex uniformization of the Deligne--Mostow Shimura variety $\Sh$ associated with the one dimensional family of these curves.  
By analyzing the real points on $\Sh$, 
we compute three geodesics in the upper half plane that are edges of a fundamental triangle for the action of the unitary similitude group.
Using representations of quadratic forms, 
we determine the points on $\Sh$ which represent curves whose Jacobians have 
complex multiplication by certain quadratic extensions of the cyclotomic field $\QQ(\zeta_5)$.
We conclude by studying the equidistribution of these points and the reduction of these CM cycles on the Shimura variety.


MSC20 classifications: 
primary 11F06, 11G15, 11G18, 11M38, 14G35; 
secondary 11E12, 11R29, 14H10, 14K22, 32M15.

\begin{comment}
SAVE Primary:
11F06 Structure of modular groups and generalizations; arithmetic groups
11G15 Complex multiplication and moduli of abelian varieties
11G18 Arithmetic Geometry: Arithmetic aspects of modular and Shimura varieties 
11M38 Zeta and L-functions in characteristic p
14G35 Arithmetic questions, Modular and Shimura varieties

SAVE Secondary:
11E12 Quadratic forms over global rings and fields
11R29 Class numbers, class groups, discriminants
14H10 Curves in algebraic geometry, Families, moduli of curves (algebraic)
14K22 (Abelian varieties and schemes) Complex multiplication and abelian varieties
32M15 (Several complex variables and analytic spaces) (Complex spaces with a group of automorphisms) Hermitian symmetric spaces, bounded symmetric domains, Jordan algebras (complex-analytic aspects)
\end{comment}
\end{abstract}

\maketitle

%%%%%%%%%%%%%SECTION 1

\section{Introduction}

\subsection{Infinitely many primes of supersingular reduction}

If $E$ is an elliptic curve defined over $\QQ$,
then Elkies proved that there are infinitely many primes $p$ for which the reduction of $E$ modulo $p$ is supersingular \cite{Elkies1}.
Elkies also generalized this result for elliptic curves $E$ defined over other number fields, including those having at least one real embedding \cite{Elkies2}. In the work of Jao \cite{JAO,jao2003supersingular}, this result was extended to some elliptic curves parameterized by $\bQ$-points on modular curves $X_0(p)/\omega_p$ with small $p$, (including cases where $E$ is defined over an imaginary quadratic field).

For most curves $C$ of genus $g > 1$, not much is known about the 
primes of supersingular reduction of $C$.  
If the Jacobian of $C$ does not have complex multiplication, 
the expectation is that primes of supersingular reduction are rare for $C$. 
So it is intriguing to find situations where this set of primes is infinite. 

The result of Elkies was extended by Sadykov \cite{sadykov2004two} and  Baba--Granath \cite{babagranath}
to certain curves $C$ of genus $2$.
In the case of Baba--Granath, the curve $C$ has field of moduli 
$\QQ$, and its Jacobian $\mathrm{Jac}(C)$ has multiplication by the 
maximal quaternion order of discriminant 6. Under the condition that  
$C$ has potentially smooth stable reduction at 2 and 3, Baba--Granath \cite{babagranath} prove that 
$\mathrm{Jac}(C)$ has superspecial (and thus supersingular) reduction at infinitely many primes $p$.

In this paper, we extend the results of Elkies, Sadykov, and 
Baba--Granath to certain curves of genus $4$ having 
an automorphism of order $5$.
There are more possibilities for the Newton polygons of the reductions of these curves; the appropriate generalization of supersingular reduction is 
\emph{basic reduction}.
For the definition of basic reduction of these curves, see Section~\ref{Sbasic}, specifically Example~\ref{EbasicM11}.

Here is a simplified version of our main theorem, whose full statement can be found in Theorem~\ref{thm:M11}. In particular, Theorem \ref{Tintro} restricts to curves defined over $\bQ$ while Theorem~\ref{thm:M11} includes curves defined over $\bQ(\sqrt{5})$.

\begin{theorem} \label{Tintro}
Suppose $C_t$ is a smooth projective genus $4$ curve 
with an affine equation of the form 
\begin{equation} \label{EfamilyM11}
C_t \colon y^5=x(x-1)(x-t).
\end{equation}
Assume that the reduction of $C_t$ at $5$ is singular.
Suppose that $J(t):=(t^2 -t+1)^3/t^2(t-1)^2$ is in $\QQ \cap (-\infty, 27/4)$.
Then $\mathrm{Jac}(C_t)$ has basic reduction at infinitely many primes.
\end{theorem}

\subsection{An approach using moduli spaces and complex multiplication}
\label{Sapproach}

The essential idea of the paper is to study the family 
$C_t$ for $t \in \CC-\{0,1\}$, with a focus on values of $t$ for which the 
Jacobian $\mathrm{Jac}(C_t)$ has complex multiplication (CM).
The family of curves in \eqref{EfamilyM11} has several important properties 
which were studied in earlier papers of multiple authors, including Shimura \cite{shimuratranscend}, de Jong--Noot \cite{dejongnoot}, Moonen \cite{moonen}, and van Geemen--Sch\"utt \cite{geemenschutt}.

This family can be studied from many viewpoints:
as a Hurwitz space parametrizing cyclic covers of the projective line;
as a Deligne--Mostow Shimura variety $\Sh$ parametrizing abelian fourfolds with 
an action of $\mu_5$; as a quotient of the upper halfplane $\bH$ by a unitary
similitude group; or as a quotient of $\bH$ by a triangle group 
$\Delta(2,3,10)$.  

We use each of these perspectives to obtain key information. 
The Hurwitz space yields information about the Klein $J$-function 
$J(t)$ and the field of definition of $C_t$.
The Shimura variety perspective, together with Serre--Tate and Lubin--Tate theory,  gives information about Tate modules, 
basic reduction, $p$-divisible groups, and CM-cycles. 
The action of the unitary similitude group, or the triangle group, 
allows us to encode information 
about the real points $\Sh(\RR)$ using hyperbolic geodesics. 
Furthermore, we can describe the points of $\Sh(\RR)$ representing 
abelian fourfolds $\mathrm{Jac}(C_t)$ with complex multiplication 
by solutions to quadratic forms.

\subsection{Review of proof of Elkies}

Before giving a more technical description of the proof 
in Section~\ref{Sproof}, we recall some key points for the genus $1$ case.
Given an elliptic curve $E/\QQ$, Elkies wrote a `Euclid-style' proof to 
show that $E$ has infinitely many primes of supersingular reduction.

Given $D \equiv 0, 3 \bmod 4$, let $\mathcal{O}_D=\ZZ[(D+\sqrt{-D})/2]$.
For a prime $p$, the reduction $E_p$ is supersingular if and only if it has
complex multiplication by some ${\mathcal O}_D$ 
such that $p$ is ramified or inert in $\QQ(\sqrt{-D})/\QQ$.

Let $P_D(x)$ be the monic polynomial whose roots are the $j$-invariants of elliptic curves 
having complex multiplication by $\mathcal{O}_D$.  
Then $P_D(x) \in \ZZ[x]$ because these $j$-invariants 
are algebraic integers and conjugate under the action of 
the absolute Galois group $G_\QQ$.
If the $j$-invariant $j_E$ of $E$ is a root of $P_{D}(x)$ modulo $p$, then 
the reduction $E_p$ has complex multiplication by ${\mathcal O}_{D'}$ for some 
$D'$ such that $D/D'$ is a square.

Let $D=\ell$ or $D=4\ell$ for a prime $\ell \equiv 3 \bmod 4$.
Elkies proved that:
\begin{itemize}
\item[(i)] when working modulo $\ell$, the polynomial
$P_{\ell}(x)$ (resp.\ $P_{4\ell}(x)$) has a unique root of odd multiplicity, which is $1728$;
thus $P_\ell(x) \cdot P_{4\ell}(x)$ is the square of a polynomial modulo $\ell$.
\item[(ii)] $P_{\ell}(x)$ (resp.\ $P_{4\ell}(x)$) has a unique real root 
and it has limit $-\infty$ (resp.\ $+ \infty$) as $\ell \to \infty$.
\end{itemize}

Let $\Omega$ be the set of primes of supersingular reduction 
(and bad reduction) for $E$ and assume that $\Omega$ is finite.
There exist arbitrarily large primes $\ell$ such that $\ell \equiv 3 \bmod 4$ 
and $\binom{p}{\ell}=1$ for all primes $p \in \Omega$, 
where $\binom{*}{*}$ denotes the quadratic residue symbol.

Consider the even-degree polynomial $P_\ell(x) \cdot P_{4\ell}(x)$.  
By (ii), its value at $j_E$ is a negative rational number whose denominator is a square.
If its numerator is divisible by $\ell$ or by a prime
$\rho$ which is a quadratic non-residue modulo $\ell$, 
then the proof is complete because $\ell$ and $\rho$ are not in $\Omega$. 
If not, the fact that $\binom{-1}{\ell}=-1$ implies that $P_\ell(j_E) \cdot P_{4\ell}(j_E)$ 
is a quadratic non-residue modulo $\ell$, contradicting (i).
This proves that $\Omega$ is infinite.  

We note that the proof provides no congruence information 
about the primes in $\Omega$. It remains an interesting open problem whether $\Omega$ contains infinitely many primes satisfying a given congruence condition.

\subsection{Strategy of the proof} \label{Sproof}

The strategy in this paper shares broad outline with Elkies' proof; 
however, every step becomes more subtle and complicated.  
This includes: the properties of the polynomial analogue of $P_D(x)$;  
the parametrization of the family; the distinguished points in the family;
the arithmetic of CM fields of higher degree; and 
quadratic reciprocity and quadratic forms over $\QQ(\sqrt{5})$. 

Let $F = \bQ(\zeta_5)$, where $\zeta_5$ is a primitive fifth root of unity;
let $F_0$ denote its maximal totally real subfield.
Consider a totally positive element $\lambda\in F_0$ such that $\langle \lambda \rangle \subset \cO_{F_0}$ is a prime ideal; let $\lambda^\tau$ denote the $\Gal(F_0/\QQ)$-conjugate of $\lambda$.

Consider the Shimura curve $\Sh$ and the point $[C]\in \Sh(\QQ)$ representing the curve $C = C_t$.
On the Shimura curve $\Sh$, we consider Heegner cycles/sets of CM points $\tilde{Z}(\lambda)$ consisting of points/$J$-invariants corresponding to abelian varieties 
with CM by $\cO_F[\sqrt{-\lambda}]$. The reduction types of these CM points are well understood by the Shimura--Taniyama formula. In particular, to show that $C$ has a prime of basic reduction, we only need to show that there exists a prime $\fp$ of $\cO_{F_0}$ which is inert or ramified in $F_0(\sqrt{-\lambda})/F_0$ 
such that the $\bmod \ \fp$ reduction of $[C]\in \Sh(\QQ)$ coincides with the $\bmod \ \fp$ reduction of some point in $\tilde{Z}(\lambda)$.
Motivated from Elkies's argument, we use quadratic reciprocity for $F_0$ to reduce this task to analogues of statements (i) and (ii) above for $\Sh$.

The analogue of (i) is about the $\bmod \ \lambda$ reduction of $\tilde{Z}(\lambda)$. Vaguely speaking, for any $\bmod \ \lambda$ point $x_0$ of $\Sh$, the points in $\tilde{Z}(\lambda)$ whose reductions are $x_0$ show up in pairs, except when $x_0$ is the reduction of certain elliptic points\footnote{
These are the two elliptic points with automorphism groups of even order; these two points are the analogue of $1728$ in Elkies's proof.} of $\Sh$. 
We prove this property using Lubin--Tate theory, see \Cref{RootInPairs}. 
Due to the lack of a cusp on $\Sh$ 
(unlike the $j$-line for the classical modular curve), we no longer have a statement analogous to $P_D(x) \in \ZZ[x]$; we carry out a more refined analysis of the local behavior of $\tilde{Z}(\lambda)$ at the two elliptic points to deduce the full analogue of (i); see \Cref{coroSQ}.

The analogue of (ii) is about the $\RR$-points of $\tilde{Z}(\lambda)$.
For well-chosen $\lambda$, 
we prove that $\tilde{Z}(\lambda)$ has exactly two real points (\Cref{CMbyorderQP}). The analogue of (ii) is a statement about the relative positioning of $[C]$, the two real points of $\tilde{Z}(\lambda)$, and the two real points of $\tilde{Z}(\lambda^\tau)$ (see the figure in the proof of \Cref{thm:M11}). Using concrete computations on the complex uniformization of the Shimura curve, we prove the desired result by relating the position of the roots to the representability of primes by certain quadratic forms over $F_0$ (see Sections \ref{Scomplexunif},\ref{FundTr},\ref{SrealCM}) and by applying Hecke's equidistribution theorem (see \Cref{thm:infintely-many-lambda}).

Given a finite set of primes of basic reduction, 
we use these techniques to find (infinitely many) $\lambda$
which allow us to verify that we can always produce new primes of basic reduction.
This can be achieved as long as we find new primes of basic reduction
that do not divide $5$.  One way to deal with this issue is to require that $C$ 
and $\tilde{Z}(\lambda)$ do not have the same reduction type at $5$. 
Indeed, we prove that the Heegner points that we construct are Jacobians of smooth curves and thus $C_{\overline{\FF}_5}$, which is assumed to be singular, does not lie in $\tilde{Z}(\lambda)$; see \Cref{CMmod5}. Thus we can always construct more and more primes of basic reduction, 
finalizing the proofs of \Cref{Tintro} and \Cref{thm:M11}.

\subsection{Related work}

\begin{remark}
   For a curve $C$ as in \eqref{EfamilyM11},
in Cantoral-Farf\'an--Li--Mantovan--Pries--Tang
\cite[Corollary~5.1]{CLMPTmuordinaryreduction}, the authors prove that 
    the set of primes where the reduction of $C$ is not basic has density 1.
\end{remark}

\begin{remark}
The family \eqref{EfamilyM11} is \emph{special}, 
meaning that the image of the Torelli morphism is open and dense in $\Sh$.
Up to equivalence, 
there are exactly 20 special families of cyclic covers of ${\mathbb P}^1$;
of these $14$ are one-dimensional by the work of Moonen \cite{moonen}.
The family \eqref{EfamilyM11} is called $M[11]$ because it is the 11th entry of the table \cite[Table~1]{moonen}.

The result of Elkies on infinitely many primes of supersingular reduction is about the Legendre family, which is $M[1]$.
For $M[3,4,5,7,12]$, the curves in the family dominate a 
non-isotrivial family of elliptic curves. 
Applying Elkies' result, together with a short argument
about the decomposition of the Jacobians, implies that each curve with a suitable field of definition in these  
families has infinitely many primes of basic reduction.
 
We expect that the methods of this paper will yield a similar result for the family 
$M[17]$ consisting of curves of genus $6$ of the form $y^7=x(x-1)(x-t)$.
\end{remark}

\begin{remark}
By work of de Jong--Noot \cite[Proposition 2.7]{dejongnoot}, it 
was already known that infinitely many CM fields of degree $8$ 
occur for the Jacobians of the curves in \eqref{EfamilyM11}.  
The results in this paper provide more information about
the curves in the family whose Jacobian has CM by a particular CM field.
\end{remark}

\begin{remark}
In Section~\ref{Scomplexunif}, we provide 
a complex parametrization of the $M[11]$ family.
Another parametrization of \eqref{EfamilyM11} using projective embeddings and vanishing of theta nulls is given in van Geemen--Sch\"utt \cite[Section 4]{geemenschutt}.
In greater generality, one can find 
a numerical parametrization of compact Shimura curves, and their CM points, 
from the perspective of triangle groups by Klug-Musty-Schiavone-Voight in \cite{KMSVoight}, and by Voight in \cite{Voight06}.
\end{remark}

\bigskip

\subsection{Table of Contents}

For a paper of this length, we think the section headings provide the most efficient overview of the organization and contents of the paper.

\tableofcontents

%%%%%%%%%%%%%SECTION 2

%%%%%%%%%%%%%%%%%%%%%%%%%%%%%%%%%%%%%%
% Section 2
%%%%%%%%%%%%%%%%%%%%%%%%%%%%%%%%%%%%%%

\section{Moduli of cyclic covers and the Shimura variety}\label{sec_curvefamily}

In this section, we provide information about certain families of $\mu_m$-covers of the projective line $\mathbb{P}^1$, 
for a prime integer $m > 3$.
We suppose the covers are branched at $4$ points and that they have 
inertia type $a=(1,1,1,m-3)$.  
Over an algebraically closed field $k$ (whose characteristic is 0 or relatively prime to $m$),
each such $\mu_m$-cover has an affine equation of the form
\begin{equation} \label{Emumcover}
C_t: y^m = x(x-1)(x-t),
\end{equation} 
for some $t \in k - \{0,1\}$.  
Let 
$\hpi:C_t \to {\mathbb P}^1$ denote the $\mu_m$-cover taking $(x,y) \mapsto x$.  

In Section~\ref{Ssignature}, we determine the signature of the family.
In Section~\ref{subsec-special-points}, we determine the curves $C_t$ that have additional automorphisms.
In Section~\ref{Sjfunction}, we parametrize the family using the Klein $J$-function.
In Section~\ref{M11asSh}, we describe the 
Deligne--Mostov Shimura variety associated with the family \eqref{Emumcover}.
In Section~\ref{Sbasic}, we review the $\mu$-ordinary and basic Newton polygons for a Shimura variety of PEL-type, 
focusing on the families $M[11]$ (resp.\ $M[17]$) when $m=5$ (resp.\ $m=7$).

\subsection{Description of curves and signature types} \label{Ssignature}

Let $C_t$ be the smooth projective curve with equation $y^m = x(x-1)(x-t)$ as in \eqref{Emumcover}.
By the Riemann--Hurwitz formula, $C_t$ has genus $g=m-1$.
Let $\tau \in \mathrm{Aut}(C_t)$ be the automorphism $\tau((x,y))=(x, \zeta_m y)$.

For a fixed point $t \in {\mathbb C}$, the holomorphic differentials $H^0(C_t({\mathbb C}), \Omega^1_{C_t})$ 
form a $\mu_m$-module.  For $0 < n < m$, let $\cf_n$ denote the dimension of its $\zeta_m^n$-th eigenspace.    
For any $r\in \QQ$, let $\langle r\rangle$ denote the fractional part of $r$. 
By \cite[Lemma 2.7, \S3.2]{moonen}, 
$\cf_n = -1+\sum_{i=1}^4\langle\frac{-na_i}{m} \rangle$; 
this dimension is independent of $t$.
The \emph{signature type} is $\cf=(\cf_1, \ldots, \cf_{m-1})$.

When $a=(1,1,1,m-3)$, then $\cf_n = 2 - \frac{3n}{m} + \langle \frac{3n}{m}\rangle$.
In particular, when $m=5$, then $\cf=(2,1,1,0)$;
and when $m=7$, then $\cf=(2,2,1,1,0,0)$.

It will be convenient in later sections to adjust to a new signature $\cf'$ such that $\cf'_1 = 1$.
Note that $\cf_n = 1$, for $n=(m+1)/2$.
So, as in \cite[Lemma~2.1]{LMPTshimuradata}, this adjustment can be made using the automorphism 
$\sigma_2 \in \mathrm{Aut}(\mu_m)$.  
%SAVE This changes the inertia type to 
%$a'=((m+1)/2, (m+1)/2, (m+1)/2, (m-3)/2)$ and the signature to $\cf'$ such that $\cf'_n = \cf_{n(m+1)/2}$. 
In particular, when $m=5$, this changes the inertia type to $a'=(3,3,3,1)$ and $\cf' = (1,2,0,1)$; 
and when $m=7$, then $a'=(4,4,4,2)$ and $\cf'=(1,2,0,2,0,1)$.

\subsection{Curves in the family with extra automorphisms} \label{subsec-special-points}

\subsubsection{Two curves in the family with extra automorphisms} \label{SdefQR}

Let $C_R$ be the curve when $t=-1$.  It is given by the equation 
$y^m = x^3-x$.  Then $\mathrm{Aut}(C_R) \simeq \ZZ/m \ZZ \times \ZZ/2 \ZZ$; 
the extra automorphism of order two is given by $(x,y) \mapsto (-x,-y)$.
 
Let $C_Q$ be the curve when $t = - \zeta_3$.
The cross ratios of $\{\infty, 1, 0, - \zeta_3\}$ and $\{\infty, 1, \zeta_3, \zeta_3^2\}$ are the same.
So $C_Q$ is isomorphic to the curve with equation $y^m = x^3-1$.
Then $\mathrm{Aut}(C_Q) \simeq \ZZ/m \ZZ \times \ZZ/3 \ZZ$; with respect to the latter equation, 
the extra automorphism of order $3$ is given by $(x,y) \mapsto (\zeta_3 x, y)$.

\subsubsection{No other curves in the family have extra automorphisms}

\begin{proposition} \label{Pnomorebigaut}
Let $m>3$ be prime.  
Suppose $\hpi: C \to {\mathbb P}^1$ is a $\mu_m$-cover of smooth projective 
curves over $k$ that is branched at $4$ points and has inertia type $a=(1,1,1,m-3)$.
If $\#\mathrm{Aut}(C) > m$, then $C$ is isomorphic to either $C_R$ or $C_Q$.
\end{proposition} 

\begin{proof}
It suffices to prove the result over ${\mathbb C}$.
By \cite[Theorem~8.1, Table~7]{wootton}, the fact that $C$ has genus $g=p-1$ 
shows that $\langle \tau \rangle$ is normal in $\mathrm{Aut}(C)$, 
except possibly when $m=5$.
For $m=5$ and $g=4$, the only exception to $\langle \tau \rangle$ being normal in $\mathrm{Aut}(C)$ is when 
$C$ is Bring's curve.  
By \cite[Section~5.3]{bradennorthover}, the signature for the $\mu_5$-action on Bring's curve 
is $\cf=(1,1,1,1)$, which is not the signature for the family $\eqref{EfamilyM11}$.

Thus $\langle \tau \rangle$ is normal in $\mathrm{Aut}(C)$.
The result is then a special case of \cite[Proposition~3.6]{obusshaska}.
As a brief explanation, any $\sigma \in \mathrm{Aut}(C)$ descends to an automorphism $\bar{\sigma}$ of ${\mathbb P}^1$.
The automorphism $\sigma$ fixes the ramification point whose generator of inertia is different from the others.
Without loss of generality, this point maps to $\infty$ and so $\bar{\sigma}(x)$ fixes $\infty$.
By depressing the cubic, $C$ has an equation of the form $y^m=x^3 + Ax +B$.  This shows that $\bar{\sigma}(x) = ax$.
A case-by-case analysis shows that $C$ is isomorphic to $C_R$ or $C_Q$.
\end{proof}

Let $J_t=\mathrm{Jac}(C_t)$. 
Since $C_t$ is not hyperelliptic, $\mathrm{Aut}(J_t) \simeq \mathrm{Aut}(C_t)$ by \cite[Appendice]{lauterappendix}.

\subsubsection{A singular curve in the family} \label{SdefP}

Let  $D_1$ (resp.\ $D_2$) be the smooth projective curve with affine equation $y^m=x(x-1)$ 
(resp.\ $y^m=x^2(x-1)$).
The $\mu_m$-cover $\psi: D_i \to \bP^1$ is branched at three points, 
with inertia type $(1,1,m-2)$ when $i=1$ and $(2,1,m-3)$ when $i=2$.
Let $J_i=\mathrm{Jac}(D_i)$. 

Let $C_P$ denote the singular curve, whose irreducible components are $D_1$ and $D_2$, formed by identifying 
the point of $D_1$ above $\infty$ with the point of $D_2$ above $0$, in an ordinary double point.
The curve $C_P$ admits an admissible $\mu_m$-cover $\psi$ to a chain of two projective lines.
So $\psi$ can be deformed to a $\mu_m$-cover of ${\mathbb P}^1$ branched at $4$ points 
with inertia type $(1,1,1,m-3)$.
This implies that the moduli point of $C_P$ is in the boundary of the family \eqref{EfamilyM11}; 
it plays the role of being the third distinguished point in the family.
 
Note that $J_P=\mathrm{Jac}(C_P)$ decomposes, together with the product polarization, as $J_1 \oplus J_2$.
Thus $J_P$ has complex multiplication by $\QQ(\zeta_m)$.
Also, $J_P$ has an extra automorphism of order $2m$, given by $\mathrm{diag}[-\zeta_m^{-1},\zeta_m]$.

\subsection{The Klein $J$-function and field of definition} \label{Sjfunction}

Consider the Klein $J$-function 
\begin{equation} \label{Ejfunction}
J(t) = (t^2-t+1)^3/t^2(t-1)^2.  
\end{equation}

\begin{lemma} \label{Lfomfod}
Let $C_t:y^m=x(x-1)(x-t)$.
\begin{enumerate}
\item Then $C_{t_1}$ is geometrically isomorphic to $C_{t_2}$ if and only if $J(t_1)=J(t_2)$.
\item If $t \in \bar{\QQ}$, then the field of definition of $C_t$ is $\QQ(J(t))$.
\item The curve $C_Q$ has $J(-\zeta_3)=0$; the curve $C_R$ has
$J(-1)=27/4$ and $C_P$ has $J(\infty)=\infty$.
\end{enumerate}
\end{lemma}

\begin{proof}
\begin{enumerate}
\item 
There are three branch points $0,1,t$ of $\hpi: C_t \to {\mathbb P}^1$ that have the same generator of inertia.  
We consider a linear fractional transformation $L_t$ that moves these to $0,1,\infty$ respectively: 
it is $L_t(x) = (1-t)x/(x-t)$.  Then $L_t(\infty) = 1-t$.  
As a result, $C_t$ is isomorphic to the curve $C'_t : y^m= x(x-1)(x-(1-t))^{m-3}$.
It suffices to show that $C'_{t_1}$ is isomorphic to $C'_{t_2}$ if and only if $J(t_1)=J(t_2)$.

The function $J(t)$ is invariant under the six fractional linear transformations that stabilize $\{0,1,\infty\}$; 
in particular, $J(1-t) = J(t) = J(1/t)$.  It is the unique such function up to scaling.  

Suppose $J(t_1)=J(t_2)$.  Then there is a fractional linear transformation $L$ stabilizing $\{0,1,\infty\}$ such that $L(t_1)=t_2$.
So the composition of $C'_{t_1} \to {\mathbb P}^1$ with the map ${\mathbb P}^1 \to {\mathbb P}^1$ induced by $L$
is a $\mu_m$-cover branched at $\{0,1,t_2, \infty\}$ with inertia type $(1,1,m-3, 1)$.  
There is a unique such cover over $k$, thus $C'_{t_1}$ and $C'_{t_2}$
are geometrically isomorphic.

Conversely, suppose there is an isomorphism $\phi: C'_{t_1} \to C'_{t_2}$.
This proof uses the ideas in \cite[Propositions~4.1,4.2]{kontogeorgis};
the hypothesis on the number of branch points in those results is not necessary in this case because
there is a unique subgroup of order $m$ in the automorphism group.
Thus $\phi$ descends to ${\mathbb P}^1$. 
So $\phi$ acts via a fractional linear transformation $L$ on $x$.  
Also $L$ stabilizes $\{0,1,\infty\}$ because these values correspond to branch points with 
canonical generator of inertia $1$ and so $L(t_1)=t_2$.  
Thus $J(t_1)=J(t_2)$.

\item For the curve $C=C_Q$ (resp.\ $C=C_R$), 
the action of $\mathrm{Aut}(C)$ yields a cover $C \to {\mathbb P}^1$ branched at
three points with inertia groups of order $3,m,3m$ (resp.\ $2,m,2m$).
By \cite[Theorem~5.1]{Wolfart}, in this situation the field of moduli of $C$ is a field of definition.
Thus $C_Q$ and $C_R$ are defined over $\QQ$.
The same is true for the curve $C_P$, because the curves $D_1$ and $D_2$ are covers of 
${\mathbb P}^1$ branched at three points.

Let $C$ be a curve in the family \eqref{Emumcover} other than $C_Q$ or $C_R$.
Then the field of moduli of $C$ is a field of definition of $C$ by \cite[Theorem 1.1]{kontogeorgis}.
(The hypothesis that $2m$ is bounded by the number of branch points in that result is not necessary, because
$\mathrm{Aut}(C) = \langle \tau \rangle$ by Proposition~\ref{Pnomorebigaut}.) 

To determine the field of moduli, consider $\sigma \in \mathrm{Gal}(\bar{\QQ}/\QQ)$.  
The action of $\sigma$ takes $C_t$ to $C_{\sigma(t)}$, 
and thus takes $J(t)$ to $J(\sigma(t))=\sigma(J(t))$.
So $C_t$ is isomorphic to $\sigma(C_t)$ if and only if $J(t) = \sigma(J(t))$, or equivalently 
$\sigma \in \mathrm{Gal}(\bar{\QQ}/\QQ(J(t)))$. 
This implies that the field of moduli of $C_t$ is $\QQ(J(t))$.

\item Direct computation.
\end{enumerate}
\end{proof}

\subsection{The Deligne-Mostow Shimura variety}\label{M11asSh}

Given the data $\gamma = (m, 4, a)$,
there is a Hurwitz space parametrizing the family \eqref{Emumcover}
of $\mu_m$-covers of curves branched at $4$ points with inertia type $a$.  
Let $Z_\gamma$ be the closure of the image in $\mathcal{A}_g$ (of the projection to $\mathcal{M}_g$) of this Hurwitz space under the Torelli morphism.

Given the degree $m$ and signature type $\cf$, 
there is an associated PEL-type moduli stack 
$\Sh(\mu_m, \cf)$ introduced by
Deligne and Mostow \cite[2.21, 2.23]{deligne-mostow}.
It is defined over $\QQ(\zeta_m)$.
In general, the image of $\Sh(\mu_m, \cf)$ in $\mathcal{A}_g$ contains $Z_\gamma$.

\begin{notation}
When $a = (1,1,1,m-3)$ with $m=5$ (resp.\ $m=7$), 
we denote the Shimura variety $\Sh(\mu_m, \cf)$ by $\Sh$, 
or by $M[11]$ (resp.\ $M[17]$) as in \cite[Table~1]{moonen}.
\end{notation}

\begin{lemma} \label{LM11defQ}
Suppose $a = (1,1,1,m-3)$ with $m=5$ or $m=7$.

Then $Z_\gamma$ is a projective line with three marked points defined over $\QQ$.

Also $\Sh=\Sh(\mu_m, \cf)$ has dimension $1$ and is connected.
Furthermore, $\Sh =Z_\gamma$.
\end{lemma}

\begin{proof}
The first statement follows from Lemma~\ref{Lfomfod}. 

When $a=(1,1,1,m-3)$, with $m=5$ or $m=7$, then $\mathrm{dim}(\Sh)=1$
because of the signature $\cf$ computed in Section~\ref{Ssignature}.
The signature condition forces each principally polarized abelian variety corresponding to a point on $\Sh$ to admit a unique 
$\ZZ[\zeta_m]$-action up to equivalence. 
Thus $\Sh$ is a subvariety of ${\mathcal A}_g$.
%SAVE In general, a PEL type only admits a finite map to $A_g$; our particular signature condition guarantees injectivity, but this is not true in general.
So $Z_\gamma$ is a connected component of $\Sh(\mu_m, \cf)$.
Furthermore, in these cases, $\Sh(\mu_m, \cf)$ is connected by \cite{shimuratranscend}.  Hence $\Sh =Z_\gamma$.
%SAVE More precisely, as we work with no level structure, the number of connected components is the same as isometry classes in the same genus (for this argument, see for instance, BHKRY-I Prop 2.1.1), which is $1$ in this case by Shimura as the field has class number $1$; this is also what Gross claimed in his talk in Kudla 70 for the $m=7$ case.}
\end{proof}

\subsection{The $\mu$-ordinary and basic locus} \label{Sbasic}

Consider a Shimura variety $\Shg$ of PEL-type.
For a (good) prime $p$,  
in \cite[\S 5]{kottwitz1} and \cite[\S 6]{kottwitz2}, Kottwitz introduced a partially ordered set $B$ of Newton polygons.
In \cite[Theorem 1.6]{viehmann-wedhorn}, Viehmann and Wedhorn proved that these all occur on $\Shg$. 

Kottwitz proved that $B$ has a maximal element 
(called the \emph{$\mu$-ordinary} Newton polygon)
and a minimal one (called the \emph{basic} Newton polygon).
The $\mu$-ordinary Newton polygon occurs on an open dense subset of $\Shg$.
If $\Shg$ has dimension $1$, then for each prime, there are only two Newton polygons in 
$B$ and the locus of $\Shg$ where the basic Newton polygon occurs is closed.

For $\Sh(\mu_m, \cf)$,
a prime $p$ is good if and only if $p \nmid m$.
The set $B=B(\mu_m,\cf)$ depends on the congruence of $p$ modulo $m$.
The elements in $B(\mu_m, \cf)$ 
are symmetric convex polygons, with endpoints $(0,0)$ and $(2g,g)$, integral break-points, 
and rational slopes in $[0,1]$. 

\begin{notation}
Let $ord$ be the Newton polygon  $\{0,1\}$ and $ss$ be the Newton polygon $\{1/2,1/2\}$.
Let $\oplus$ denote the union of multi-sets.  
For any multi-set $\nu$, and $n\in\NN$, we write $\nu^n$ for $\nu\oplus \cdots \oplus \nu$, $n$-times. 
Thus  $ord^g$ (resp.\ $ss^g$) denotes the Newton polygon for 
an ordinary (resp.\ supersingular) abelian variety of dimension $g$.
For $s,t\in \NN$, 
with $s \leq t/2$ and $\mathrm{gcd}(s,t)=1$, let $(s/t, (t-s)/t)$ denote the Newton polygon with slopes $s/t$ and  $(t-s)/t$, each with multiplicity $t$.
\end{notation}

\begin{example} \label{EbasicM11} \cite[Section 6.2]{LMPT2}
For the family $M[11]$, with $m=5$ and $a=(1,1,1,2)$ and $g=4$, 
the $\mu$-ordinary and basic Newton polygon 
are as follows:
\[\begin{array}{|c|c|c|c|}
\hline 
\mod 5 &p\equiv 1 & p\equiv 4 &p\equiv  2,3\\
\hline

\mu-\mathrm{ord} & \mathrm{ord}^4 & \mathrm{ord}^2\oplus \mathrm{ss}^2 & (1/4,3/4) \\
\hline

\text{basic} &  \mathrm{ord}^2\oplus \mathrm{ss}^2 & \mathrm{ss}^4 &\mathrm{ss}^4\\
\hline
\end{array}\]

In \cite[Theorem~5.11]{LMPT2}, we proved that the basic Newton polygon occurs for the Jacobian 
of a {\it smooth} curve in the family $M[11]$, 
under the condition that $p \gg 0$ when $p \not \equiv 1 \bmod 5$.
\end{example}

\begin{example} \cite[Section 6.2]{LMPT2}
For the family $M[17]$, with $m=7$ and $a=(1,1,1,4)$ and $g=6$, 
the $\mu$-ordinary and basic Newton polygon 
are as follows:
\[\begin{array}{|c|c|c|c|c|}
\hline 
\mod 7 &p\equiv 1 & p\equiv 2,4 &p\equiv  3,5 & p \equiv 6\\
\hline

\mu-\mathrm{ord} & \mathrm{ord}^6 & \mathrm{ord}^3\oplus (1/3, 2/3) & (1/3,2/3)^2 & \mathrm{ord}^2\oplus \mathrm{ss}^4 \\
\hline

\text{basic} &  \mathrm{ord}^4\oplus \mathrm{ss}^2 & (1/6,5/6) & \mathrm{ss}^6 &\mathrm{ss}^6\\
\hline
\end{array}\]

\end{example}

%%%%%%%%%%%%%SECTION 3

%%%%%%%%%%%%%%%%%%%%%%%%%%%%%%%%%%%%%%
% Section 3
%%%%%%%%%%%%%%%%%%%%%%%%%%%%%%%%%%%%%%

\section{Structure of complex multiplication}

Let $m > 3$ be prime and let $\zeta_m = e^{2 \pi i/m} \in \CC$. 
Consider $F=\QQ(\zeta_m)$ which is a CM field over $\QQ$ with  
maximal totally real subfield $F_0=\QQ(\zeta_m + \zeta_m^{-1})$.

Our goal is to study curves of genus $m-1$ in the family \eqref{Emumcover} given by the affine equation $y^m=x(x-1)(x-t)$ 
whose Jacobians have complex multiplication by certain degree two extensions of $F$ that are CM fields.
The main outputs of the section are:
Theorem~\ref{Tuniqueconganym}, which proves a uniqueness statement for 
principally polarized abelian varieties defined over $\RR$ with certain CM types that arise in this context;
and Proposition~\ref{ST}, in which we produce congruence classes of primes of basic reduction for abelian varieties 
with these CM types when $m=5$, using the Shimura--Tanayama formula.

\subsection{Construction of a CM extension} \label{Scmextension}
 
\begin{assumption}\label{def:good}
    Throughout the paper, we assume that $\lambda \in {\mathcal O}_{F_0}$ is totally positive, is relatively prime to $m$, generates a prime ideal, and has odd norm in $\QQ$.
    We further assume that $-\lambda$ is a square modulo $4{\mathcal O}_{F_0}$.
\end{assumption}

When there is no ambiguity, we denote by $\lambda$ also the ideal in $\mathcal{O}_{F_0}$ generated by $\lambda$.
Define
\begin{equation} \label{EdefEfield}
E=F(\sqrt{-\lambda}), \ \mbox{and} \ E_0=F_0((\zeta_m -\zeta_m^{-1})\sqrt{-\lambda}).
\end{equation}
Then $E$ is a CM field and $E_0$ is its maximal totally real subfield. 

The next lemma explains the reason for the last condition on $\lambda$.

\begin{lemma}\label{Checkram2Mod4}
Let $\mathfrak{p}$ be a prime of $F_0$ dividing $2$.
The last condition in Assumption~\ref{def:good}
(that $-\lambda$ is a square modulo $4{\mathcal O}_{F_0}$) is equivalent to $\mathfrak{p}$ being unramified in 
$F_0(\sqrt{-\lambda})$. 
\end{lemma}

\begin{proof}
By \cite[Exercise 9.3]{washington},
%SAVE or \cite[Remark A.6, page 223]{milnenotes}: with $\pi = -2$.
%See also Frohlich 1967, Exercise 2.12, p353. 
if $a \in F_0^*$ is a non-square relatively prime to $2$, and 
if $\mathfrak{p}$ is a prime of $F_0$ dividing $2$, 
then $\mathfrak{p}$ is unramified in $F_0(\sqrt{a})$ if and only if $a \equiv X^2 \bmod 4{\mathcal O}_{F_0}$ has a solution $X$ (of odd norm) in ${\mathcal O}_{F_0}$.
Setting $a=-\lambda$ completes the proof.
\end{proof}

\begin{lemma} \label{Lotherram}
Under Assumption~\ref{def:good}:
$E/F$ is ramified only over the primes of $F$ above $\lambda$;
also $E/E_0$ is ramified at no finite prime.
\end{lemma}

\begin{proof}
The extension $E/F_0$ is biquadratic, with the three intermediate field extensions
being $E_0$, $F$, and $F_0(\sqrt{-\lambda})$.
The extension $F/F_0$ ramifies at $\sqrt{m}$ and the infinite primes by \cite[Proposition 2.15]{washington}. 
%Save: the ramification above infinite primes is not used later. 
This, together with Lemma~\ref{Checkram2Mod4}, implies 
that $E/F_0$ is not ramified at any prime of $F_0$ above $2$.
So $F_0(\sqrt{-\lambda})/F_0$ is ramified only at $\lambda$ and the infinite primes.
Any prime of odd norm that ramifies in a biquadratic extension 
has a cyclic inertia group, and thus is
ramified in exactly two of the three intermediate 
degree two field extensions of $F_0$.
We deduce that $E/F$ ramifies only at $\lambda$ and $E/E_0$ ramifies only at the infinite primes.
\end{proof}

When $K$ is a number field, we use $\cl_K$ to denote its ideal class group. In the next result, we study the parity of the class numbers of $E$ and $E_0$.

\begin{proposition}\label{Poddclassnumber}
Under Assumption~\ref{def:good}, suppose also that 
$\lambda$ is inert in the extension $F/F_0$.
If $|\cl_F|$ is odd, then $|\cl_E|$ and $|\cl_{E_0}|$ are odd.
\end{proposition}

\begin{proof}
Since $\lambda$ is inert in $F/F_0$, there is one prime of $F$ above $\lambda$.  
By Lemma~\ref{Lotherram}, 
$E/F$ is a $2$-group extension ramified at only one prime.
In this situation, by \cite[Theorem 10.4]{washington}, if $|\cl_E|$ is even then $|\cl_{F}|$ is even.
Thus by the assumption of $|\cl_F|$ being odd, we deduce that $|\cl_E|$ is odd.
Since $E_0$ is the maximal totally real subfield of the CM field $E$, 
it follows from \cite[Theorem~4.10]{washington} that $|\cl_{E_0}|$ divides $|\cl_{E}|$. 
\end{proof}

\subsection{Totally positive units}

Let $\mathcal{U}_{E_0}^+$ denote the totally positive units of $E_0$.
Let $N_{E/E_0}:E \to E_0$ denote the norm map.
Since $E$ is a CM field, quadratic over its maximal totally real subfield $E_0$,
it follows that $N_{E/E_0}(\mathcal{U}_E) \subseteq \mathcal{U}_{E_0}^+$.
In this section, we prove that $N_{E/E_0}(\mathcal{U}_E)=\mathcal{U}_{E_0}^+$ when $\lambda$ satisfies certain congruence conditions.

Recall the {\em Hasse unit index} of the CM extension $E/E_0$ is 
defined as $Q(E):=[\mathcal{U}_E:\mu_E\mathcal{U}_{E_0}]$, where $\mu_E$ is the group of roots of unity of $E$.
By \cite[Theorem 4.12]{washington},
%SAVE Hasse \cite[Satz 14]{bookHASSE}, 
$Q(E)=1$ or $2$. 

Since $\ker(N_{E/E_0})=\mu_E$ and $N_{E/E_0}(\mathcal{U}_{E_0}) = \mathcal{U}_{E_0}^2$, it follows that 
\begin{equation} \label{EQexact}
Q(E)=[N_{E/E_0}(\mathcal{U}_E):\mathcal{U}_{E_0}^2].
\end{equation}

Let $n=\mathrm{deg}(E_0/\QQ)$.
We fix an ordering $\tau_1, \ldots, \tau_n$ of the $n$ real embeddings $E_0 \hookrightarrow \RR$. 
Consider the group homomorphism
\begin{equation} \label{Edefrho}
\rho_{E_0}:  \mathcal{U}_{E_0} \to \{\pm 1\}^{n}, \ \rho_{E_0}(u)= (\tau_i(u)/|\tau_i(u)|)_{1 \leq i \leq n} \ \mbox{ for } 
u \in \mathcal{U}_{E_0}.
\end{equation}
Following \cite[Lemma 11.2, Definitions 12.1, 12.13]{bookCH}, 
we say that $E_0$ has {\it units of independent signs} if $\rho_{E_0}$ is surjective and 
that $E_0$ has {\it units of almost independent signs} if $|\mathrm{coker}(\rho_{E_0})|=2$.

%SAVE Definition 12.1 for independent signs: 
%Definition 12.13 for almost independent signs.
%Connection with $\rho_{E_0}$ in \cite[Lemma~11.2]{bookCH}.

\begin{proposition}\label{Poddclassuniquepp}
Under Assumption~\ref{def:good}, suppose also that 
$\lambda$ is inert in the extension $F/F_0$.
If $|\cl_F|$ is odd, then $E_0$ has units of almost independent signs, $Q(E)=2$, and $[\cU_{E_0}^+:N_{E/E_0}(\cU_E)] =1$.
\end{proposition}

\begin{proof} 
By Lemma~\ref{Lotherram}, $E/E_0$ is unramified at all finite primes.
The hypotheses of Proposition~\ref{Poddclassnumber} 
are satisfied, so $|\cl_E|$ is odd.
The facts that $|\cl_E|$ is odd and $E/E_0$ does not ramify at finite primes imply that $E_0$ has units of almost independent signs by 
\cite[Corollary 13.10]{bookCH} and $Q(E)=2$ by a theorem of Kummer \cite[Theorem 13.4, page 73]{bookCH}.

We have a sequence of inclusions of groups
\[\cU_{E_0}^2 \subseteq N_{E/E_0}(\cU_E) \subseteq \cU_{E_0}^+ \subseteq \cU_{E_0}.\]

Let $r$ be the rank of $\cU_{E_0}$, then $[\cU_{E_0}:\cU_{E_0}^2]=2^r$.
Since $E_0$ has units of almost independent signs, $[\cU_{E_0}:\cU_{E_0}^+] = 2^{r-1}$. 
By \eqref{EQexact}, $[N_{E/E_0}(\mathcal{U}_E):\mathcal{U}_{E_0}^2]=Q(E)=2$. 
Thus $[\cU_{E_0}^+:N_{E/E_0}(\cU_E)] =1$.
\end{proof}

\subsection{Quadratic reciprocity}

Recall that $F_0=\QQ(\zeta_m+\zeta_m^{-1})$.
Let $N:F_0 \to \QQ$ be the norm map.
Suppose $\alpha, \beta \in {\mathcal O}_{F_0}$ 
have odd norm in $\ZZ$.

Recall the quadratic Legendre symbol (\cite[Chapter 12, Section 4]{ArtinTate}) which takes values in $\{\pm 1\}$.
If $\alpha$ and $\beta$ are relatively prime, 
and $\beta$ is prime, it is defined by
$$\left(\frac{\alpha}{\beta}\right)= \alpha^{(N(\beta)-1)/2} \bmod \beta.$$
When $\beta$ is prime, $\left(\frac{\alpha}{\beta}\right)=1$ if and only if $\alpha$ is a square in 
${\mathcal O}_{F_0}/\langle \beta \rangle$.
If $u$ is a unit, then $\left(\frac{\alpha}{u}\right)= 1$.

We recall the quadratic Hilbert symbol (\cite[Chapter 14]{SerreLF}).
For $\nu$ a prime of ${\mathcal O}_{F_0}$,  
writing $K_\nu=(F_0)_\nu$ for the local field, it is
the symmetric non-degenerate symbol
defined by \[(\alpha, \beta)_\nu = \begin{cases} 
1 & \text{if $\beta$ is a norm of an element in } K_v(\sqrt{\alpha}), \\ 
-1 & \text{otherwise}.
\end{cases}\]

By a classical result of Hasse \cite{Hasse} (see \cite[Page 171, Corollary]{ArtinTate}),
\begin{equation} \label{EflipQR}
\left(\frac{\alpha}{\beta}\right)\left(\frac{\beta}{\alpha}\right) =
\prod_{\nu \mid 2 \infty} (\alpha, \beta)_\nu.
\end{equation}

If either $\alpha$ or $\beta$ is totally positive, then 
$\prod_{\nu \mid \infty} (\alpha, \beta)_\nu=1$.

For $\nu \mid 2$, by Hensel's Lemma, $(\alpha, \beta)_\nu$ is determined 
by the congruence of $\alpha$ and $\beta$ modulo $4{\mathcal O}_{F_0}$.
If $\alpha$ is a square modulo $4{\mathcal O}_{F_0}$ of an element of odd norm, then $(\alpha, \beta)_\nu = 1$.
Also $(1-\alpha, \alpha)_\nu = 1$.

\begin{lemma} \label{QR}
Under Assumption~\ref{def:good}:
suppose $\lambda, \beta \in {\mathcal O}_{F_0}$ have odd norm in $\ZZ$ and are relatively prime.
If $\beta$ is totally positive, then
$
\left(\frac{-\lambda}{\beta}\right)
\left(\frac{\beta}{\lambda}\right)=1$.
\end{lemma}

\begin{proof}
By \eqref{EflipQR}, 
\begin{eqnarray*}
\left(\frac{-\lambda}{\beta}\right)
\left(\frac{\beta}{\lambda}\right)= \left(\frac{-1}{\beta}\right)\left(\frac{\lambda}{\beta}\right)
\left(\frac{\beta}{\lambda}\right)=
\left(\frac{\beta}{-1}\right)\prod_{\nu \mid 2 \infty} (-1, \beta)_\nu 
\prod_{\nu \mid 2 \infty} (\lambda, \beta)_\nu.
\end{eqnarray*}
Note that $\left(\frac{\beta}{-1}\right)=1$.  
The hypothesis that $-\lambda$ is a square modulo $4{\mathcal O}_{F_0}$ implies that 
$(-1, \beta)_\nu= (\lambda, \beta)_\nu$ for each $\nu \mid 2$.
Also $(-1, \beta)_\nu= (\lambda, \beta)_\nu = 1$ for each $\nu \mid \infty$ because $\beta$ is totally positive.
\end{proof}

When $m=5$ then $F_0= \mathbb{Q}(\sqrt{5})$, and $F_0$ has narrow class number $1$. Hence, an element $\lambda\in\cO_{F_0}$ is prime if and only if it is irreducible.
Denote by $\tau$ the nontrivial automorphism in $\Gal(F_0/\QQ)$.
Consider the unit $u=(1+\sqrt{5})/2$.
Note that ${\mathcal O}_{F_0} = [1, u]_{\ZZ}$.
By direct computation (see also \cite[Chapter 12, page 15]{lemmermeyer}), we obtain the following:

\begin{lemma}\label{quadrec5} For $F_0=\QQ(\sqrt{5})$, let $\lambda\in \cO_{F_0}$ be an irreducible, totally positive element
which is relatively prime to $2 \sqrt{5}$.
\begin{enumerate}
    \item Then Assumption~\ref{def:good} is satisfied 
    if and only if $-\lambda \bmod 4\cO_{F_0}$ is in $ \{ 1, 1+u, 1+u^\tau\}$, or equivalently, if and only if $u_0\lambda\equiv -1\bmod 4\cO_{F_0}$ for some $u_0\in \mathcal{U}_{F_0}^+$;
    % SAVE $\lambda \equiv 3,-\frac{(3\pm\sqrt{5})}{2} \bmod 4$);
    \item if this condition is satisfied, then 
    $\left(\frac{-1}{\lambda}\right)=1$, $\left(\frac{u}{\lambda}\right)=-1$, and
    $\left(\frac{u^\tau}{\lambda}\right)=-1$.
\item 
The ideal $\langle \lambda \rangle$ is inert in $F/F_0$ if and only if 
the rational prime $p$ under $\lambda$ satisfies $p \equiv 2,3,4 \bmod 5$, which is equivalent to
$N_{F_0/\QQ}(\lambda) \equiv 4 \bmod 5$.    
\end{enumerate}
\end{lemma}

In later sections, we specialize to the case $\lambda\equiv -1\bmod 4 \cO_{F_0}$.

\subsection{CM types}\label{CMtype}

Recall $m>3$ is prime and consider the CM field $F = \QQ(\zeta_m)$. 
We define the following embeddings, for $1 \leq j \leq m-1$:
\begin{equation}\label{EembedF}
\sigma_j: F \to \CC, \ \sigma_j(\zeta_m) = e^{2 \pi i \cdot j /m}.
\end{equation}

Let $\cf$ be a signature for $\mu_m$ as was defined in Section \ref{Ssignature}. 
Consider the PEL Shimura variety $\Shg=\Sh(\mu_m,\cf)$ as was described in \cite{LMPTshimuradata}. 
Assume $\dim(\Shg)=1$; this is equivalent to
$0\leq \cf_j\leq 2$ for all $1\leq j\leq m-1$, and $\cf_j\cf_{m-j}=0$ for all but two $1\leq j\leq m-1$. Such a signature $\cf$ is called \textit{simple.}

As in Section \ref{Scmextension}, let $E=F(\sqrt{-\lambda})$.
A complex embedding of $E$ is determined by $(\sigma, \pm)$, where $\sigma:F \hookrightarrow \bC$ is an embedding
and $\pm$ is the sign
of the imaginary part of the image of $\sqrt{-\lambda}$ under $\sigma$.  Complex conjugation acts by 
$(\sigma, \pm) \mapsto (\bar{\sigma}, \mp)$.

\begin{definition}\label{CM_def} 
A \emph{CM type} for $E$ is a subset $\Phi$ of $m-1$ complex embeddings of $E$, 
no two of which are complex conjugate.
We say that $\Phi$ is compatible with $\cf$
if $\cf_j$ equals the number of embeddings in $\Phi$ whose restriction to $F$ is $\sigma_j$ for every $1 \leq j \leq m-1$. 
\end{definition}

\begin{example} \label{Ecmtypem5}
For $m=5$ and $\cf=(1,2,0,1)$, then 
$\Phi=\{(\sigma_1, +), (\sigma_2, +), (\sigma_2,-), (\sigma_4,+)\}$ is a CM type for $E$ compatible with $\cf$.
\end{example}

\begin{lemma}\label{LsimpleCMtype} 
Given $E=F(\sqrt{-\lambda})$ as above, there is a unique CM type $\Phi$ of $E$ compatible with a simple signature type $\cf$ 
up to the action of $\mathrm{Gal}(E/F)$.
The CM type $\Phi$ is primitive.
Hence, an abelian variety $A$ with complex multiplication 
by $E$ and CM type $\Phi$ compatible with $\cf$ is simple.
\end{lemma}

\begin{proof}
Without loss of generality, we assume $\cf(\sigma_1) =\cf(\sigma_{m-1}) =1$. 

First, we prove uniqueness. 
Consider a pair $1\le j,m-j \le m-1$ where $\cf_j=2, \cf_{m-j}=0$. For $\Phi$ to be compatible with $\cf$, we need $(\sigma_j,+),(\sigma_j,-) \in \Phi$.
By Definition~\ref{CM_def}, the pair $\{(\sigma_1,+),(\sigma_{m-1},-)\}$ are complex conjugates and thus only one is in $\Phi$. This implies there is a unique CM type  $\Phi_+$ (resp. $\Phi_-$) compatible with $\cf$ and satisfying  $(\sigma_1,+)\in \Phi_+$  (resp. $(\sigma_1,-)\in \Phi_-$).; the action of $\mathrm{Gal}(E/F)$ maps $\Phi_+$ to  $\Phi_-$.  

We prove $\Phi$ is primitive. Let  $\alpha \in \Gal(\overline{\bQ}/\bQ)$ be such that $\alpha \Phi=\Phi$. Then either $\alpha(\sigma_1)=\sigma_1$ or $\alpha(\sigma_1)=\sigma_{m-1}$. Since $\alpha\Phi=\Phi$, it follows that $\alpha(\sigma_j)\neq \sigma_{m-j}$ for all $2\leq j\leq m-2$. This implies $\alpha(\sigma_1)=\sigma_1$.
By definition, $(\sigma_1,+)\in \Phi$ if and only if $(\sigma_1,-)\not\in \Phi$, hence $\alpha(\sqrt{-\lambda})=\sqrt{-\lambda}$. 
Since $m>3$, we deduce $\alpha \in \Gal(\overline{\bQ}/E)$.

The simplicity of $A$ follows from \cite[Chapter~1]{LangCM}.
\end{proof}

\subsection{Uniqueness of CM abelian varieties}

\begin{proposition} \label{Puniquepp4} \cite[Proposition 4.5(1)]{LMPTshimuradata} Let $E$ be a CM field and $E_0$ its maximal totally real subfield. 
Suppose $E_0$ has units of almost independent signs and $Q(E)=2$. 
Let $(E, \Phi)$ be a primitive CM type. 
Then the number of isomorphism classes of principal polarizations on a CM abelian variety of type 
$({\mathcal O}_E,\Phi)$ is at most one.
\end{proposition}

\begin{theorem} \label{Tuniqueconganym}
Let $F=\QQ(\zeta_m)$ and suppose $|\cl_F|$ is odd.
Let $F_0=\QQ(\zeta_m+\zeta_m^{-1})$.
Under Assumption~\ref{def:good}, suppose also that $\lambda$ is inert in $F/F_0$.
Let $E=F(\sqrt{-\lambda})$.
Suppose $(E, \Phi)$ is a primitive CM type. 
If there exists a principally polarized abelian variety $(A, \eta)$ with CM by 
${\mathcal O}_E$ of type $\Phi$, which is fixed under complex conjugation, then it is unique up to isomorphism.
\end{theorem}

\begin{proof} 
By Proposition~\ref{Poddclassuniquepp},
the hypotheses of Proposition \ref{Puniquepp4} are satisfied. 

Suppose there exists an abelian variety $A$ with CM by ${\mathcal O}_E$ of type $\Phi$, which is fixed under complex conjugation, and which admits a principal polarization $\eta$.  Let $n=\mathrm{dim}(A)$.
The claim is that $(A, \eta)$ is unique up to isomorphism.
To do this, we show that the corresponding ideal class is trivial.

Let $\fa$ be an ideal class fixed by complex conjugation. 
For a simple CM type $\Phi$, by \cite[Theorem~3]{vanwamelen},
the complex torus $\CC^n/\Phi(\fa)$ admits a principal polarization of type $(E, \Phi)$ if and only if there exists $\xi\in E$ 
satisfying $E=E_0(\xi)$, $ \xi^2\in E_0$, and $D_{E/\QQ}\fa\bar{\fa}= \langle \xi^{-1} \rangle$, along with the 
positivity condition $\mathrm{Im}(\tilde{\sigma}(\xi)) >0$ for $\tilde{\sigma} \in \Phi$; furthermore, all principal polarizations 
on $\CC^n/\Phi(\fa)$ arise from such a $\xi$.

By \cite[Chapter 3, Propositions (2.2) and (2.4)]{Neukirch}, the different $D_{E/\QQ}$ is principal. 
%SAVE Prop 2.2 $D_{E/\QQ}=D_{E/F}D_{F/\QQ}$; 
% Prop 2.4 both $D_{E/F}$ and $D_{F/\QQ}$ are principal ideals. 
Since $\fa = \bar{\fa}$, the ideal class $\fa$ is in $\cl_E[2]$, which is trivial by Proposition \ref{Poddclassnumber}. 
This implies that there exists a unique isomorphism class of abelian variety $\CC^n/\Phi(\fa)$ 
with CM by $\mathcal{O}_E$ stable under complex conjugation that admits a principal polarization.
The uniqueness of the principal polarization is guaranteed from Proposition~\ref{Puniquepp4}.
\end{proof}

\begin{theorem} \label{Tuniquecong}(Special case of Theorem \ref{Tuniqueconganym})
Let $m=5$ and $F_0=\QQ(\sqrt{5})$.
Let $\lambda \in {\mathcal O}_{F_0}$ be a totally positive, irreducible element satisfying $\lambda\equiv -1\bmod 4{\mathcal O}_{F_0}$ and $N_{F_0/\QQ}(\lambda) \equiv 4 \bmod 5$.
Let $\Phi$ be the CM type in Example~\ref{Ecmtypem5}. 
If there exists a principally polarized abelian fourfold with CM by 
${\mathcal O}_E$ of type $\Phi$, which is fixed under complex conjugation, then it is unique up to isomorphism.
\end{theorem}

\begin{proof} 
By Lemma~\ref{quadrec5}, $\lambda$ satisfies Assumption~\ref{def:good} and $\lambda$ is inert in $F/F_0$.
By Lemma~\ref{LsimpleCMtype}, the CM type $\Phi$ is primitive.
The result follows from Theorem~\ref{Tuniqueconganym}.
\end{proof}

\begin{remark}\label{notmaxNT}  
With notation and hypotheses as in Theorem \ref{Tuniquecong}, 
the same argument shows that if there exists a principally polarized abelian variety with CM by the non-maximal order 
$\cO_F[\sqrt{-\lambda}]$ of type $\Phi$, which is fixed by complex multiplication, then there are at most two of these 
up to isomorphism.  We explain how to see this.
 
The hypotheses imply $\cO_E=\cO_F[(1+\sqrt{-\lambda})/2]$.
By Proposition~\ref{Poddclassuniquepp}, the class group $\cl_E$ of $\cO_E$ has odd size and $[\cU_{E_0}^+:N(\cU_E)] =1$.
Let $R=\cO_F[\sqrt{-\lambda}]$ and $\cU_R$ be the group of units in $R$.
Let $R_0=R\cap E_0$ and $\cU^+_{R_0}=\cU_R\cap E_0^+$.
To deduce the statement, it suffices to observe two things: first,
 the class semigroup of $R$ is 
$\cl_R=\cl_E\cdot  \langle1\rangle_R\cup \cl_E\cdot \langle2,\sqrt{-\lambda}\rangle_R$, where $\cl_E\cdot I$ denotes the orbit of $I\in\cl_R$ under multiplication by $\cl_E$, (which follows from \cite[Theorems 16 and 17; Example 20]{zanardo-zannier}); 
 and, second, that $[\cU^+_{R_0} \colon N(\cU_R)]=1$ (which follows by direct computations from the analogous statement for $\cO_E$). 
%Save by direct computations modulo $\langle 2\rangle$, and by 
%direct computations show $\cU_{E}\cap R=\cU_R$, $N(\cU_E)\cap R=N(\cU_R)$, and hence $[\cU^+_{R_0}:N(\cU_R)]=1$ 
\end{remark}

\subsection{Reduction of CM abelian fourfolds}

We continue with previous notation.
For CM abelian varieties $A$ of CM type $(E,\Phi)$, we identify the primes of basic reduction for $A$ 
using the Shimura--Tanayama formula.

\begin{proposition}\label{ST}
Let $m=5$.
Suppose $A$ is an abelian fourfold defined over a field $K$ containing $F_0$. 
Suppose $A$ has complex multiplication by an order in $E$ and has CM type $\Phi$.
If a prime ideal $\mathfrak{p} \subset {\mathcal O}_{F_0}$ does not split in $F_0(\sqrt{-\lambda})/F_0$, then 
the reduction of $A$ at primes of $K$ above $\mathfrak{p}$ is basic.
\end{proposition}

\begin{proof}
Let $w:K\to \overline{\QQ}_p$ be a place of $K$ above $v=v_\mathfrak{p}:F_0\to \overline{\QQ}_p$. We write $K_w$ for the completion of the image of $K$ under $w$.
By Example~\ref{EbasicM11},
the statement is equivalent to showing  that the slopes of the Newton polygon of the reduction of $A$ at $w$ are equal to $1/2$ if $p\not\equiv 1 \bmod 5$ and to $0$, $1/2$, and $1$ if $p\equiv 1\bmod 5$.

Following \cite[Section~5]{tate}, 
a $p$-divisible group $G$ over $\mathcal{O}_{K_w}$, of height $h$, has CM by a $p$-adic field $L/\QQ_p$ if there exists a $\QQ_p$-linear embedding of $L$ into $\mathrm{End}^0(G)$ and 
$[L:\QQ_p]=h$. By the Shimura--Tanayama formula \cite[Section~5, page 107]{tate}, if $G$ has CM then its reduction modulo $w$ is isoclinic, of slope $\dim(G)/h$. In particular, if $A$ is a CM abelian variety then $A[p^\infty]$ decomposes as sum of (isoclinic) CM $p$-divisible groups.

Assume $\mathfrak{p}\neq \lambda$; (a similar argument applies when $\mathfrak{p}=\lambda$). 
Suppose $p\not\equiv 1\bmod 5$. Then the prime $\mathfrak{p}$ is inert in $F/F_0$ 
and, by assumption, also inert in $F_0(\sqrt{-\lambda})/F_0$. 
Hence, $\mathfrak{p}$ is totally inert in $E/F_0$. 

If $p\equiv 2,3\bmod 5$, then $p$ is totally inert in $E/\QQ$ 
and the $p$-divisible group $A[p^\infty]$ has CM by $E_\mathfrak{v}$, 
for $\mathfrak{v}$ the unique prime of $E$ above $p$. Hence, it is isoclinic of slope $1/2$.

If $p\equiv 4\bmod 5$, then the $p$-divisible group $A[p^\infty]$ decomposes as a sum of two CM $p$-divisible groups $H,H'$, respectively with CM by $E_{\mathfrak{v}}, E_{\mathfrak{v}'}$, for $\mathfrak{v}, \mathfrak{v}'$ the unique primes of $E$ above $\mathfrak{p}, \mathfrak{p}^\tau$. Note that $\mathfrak{v}, \mathfrak{v}'$ are stable under complex conjugation; hence $H,H'$ are self-dual. We deduce each is isoclinic of slope $1/2$.

If $p\equiv 1\bmod 5$, then $p$ is totally split in $F/\QQ$.
By assumption, there is a unique place of $E$ above each place of $F$ above $p$.
Thus the $p$-divisible group $A[p^\infty]$ decomposes as a sum of four CM $p$-divisible groups, each of height $2$,  with dimensions $1, 2, 0, 1$, respectively.
Hence, their slopes are $1/2,1,0,1/2$, respectively.
\end{proof}

%%%%%%%%%%%%%SECTION 4

%%%%%%%%%%%%%%%%%%%%%%%%%%%%%%%%%%%%%%
% Section 4
%%%%%%%%%%%%%%%%%%%%%%%%%%%%%%%%%%%%%%

\section{Complex uniformization} \label{Scomplexunif}

Suppose $\Shg$ is a one-dimensional unitary Shimura variety parameterizing principally polarized abelian varieties having an action by a field $F$ of complex multiplication.
In Section~\ref{unif_sec}, we study the complex uniformization map
$\pi: \bH \to \Shg(\CC)$, which realizes the Shimura curve as a quotient of the upper half plane by a unitary group.
In Proposition~\ref{geodesic}, we prove that the 
pre-image under $\pi$ of the real points of $\Shg$ is a union of 
hyperbolic geodesics.
In Section~\ref{tzsqf}, we establish a connection between real CM points on $\Shg$ 
and solutions to certain quadratic forms.

Starting in Section~\ref{Srestricttocyc}, 
we restrict to the case of interest in this paper, 
where $F=\QQ(\zeta_m)$ and
$\Shg$ is a Shimura curve of genus $0$ defined over $\QQ$. 
In particular, we consider the families $\Sh$ from Section~\ref{M11asSh} when $m=5$ and $m=7$.
In Proposition~\ref{prop:QuadraticFormToCMpoints}, we describe the set of $\RR$-points of $\Sh$ that represent principally polarized abelian varieties having complex multiplication by certain quadratic extensions of $F$.

\subsection{The Shimura datum}\label{Shdatum}

Let $F$ be a CM field and let $F_0$ be the maximal totally real subfield of $F$. Let $n=[F_0:\QQ]$.  
We assume $n \geq 2$.  We label the embeddings of $F_0 \to \RR$
by $\tau_1, \ldots, \tau_n$.

We consider an integral PEL datum $(V,\langle\cdot,\cdot\rangle)$ 
associated with $F$ as in \cite[Definition~2.6]{LMPTshimuradata} (with respect to the families of curves in Section \ref{sec_curvefamily}).
In particular, $V$ is a $2$-dimensional vector space over $F$.
We write 
\begin{equation} \label{Vdecomp}
V\otimes_\QQ\RR = \oplus_{i=1}^n \left(V\otimes_{F_0,\tau_i}\RR\right).
\end{equation}

The integral PEL datum is determined by the $F$-vector space $V$, with the standard  
$\mathcal{O}_{F}$-lattice $\Lambda=\mathcal{O}_F^2\subset V$,
together with a symplectic form $\langle \cdot,\cdot\rangle$ on $V$, 
taking integral values on $\Lambda$. 
The symplectic form $\langle \cdot, \cdot \rangle: V\times V \to \QQ$ is given by
\begin{equation} \label{Exi}\langle x, y \rangle = \mathrm{tr}_{F/\QQ}(x\, B \, ^t\bar{y}), \text{ for } B= \mathrm{diag}[\xi_1,\xi_2]\in \mathrm{GL}_2(F),\end{equation}
where $\xi_1,\xi_2 \in F^*$ are each totally imaginary, and contained in the codifferent $D_{F/\QQ}^{-1}$. 

\begin{definition} \label{ShDef}
Let $\Shg$ be the PEL moduli space determined by the integral PEL datum as in \cite[Section 5]{kottwitz92}, \cite[Theorem~1.4.1.11]{Lan}, 
(see also \cite[Section 2]{LMPTshimuradata}).
\end{definition}

The points of $\Shg$ represent principally polarized abelian varieties of dimension $g=2n$ equipped with an action of $\mathcal{O}_F$ satisfying a certain signature condition (up to equivalence given by $\Gal(F/\bQ)$).
See related material in Lemma~\ref{LM11defQ}.
The integral PEL Shimura datum $\mathcal{D}$ for $\Sh=M[11]$ (resp.\ $M[17]$)
is described in \cite[Corollaries~6.4, 6.9]{LMPTshimuradata}. 

\subsection{Signature for Hermitian form}

We place conditions on the signature to guarantee that $\Shg$ is one-dimensional and compact.

\begin{notation} \label{Dbeta0Hermitian}
We fix a totally imaginary generator $\beta_0 \in F$ of $D_{F/\QQ}$ with $\mathrm{Im}(\sigma_1(\beta_0)) > 0$.

For $x, y \in V$, define a Hermitian form $(\cdot,\cdot) : V\times V \to F$ by 
\begin{equation} \label{Ebbi}
(x,y) = x\  A \  ^t\bar{y}, \ \text{ where } A=\beta_0 B = \mathrm{diag}[\bb_1,\bb_2].
\end{equation}
where $\bb_i \in \mathcal{O}_{F_0}$ are given by $\bb_i= \xi_i\beta_0$ { for } $i=1,2$.
\end{notation}

Note that
$\langle x, y \rangle = \mathrm{tr}_{F/\QQ} (\beta_0^{-1} (x, y))$. We use $\GU_2=\GU(V, (\cdot, \cdot))$ to denote the unitary similitude group.
%SAVE Here the similitude is in $F_0$; we do not require it in $\bQ$ as in mu-ordinary reduction paper.

\begin{example}
Let $m \geq 5$ be an odd prime.
Let $F=\QQ(\zeta_m)$.
Then $F_0=\QQ(\zeta_m + \zeta_m^{-1})$ and $n=(m-1)/2$.
From \eqref{EembedF}, consider the embedding $\sigma_1 : F \to \CC$ such that $\sigma_1(\zeta_m)=e^{2\pi i/m}$, 
By \cite[Lemma~3.7]{LMPTshimuradata}, in Notation~\ref{Dbeta0Hermitian}, we can choose
\begin{equation} \label{Ebeta0}
\beta_0 := m/(\zeta_m^{(m+1)/2} - \zeta_m^{(m-1)/2}).
\end{equation}
\end{example}

\begin{notation} \label{Nsignature}
We assume that $n \geq 2$ and that the signature of the unitary group $U(V, (\cdot ,\cdot ))$  is 
$(1,1)$ at $\tau_1$, and either $(2,0)$ or $(0,2)$ at $\tau_j$, for $2\leq j\leq n$. 
Equivalently, this assumption means
the elements $\bb_1,\bb_2\in \mathcal{O}_{F_0}$ satisfy $\tau_1(\bb_1\bb_2)<0$, and $\tau_j(\bb_1\bb_2)>0$ for all $2\leq j\leq n$. Without loss of generality, we assume $\tau_1(\bb_1)> 0$.
\end{notation}

\begin{lemma}
Under Notation~\ref{Nsignature},
the PEL moduli space $\Shg$ is $1$-dimensional and compact.
\end{lemma}

\subsection{Bounded complex uniformization}

Let $\Shg$ denote the unitary Shimura curve with no level structure defined by the Shimura datum in Section~\ref{Shdatum} and Definition~\ref{ShDef}.
We start by recalling the (bounded) complex parametrization of $\Shg$. 

Consider $V \otimes_{F_0, \tau_1}\RR$ as a vector space of dimension $2$ over $F \otimes_{F_0, \tau_1}\RR \simeq \CC$.
Consider its complex projectivization $\bP(V \otimes_{F_0, \tau_1}\RR)$.
For any $w\in\bP(V \otimes_{F_0, \tau_1}\RR)$, the sign of the Hermitian form $(\cdot,\cdot)$ on $w$ is well-determined, 
(meaning that the sign of $(v,v)$ is independent
of the choice of a non-zero vector $v\in w$) 
and we denote it by $(w,w)$.

Let 
\begin{equation} \label{EdefD}
\bD^- = \{ w \in \bP(V \otimes_{F_0, \tau_1}\RR) \mid (w,w)  <0 \}
\end{equation} 
and 
\begin{equation} \label{EdefDperp}
	\bD^+= \{ w  \in \bP(V \otimes_{F_0, \tau_1}\RR) \mid (w,w)  >0 \}.
\end{equation} 

Write $\bD=\bD^-\cup\bD^+$. For any $w\in\bD$, we write 
\begin{equation} \label{Ewperp}
w^\perp=\{ x\in V\otimes_{F_0,\tau_1}\RR \mid \forall v\in w:\, (v,x)=0 \}.
\end{equation}
If $w\in\bD^\pm$, then $w^\perp \in \bD^\mp$, and $V\otimes_{F_0,\tau_1}\RR=w\oplus w^\perp.$

To each $w\in\bD$, we associate a complex structure $\cdot$ on $V\otimes_\QQ \RR$; we denote 
the associated complex vector space by $V_w \cong \CC^{g}$.
For all $a\in\CC$, if $w\in \bD^-$, then we define:
\begin{equation} 
a\cdot v = 
\begin{cases}
\bar{a}v  \text{ if }v\in w;\\ 
 {a}v \text{ if }v\in w^\perp;\\
av\text{ if }v\in V\otimes_{F_0,\tau_i}\RR \text{ for some } 2\leq i\leq n \text{ and the signature at } \tau_i \text{ is }(2,0);\\
\bar{a}v\text{ if }v\in V\otimes_{F_0,\tau_i}\RR \text{ for some } 2\leq i\leq n \text{ and the signature at } \tau_i \text{ is }(0,2).
\end{cases}
\end{equation}
(From \eqref{Vdecomp}, the conditions are well-defined, 
disjoint and span $V \otimes_\QQ \RR$.)
If $w\in \bD^+$, then $w^\perp\in \bD^-$, and we define the complex  structure on $V_w$ as the conjugate of the complex structure on $V_{w^\perp}$.

\begin{definition}\label{def-pi}
We define the (bounded) complex parametrization $\pi:\bD \to \Shg(\CC)$ as follows. 
For any $w\in\bD$, let
\[\pi(w)=(A_w,\lambda_w, \iota_w),\]
where
\begin{enumerate}
\item $A_w$ is the complex torus  $A_w=V_w/\Lambda$,
\item $\lambda_w$ is the Riemann form on $A_w$, namely the symplectic form $\langle\cdot,\cdot\rangle$ on $V_w$.

Note that $\lambda_w$ takes integral values on the lattice $\Lambda$, and  it is positive definite
with respect to the complex structure $\cdot$ on $V_w$,  that is $\langle i\cdot x,x\rangle_\RR >0$ for all $0\neq x\in V_w$.  

\item  $\iota_w:\mathcal{O}_F\to \mathrm{End}(A_w)$ is defined via the complex structure $\cdot$ on $V_w$. 
\end{enumerate}
\end{definition}

\subsection{Uniformization of unitary Shimura curves} \label{unif_sec}

Let $\bH^+$ (resp. $\bH^-$) denote the complex upper (resp. lower) upper half plane; write $\bH=\bH^+\cup\bH^-$.
We give an complex parametrization of $\Shg(\CC)$, by identifying $\bD$ with $\bH$ as follows. 

A basis $\{e,f\}$ of a Hermitian space $(W, (\cdot,\cdot))$ is called {\em isotropic} if it satisfies
\begin{equation} \label{isotropic_def}
(e,e)=0,\, (f,f)=0 \text{ and } (e,f)=-(f,e)\neq 0.
\end{equation}
An isotropic basis always exists (see Lemma \ref{basis_lemma}).

Suppose $\{e,f\}$ is an isotropic basis of $V \otimes_{F_0, \tau_1} \RR$ with respect to the Hermitian form 
$(\cdot,\cdot)$.
For any $\theta \in \CC$, consider the vector
$v_\theta=\theta e+f\in V \otimes_{F_0, \tau_1}\RR$
and define $w_\theta$ to be the line in $\bP(V \otimes_{F_0, \tau_1}\RR)$ spanned by $v_\theta$. 
Define $w_{\infty}= \CC e$.

Define:
\begin{equation} \label{DmapI}
I: \CC\to \bP(V \otimes_{F_0, \tau_1}\RR), \ \mbox{by} \ I(\theta)=w_\theta.
\end{equation} 

\begin{lemma}\label{H_lemma}
Given an isotropic basis $\{e,f\}$ of $V \otimes_{F_0, \tau_1} \RR$ with respect to the Hermitian form,
the map $I$ induces a bijective complex analytic map $\bH\to\bD$.
\end{lemma}

\begin{proof}
Any element in $\bP(V \otimes_{F_0, \tau_1}\RR)$,
except for $w_{\infty}= \CC e$, can be uniquely represented as $w_\theta=\CC( \theta e+f)$, for some $\theta \in \CC$. 
By \eqref{isotropic_def}, $w_\infty\notin\bD$.  Hence, 
to prove $I$ is a bijection, 
it suffices to check that  $I(\bH)=\bD$, that is $(w_\theta,w_\theta)\neq 0 $ if and only if $\mathrm{Im}(\theta)\neq 0$.
This follows from this computation: 
\begin{equation*}
(\theta e + f, \theta e + f) =  \theta (e,f)+ \bar{\theta} (f,e) = ({\theta} - \bar{\theta}) (e,f) = 2\mathrm{Im}(\theta) (e,f).
\qedhere
 \end{equation*}
\end{proof}

From now on, we denote by 
\begin{equation} \label{Enewpi}
\pi: \bH \to \Shg(\CC),
\end{equation}
the composition of the bijective complex analytic map
$I:\bH\to\bD$ from Lemma \ref{H_lemma}
with $\pi:\bD\to \Shg(\CC)$ from Definition~\ref{def-pi}.

\begin{lemma}\label{action-on-H}
Consider the action of $X \in \GL_2(\CC)$ on $\bP(V \otimes_{F_0, \tau_1}\RR)$ induced from the action on $V \otimes_{F_0, \tau_1}\RR$ with respect to the isotropic basis $\{e,f\}$. That is, for $X = \begin{pmatrix}
	x_{11} & x_{12} \\ x_{21} & x_{22}
\end{pmatrix} \in \GL_2(\mathbb{C})$, let $X(e) =  x_{11} e+x_{21} f$ and $X(f)=x_{12} e+x_{22} f$.  If $\theta \in \bH$, then: 
\begin{enumerate}
\item The induced action of $X$ 
on $\theta \in \bH$ is given by $X(\theta) = \frac{x_{11} \theta +x_{12}}{x_{21} \theta+x_{22}}$.
\item In particular, $X(\theta) = \theta$ if and only if 
$\theta \in \bH$ satisfies 
$x_{21} \theta^2 + (x_{22} -x_{11})\theta - x_{12} = 0$.
\end{enumerate}
\end{lemma}

\begin{proof}
    Part (1) follows from Lemma~\ref{H_lemma}.
and part (2) follows by setting $X(\theta)=\theta$.
\end{proof}

\subsection{The real points on the Shimura curve} \label{Scomplexuniformreal} 

We characterize a set of points in $\mathbb{H}$ whose image under $\pi$ are all of the real points of $\Shg$.

\begin{lemma}\label{lem-complex-conjugation}
	If $\theta \in \mathbb{H}$, then $\overline{\pi(\theta)}= \pi(\bar{\theta})$.
	In other words, complex conjugation on $\Shg(\CC)$ agrees with complex conjugation on $\bH$.	
\end{lemma}

\begin{proof}
	By \eqref{Ewperp} and Definition \ref{def-pi}, for any $w\in\bD$, 
	the complex conjugate of $\pi(w)$ is $\pi(w^\perp)$. 
	Since $\CC(\theta e+f)^\perp = \CC(\bar{\theta} e+f)$, 
	for any $\theta\in\bH$, the complex conjugate of $\pi(\theta)$ is $\pi(\overline{\theta})$.
\end{proof}

Using the notation of Lemma~\ref{Lmatrix}, 
the matrix $M$ in the standard basis is in $\GU_2(\RR)$ 
if and only if the matrix $X$ in the isotropic basis is in $\CC^*\GL_2(\RR)$. 
In other words, 
with respect to the isotropic basis $\{e,f\}$ of $V\otimes_{F_0,\tau_1}\RR$, we identify
\[\GU_2(\RR)=\CC^*\mathrm{GL}_2(\RR)\subset \mathrm{GL}_2(\CC).\]
 
We write $\GU_2(F_0) = \GL_2(F) \cap \GU_2(\RR)$; 
a matrix is in $\GU_2(F_0)$ if it is defined over $F$ in the standard basis 
and is in $\CC^*\GL_2(\RR)$ in the isotropic basis.
We write $\GU_2(\mathcal{O}_{F_0}) = \GU_2(F_0) \cap M_2(\mathcal{O}_{F})$ (with respect to the standard basis).

We denote by $Z$ the center of $\GL_2(\CC)$. 
For any subgroup $G \subset \GL_2(\CC)$, we denote by $G/Z$ the quotient $G/(G\cap Z)$. 
Consider the Fuchsian group 
\begin{equation} \label{Edelta}
\Delta:=(\GU_2(\mathcal{O}_{F_0})/Z) \cap (\SL_2(\RR)/Z).
\end{equation}

\begin{proposition}
The map $\pi: \mathbb{H} \to \Shg(\CC)$ from \eqref{Enewpi}
is the quotient of $\mathbb{H}$ by the action of $\Delta$.
\end{proposition}

\begin{proof}
By Definition \ref{def-pi}, elements in $\ker \pi \subset \SL_2(\RR)/Z$ are $F$-linear maps on $V$ which preserve the polarization/Hermitian form up to scalar and the lattice $\Lambda$. This is exactly $\Delta$ by definition.
\end{proof}

Recall that a geodesic in $\mathbb{H}$ is either a semi-circle whose center is on the real line or a ray orthogonal to the real line.  

\begin{proposition}\label{geodesic} 
	Let $\pi:\bH \rightarrow \Shg(\CC)$ be the complex uniformization map from \eqref{Enewpi}. 
	Then $\pi^{-1}(\Shg(\RR))$ is a union of geodesics.
\end{proposition}

\begin{proof}
Consider the action of 
the discrete subgroup 
$\Delta \subset \mathrm{GL}_2({\mathbb R})/Z$ on $\mathbb{H}$.
%SAVE In this context, $\Gamma$ is $(\GU_2(\mathcal{O}_{F_0})/Z) \cap (\GL_2(\RR)/Z)$. 
By Lemma~\ref{lem-complex-conjugation}, complex conjugation on $\Shg$ is given by complex conjugation on $\mathbb{H}$.
The condition $\theta \in \pi^{-1}(\Shg(\RR))$ means that $\bar{\theta}$ is in the orbit of $\theta$ under the action of 
$\Delta$.
We claim that the set of points $\theta \in \mathbb{H}$ satisfying this condition is a union of geodesics.

To see this, consider $\theta=z_1+iz_2 \in \mathbb{H}$ such that $\bar{\theta} = \frac{x_{11}\theta + x_{12}}{x_{21} \theta+x_{22}}$, 
for some $\begin{pmatrix}
	x_{11} &x_{12}\\x_{21}&x_{22}
\end{pmatrix} \in \Delta$.
Then \[0 = x_{21} (z_1^2 +z_2^2) + (x_{22}-x_{11}) z_1  -iz_2 (x_{22}+x_{11}) -x_{12}.\]
Since $z_1,z_2$ are real and $x_{11},x_{12},x_{21},x_{22}$ are real, 
it follows that $x_{11}=-x_{22}$.
If $x_{21} =0$, then $z_1 = x_{12}/2x_{22}$, giving a ray orthogonal to the real line.
If $x_{21} \not = 0$, then 
%SAVE $0=z_1^2 + (2x_{22}/x_{21}) z_1 + z_2^2 -x_{12}/x_{21}$.
$(z_1+(x_{22}/x_{21}))^2+z_2^2=(x_{12}x_{21}+x_{22}^2)/x_{21}^2$, which is a circle centered on the real line.
\end{proof}

\subsection{The unitary similitude group} \label{S4unitarysimilitude}

Information about the Shimura variety $\Shg$
is naturally expressed in terms of the unitary similitude group
$\GU_2=\GU(V, (\cdot, \cdot))$.  
By definition, $\GU_2$ is the subgroup
of elements of $\GL(V)$ which preserve the Hermitian form $(\cdot,\cdot)$ in \eqref{Ebbi} up to a scalar.
We explicitly compute $\GU_2$ as a subgroup of $\Res^F_{F_0}\GL(V)=\Res^F_{F_0}\GL_{2,F}$ with respect to an isotropic basis. 
Recall the definitions of $\beta_0$, $\bb_1$, and $\bb_2$ from Notation~\ref{Dbeta0Hermitian}.

\begin{lemma}\label{basis_lemma}
With respect to the standard basis for $V$ and the embedding $\tau_1: F_0 \hookrightarrow \RR$, 
an isotropic basis for $V \otimes_{F_0, \tau_1} \RR$ is given by
\begin{equation} \label{Edefef}
e=(\sqrt{-\bb_2},\sqrt{\bb_1}) \text{ and } f=(-\beta_0\sqrt{-\bb_2}, \beta_0\sqrt{\bb_1}),
\end{equation}
where we recall that $v_i\in F_0$ and we view them in $\RR$ via $\tau_1$.

Furthermore, $(e, f) = -2 (\bb_2 \bb_1) \beta_0$.
\end{lemma}
\begin{proof}
The vectors $e$ and $f$ are defined over $\RR$ and are linearly independent.
Using \eqref{Ebbi}, we directly compute $(e,e)=(f,f)=0$, and 
\[(e,f) = \bb_1 \sqrt{-\bb_2}  (-\bar{\beta}_0 \sqrt{- \bb_2}) + \bb_2 \sqrt{\bb_1} (\bar{\beta}_0 \sqrt{\bb_1})=-2 (\bb_1 \bb_2)\beta_0 
= -(f, e).
\qedhere
\]  
\end{proof}

Recall that $\tau_1({-\bb_1\bb_2})>0$.  Define 
\begin{equation}\label{omega_eq}
\omega=\sqrt{-\bb_2/\bb_1}\in \RR^+.
\end{equation} 

\begin{notation} \label{Nrsjk}
For any $a,b,c,d\in \CC$, set 
\begin{equation} \label{Evarchange}
r=a+d, \, s=d-a, \, \jj=\omega^2 c +b, \text{ and } \kk=\omega^2 c-b.
\end{equation}
\end{notation}

By definition, $\{a,b,c,d\}\subset F$ if and only if $\{r,s,\jj,\kk\} \subset F$. 

\begin{lemma} \label{Lmatrix} 
%For simplicity, assume $\bb_1=1$.
A matrix $M=\begin{bmatrix} a & b \\  c & d \end{bmatrix}$ in $\GL_2(F)$ 
transforms, with respect to the isotropic basis $\{e,f\}$, to 
\begin{equation} \label{EmatrixX}
X=\frac{1}{2\omega}\begin{bmatrix} \omega r + \jj & \beta_0(\omega s - \kk) \\  \beta_0^{-1}(\omega s +\kk) & \omega r-\jj \end{bmatrix}.
\end{equation}

Then:
\begin{equation} \label{EdetX}
\mathrm{tr}(X) =r, \text{ and } \mathrm{det}(X) = (\omega^2r^2-\jj^2-\omega^2s^2+\kk^2)/4\omega^2.
\end{equation}
\end{lemma}

%SAVE
%\begin{proof}
%Let $C$ be the change of basis matrix whose columns are ${}^te$ and ${}^tf$.  Then $X=C^{-1} M C$.  The first and second
%claims follow from a computation.
%\end{proof}

%SAVE:
%\begin{lemma} \label{Lbacksubstitute}
%The matrix $X=\begin{bmatrix} x_{1,1} & x_{1,2} \\  x_{2,1} & x_{2,2} \end{bmatrix}$ 
%written in terms of the isotropic basis $\{e,f\}$
%transforms, with respect to the standard basis, to 
%$M=\begin{bmatrix} a & b \\  c & d \end{bmatrix}$, 
%where:
%\begin{eqnarray*}
%a & = & (r-s)/2 =  \left((x_{1,1}+x_{2,2}) - (\beta_0 x_{2,1} + \beta_0^{-1}x_{1,2})\right)/2;\\
%b & = & (\jj-\kk)/2 =  \left((x_{1,1} - x_{2,2}) - (\beta_0 x_{2,1} - \beta_0^{-1}x_{1,2})\right) \omega /2;\\
%c & = & (\jj+\kk)/2\omega^2  =   \left((x_{1,1} - x_{2,2}) + (\beta_0 x_{2,1} - \beta_0^{-1}x_{1,2}) \right)/2 \omega; \text{ and }\\
%d & = & (r+s)/2 = \left((x_{1,1}+x_{2,2}) + (\beta_0 x_{2,1} + \beta_0^{-1}x_{1,2})\right)/2.
%\end{eqnarray*}
%\end{lemma}

%\begin{proof}
%The result follows from the calculation of $r=\mathrm{tr}(X)$, $s=\beta_0 x_{2,1} + \beta_0^{-1} x_{1,2}$, 
%$\jj=\omega(x_{1,1} - x_{2,2})$, and $\kk = \omega(\beta_0 x_{2,1} - \beta_0^{-1}x_{1,2})$.
%\end{proof}

%Via the base change formula, we compute the subgroup $\GU_2(F_0)\subset \GL_2(F)$ with respect to $\{e,f\}$. 
%Furthermore, $M\in \GU_2(\RR)$ if and only if $X\in \CC^*\GL_2(\RR)$. 

\subsection{Trace zero stabilizers}\label{tzsqf}

We determine the points of $\mathbb{H}$ which have a non-trivial stabilizer in $\GU_2(\mathcal{O}_{F_0})$ under the action of $\Delta$.

For any $z\in\bH$, denote by $\mathrm{Stab}(z) \subset \GL_{2}(\RR)^+$ the subgroup of 
elements that stabilize $z$ under the action of linear fractional transformations.

We can find the following facts in \cite[Chapter 1]{Miyake}.

\begin{lemma} \label{Lstabz} If $z \in \bH$, then:
\begin{enumerate}
\item Then $\mathrm{Stab}(z) \simeq \RR^* \SO_2(\RR)$;
\item There exists $\gamma_z\in \mathrm{Stab}(z)$, unique up to scalar multiplication, 
satisfying $\mathrm{tr}(\gamma_z)=0$.  
\item The map $z\mapsto\gamma_z $ defines a bijection between $\bH$ and $\{\gamma\in \GL_2(\RR)^+ \mid \mathrm{tr}(\gamma)=0\}$ modulo scalar multiplication by $\RR^*$. 
\end{enumerate}
\end{lemma}
For a subfield $K\subset \RR$, we use $\GU_2(K)^+$ to denote $\GU_2(K) \cap \GL_2(\RR)^+$. When we consider $\gamma$ with respect to the standard basis, we write $\gamma \in \GU_2(\RR)^+$ instead of $\gamma \in \GL_2(\RR)^+$.

We deduce the following result. 

\begin{lemma}\label{detgamma}
Let $z_1, z_2 \in \bH$.  
For $i=1,2$, in standard coordinates, 
let $\gamma_i \in \GU_2(\RR)^+$ be the element 
with trace zero in $\mathrm{Stab}(z_i)$,
(which is well-defined in $\GU_2(\RR)^+=\GL_2(\RR)^+$ up to multiplication by a scalar in $\RR^*$).
Then the hyperbolic geodesic containing $z_1$ and $z_2$ is given by the fixed points of 
$M_{x,y}=x \gamma_1 + y \gamma_2$, 
for $x,y\in \RR$ such that $\mathrm{det}(M_{x,y}) >0$.

Furthermore, for a number field $K \subset \RR$ with $F_0 \subset K$, 
if $\gamma_1, \gamma_2 \in \GU_2(K)^+$ and $x,y \in K$, 
then $x \gamma_1 + y \gamma_2 \in \GU_2(K)$.
\end{lemma}

\begin{proof}
Without loss of generality, we suppose that 
$\mathrm{Re}(z_2) \geq \mathrm{Re}(z_1)$.

Let $\rho \in \GL_2(\RR)^{+}$ be a matrix that sends the geodesic segment $z_1z_2$ to a segment $T$ contained in the 
vertical ray $\RR^+i$.  Specifically, if $z_1$ and $z_2$ have the same real coordinate $x$, set 
$\rho = \begin{pmatrix}
	1 & -x \\ 0 & 1
\end{pmatrix}$.
If not, let $x_1,x_2\in \RR $ be the two end points of the semi circle containing the
geodesic segment $z_1z_2$, labeled so that $x_2 > x_1$,  
and set $\rho= \begin{pmatrix}
	1 & -x_2 \\ 1 & -x_1
\end{pmatrix}$.

In either case, 
$\rho(z_1) = t_1 i$ and $\rho(z_2) = t_2 i$ for some
$t_1, t_2 \in \RR^{+}$ with $t_1 > t_2$.
Here $t_1i$ and $t_2 i$ are the endpoints of the segment $T\subset \RR^+i$.
After scaling $\rho$, we can suppose that $t_1 = 1$ and 
$0 < t_2 < 1$.

For $\ell=1,2$, then $\gamma'_\ell=\rho\gamma_\ell\rho^{-1}$ stabilizes $t_\ell i$. 
Suppose $z \in \bH$.
Then $z$ is on the ray $\RR^+ i$ if and only if $z \in \bH$ and $z$ is 
stabilized by
$x\gamma_1' + y \gamma_2'$ for some $x,y \in \RR$.
Using the transformation $w = \rho^{-1}(z)$, 
it follows that $w$ is on the geodesic containing $z_1$ and $z_2$ 
if and only if it
is stabilized by $\rho^{-1} (x\gamma_1' + y \gamma_2') \rho = x \gamma_1 + y \gamma_2$ for the same $x,y \in \RR$.

Write $M_{x,y} = \begin{pmatrix}
a & b \\c &d
\end{pmatrix}$.  Then $a=-d$ since 
$\gamma_1$ and $\gamma_2$ have trace $0$.
Let $X=\begin{pmatrix}
x_{1,1} & x_{1,2} \\ 
x_{2,1} & x_{2,2}
\end{pmatrix}$ be the matrix for $M_{x,y}$ in isotropic coordinates.
Then $M_{x,y}$ fixes $z \in \bH$ if and only if 
$X \circ z = z$, which is equivalent to 
$0 = x_{2,1} z^2 + (x_{2,2}-x_{1,1})z -x_{1,2}$.
The condition $z \in \bH$ is equivalent to
$(x_{2,2} - x_{1,1})^2 + 4x_{1,2}x_{2,1} <0$.
Using Lemma~\ref{Lmatrix}, 
this is equivalent to 
%SAVE \[((\omega r - j) - (\omega r + j))^2 + 4(\omega s + k)(\omega s -k)\]
$(-2j)^2 + 4(\omega^2 s^2 - k^2) <0$. 
By Notation~\ref{Nrsjk}, 
%SAVE $j = \omega^2 c + b$, $k=\omega^2 c -b$, $s=d-a$
this condition simplifies to $\mathrm{det}(M_{x,y}) >0$.
%SAVE $4bc + (d-a)^2 <0$.

The last statement over $K$ is clear.
\end{proof}

\subsection{Genus zero Shimura curves with three marked points} \label{Srestricttocyc}

Let $m=5$ (resp.  $m=7$) and $\Shg=\Sh$ be as defined in Section~\ref{M11asSh}.
Then $\Sh$ has genus $0$ and is defined over $\QQ$ by Lemma~\ref{LM11defQ}.

Let $P,Q,R$ be the special points of $\Sh$ corresponding to Jacobians of curves with extra automorphisms, 
as defined in Section \ref{subsec-special-points}).
By Lemma~\ref{Lfomfod}, $P$, $Q$, $R$ are in $\Sh(\QQ)$. 

By cutting $\Sh(\CC)$ via the real line $\Sh(\RR)$, 
we obtain two simply connected domains in $\bP^1(\CC)$. 
Pulling back by $\pi$, we triangulate $\bH$ into simply connected regions whose boundaries lie on geodesics, and 
whose vertices lie above the points representing curves with extra automorphisms, namely $P$, $Q$, and $R$.

\begin{proposition}\label{edges}
	The preimage of $\Sh(\RR)$ in $\bH$ is the union of all edges of hyperbolic
	triangles whose vertices are in $\pi^{-1}(P)$, $\pi^{-1}(Q)$, and $\pi^{-1}(R)$.
\end{proposition}

\begin{proof}
By Proposition \ref{geodesic}, $\pi^{-1}(\Sh(\RR))$ is a union of geodesics.
A point $\theta$ lies on two of these geodesics if and only if there exist distinct $\sigma_1,\sigma_2 \in \Delta$
such that $\sigma_1\theta=\sigma_2\theta=\bar{\theta}$.
This implies that $\sigma^{-1}_1\sigma_2 \in \Delta \cap \mathrm{Stab}(\theta)$.
Since $\sigma^{-1}_1 \sigma_2 \not = \mathrm{id}$, it follows that
$\pi(\theta)$ represents a curve with extra automorphisms; 
the only such curves in this family 
are represented by the points $P,Q,R$ by Proposition~\ref{Pnomorebigaut}.
\end{proof}

\begin{notation} \label{Ngamma}
By Proposition \ref{edges}, we can choose a fundamental triangle $\mathfrak T$ in $\bH$ for the action of $\Delta$
whose boundaries are geodesics. 
Let $\tilde{P}$ (resp.\ $\tilde{Q}$, $\tilde{R}$) 
be the vertex of $\mathfrak T$
whose image under $\pi$ is the point $P$ (resp.\ $Q$, $R$).
For $z=\tilde{P},\tilde{Q}, \tilde{R}$, 
choose $\gamma_z\in \mathrm{Stab}(z)$ as in Lemma \ref{Lstabz}; 
we may choose $\gamma_z\in M_2(\cO_F)$. For simplicity, we write $\gamma_P=\gamma_{\tilde{P}}$, $\gamma_Q=\gamma_{\tilde{Q}}$, and $\gamma_R=\gamma_{\tilde{R}}$.
\end{notation}

\subsection{Complex multiplication and quadratic forms}
\label{SCMquadforms}

In Proposition~\ref{edges}, we characterized the points in $\mathbb{H}$ whose images under $\pi$ are the real points $\Sh(\RR)$. 
Next, we describe a subset of points whose images under $\pi$
are CM points in $\Sh(\RR)$. 
We revisit this material when $m=5$ in Section~\ref{SCMquadformsm5}.

Recall that Assumption~\ref{def:good} states that
$\lambda\in {\mathcal O}_{F_0}$ is totally positive, 
is relatively prime to $m$, 
generates a prime ideal, and has odd norm;
also $-\lambda$ is a square modulo $4{\mathcal O}_{F_0}$.

From \eqref{EdefEfield}, 
recall that $E=F(\sqrt{-\lambda})$ is a CM field. 
Let $\mathcal{O}_E$ denote the ring of integers of $E$. 
Then $\mathcal{O}_E\supseteq \mathcal{O}_F[\sqrt{-\lambda}]$.  
If $\lambda\equiv -1\bmod 4 \mathcal{O}_{F_0}$, then $\mathcal{O}_E= \mathcal{O}_F[(1+\sqrt{-\lambda})/2]$.

Recall the definition of $\gamma_P,\gamma_Q,\gamma_R$ from 
Notation~\ref{Ngamma}.
Given a pair $\gamma_1,\gamma_2$ of these,  
consider the quadratic form 
\begin{equation} \label{Equadform12}
q_{1,2}(x,y)=\det (x\gamma_1 + y \gamma_2).
\end{equation}

\begin{definition}
Under Assumption~\ref{def:good},
  we say that $q_{1,2}$ represents $\lambda$
  if $q_{1,2}(x,y) = \lambda$ for some $x,y \in F_0$ such 
  that $x\gamma_1 + y \gamma_2 \in \GU_2(\mathcal{O}_{F_0})$.
\end{definition}

We say that a point $\eta$ of $\Sh(\CC)$ 
has complex multiplication by an order $R$ in a CM field
if it represents 
a principally polarized abelian variety with 
complex multiplication by $R$.

\begin{proposition}\label{prop:QuadraticFormToCMpoints}
Under Assumption~\ref{def:good},
there exists a point in $\Sh (\RR)$ with complex multiplication by ${\mathcal O}_F[\sqrt{-\lambda}]$
if the quadratic form $q_{1,2}(x,y)=\det (x\gamma_1 + y \gamma_2)$
represents $\lambda$, for at least one pair $\gamma_1,\gamma_2$ of $\gamma_P,\gamma_Q,\gamma_R$.
Furthermore, suppose $\lambda\equiv -1\bmod 4 \mathcal{O}_{F_0}$. 
Then this point has complex multiplication by $\mathcal{O}_E$ 
if 
$\frac{1}{2}(\mathbb{I}+x\gamma_1 + y \gamma_2)\in M_2(\mathcal{O}_F)$.
\end{proposition}

\begin{proof}
By assumption, there exists $M \in \GU_2(\mathcal{O}_{F_0})$ being a linear combination $x \gamma_1 + y \gamma_2$ such that $\det(M)=\lambda$. Via $\tau_1:F_0\rightarrow \RR$, we have $\det(M)\in \RR^+$ and note that $\Tr(M)=0$ and thus $M$ has a fixed $\tilde{\eta}\in \bH$. By Lemma \ref{detgamma}, $M=x \gamma_1 + y \gamma_2$ implies that $\tilde{\eta}$ lies on on a geodesic between the lifts of $P$, $Q$, and $R$. 
By Proposition~\ref{edges}, $\eta = \pi(\tilde{\eta}) \in \Sh(\RR)$. The matrix $M$ acting on $V$ induces an endomorphism $s$ on $A_{\tilde{\eta}}$, where $A_{\tilde{\eta}}$ denote the principally polarized abelian variety represented by ${\tilde{\eta}}$. $M\in \GL_2(F)$ implies that $s$ commutes with $\cO_F$-action; $\det(M)=\lambda$ and $\Tr(M)=0$ imply that $M^2=-\lambda \mathbb{I}$ and hence $s\circ s = -\lambda$. We conclude that $A_{\tilde{\eta}}$ has CM by ${\mathcal O}_F[\sqrt{-\lambda}]$.

Suppose $\lambda\equiv-1\bmod 4 {\mathcal O}_{F_0}$, and $q_{1,2}(x,y)=\lambda$ for some $x,y \in F_0$ 
having the property that $\frac{1}{2}(\mathbb{I}+x\gamma_1 + y \gamma_2)\in M_2(\mathcal{O}_F)$.  
Consider the inclusion $\phi: {\mathcal O}_F[\sqrt{-\lambda}] \hookrightarrow \mathrm{End}(A_{\tilde{\eta}})$ 
given by  $\sqrt{-\lambda}\mapsto x\gamma_1 + y\gamma_2$.
Then $\phi$ extends to an inclusion
${\mathcal O}_E
 \hookrightarrow\mathrm{End}(A_{\tilde{\eta}})$, with $(1+\sqrt{-\lambda})/2$ mapping to
$(1+x\gamma_1 + y \gamma_2)/2$.
\end{proof}

\begin{remark}\label{CMonlyif}
Under Assumption~\ref{def:good}, there exists a real point with CM by $\cO_F[\sqrt{-\lambda}]$ if and only if $\lambda$ is represented by $q_{1,2}(x,y)$. We give a sketch of the proof of the ``only if'' part; we do not need this claim for the proof of our main theorem.

Suppose a point $\eta$ of $\Sh(\CC)$ has complex 
multiplication by ${\mathcal O}_F[\sqrt{-\lambda}]$. Let $s$ denote the endomorphism on the abelian variety $A_{\eta}$ corresponding to $\sqrt{-\lambda}$. By Lemma \ref{LsimpleCMtype}, we have that $s^\dagger=-s$, where $\dagger$ denotes the Rosati involution. Hence $s^\dagger \circ s = \lambda$ acting on $A_\eta$ and then $(sv, sw)=(s^\dagger\circ s v, w)=\lambda(v,w)$, where $(\cdot, \cdot)$ is the Hermitian form on $V \cong H^1(A_{\eta}, \QQ)$ as an $F$-vector space and $v, w \in V$. Therefore the matrix on $V$ associated to $s$ lies in $\GU_2(\cO_{F_0})$.
In other words, under the map $\pi:\bH\to \Sh(\CC)$, the point $\eta$ is the image of a point $\tilde{\eta}$ of $\bH$ 
which is stabilized by a matrix
$M\in \GU_2(\mathcal{O}_{F_0})$ corresponding to $s$.
Using our assumption that $\eta\in \Sh(\RR)$ and Proposition \ref{edges}, we may pick $\tilde{\eta}$ to lie on one of the geodesics connecting $\tilde{P}, \tilde{Q}, \tilde{R}$.
Since $s\circ s= -\lambda$, we have $M^2+\lambda \mathbb{I}=0$.
%SAVE It does not matter if $M$ is in standard or isotropic coordinate system because $\lambda \mathbb{I}$ remains the same in both coordinates.
Since $M$ cannot be a scalar matrix, the condition $M^2+\lambda \mathbb{I}=0$ is equivalent to $\mathrm{tr}(M)=0$ and $\det(M)=\lambda$. 

Recall that $\GU_2(\RR)=\CC^* \GL_2(\RR)$ and we then write $M=cM'$, where $c\in \CC^*$ and $M'\in \GL_2(\RR)$. Since $M$ fixes $\tilde{\eta}\in \bH$, we have $\det(M')>0$; since $\det(M)=\lambda >0$, we have $c\in \RR^*$ and we conclude that $M\in \GL_2(\RR)^+$. Then by Lemmas \ref{Lstabz} and \ref{detgamma}, since $\tilde{\eta}$ lies on one of the geodesics connecting $\tilde{P}, \tilde{Q}, \tilde{R}$, we have that $\lambda$ is represented by $q_{1,2}(x,y)$ with $\gamma_1\neq \gamma_2 \in \{\gamma_P, \gamma_Q, \gamma_R\}$. 
\end{remark}

%%%%%%%%%%%%%SECTION 5

\section{A fundamental triangle} \label{FundTr}

In this section, we determine a fundamental triangle 
for the action of a unitary similitude group on the upper half plane $\bH$.
The main output of the section is Corollary \ref{Lquadm5sec5},
in which we compute the quadratic forms that appear in Proposition~\ref{prop:QuadraticFormToCMpoints}.

\subsection{The triangle group} \label{S5atriangle}

Let $m=5$ and let $\zeta = \zeta_5$.
Let $F=\QQ(\zeta_5)$ and $F_0=\QQ(\zeta_5 + \zeta_5^{-1})$.
We determine three matrices in $\GL_2(F)$ that generate the triangle group $\Delta=\Delta(2,3,10)$. 

\begin{notation} \label{Nepsilonalpha}
Let $\zeta = \zeta_5$.
Let $\epsilon = \zeta+\zeta^{-1}$.  Let $\alpha = \zeta - \zeta^{-1}$.  Let $u=(1+\sqrt{5})/2$.
\end{notation}

\begin{lemma} \label{Lepsilonfacts} We note that:
\begin{enumerate}
\item $\epsilon = (-1+\sqrt{5})/2$ and $1/\epsilon = u = -(\zeta^3+\zeta^2)$;
%SAVE. Other formulas \item $\epsilon = -\zeta^2/(1+\zeta^{-1})$, 
%and $(1/\epsilon)^2 - (1/\epsilon)=1$;
\item $\epsilon^2 = (3- \sqrt{5})/2$, $\alpha^2 = -(5+\sqrt{5})/2$, and $\alpha^2= \epsilon^2-4$.
\end{enumerate}
\end{lemma}

\begin{proof}
The first part follows from the facts that $\epsilon >0$, $\epsilon$ is a root of $X^2+X-1$, 
and $0 = \zeta^4+\zeta^3+\zeta^2+\zeta+1$. The rest is a short calculation.
\end{proof}

\begin{definition} \label{DdefmatrixA}
Define the following matrices in standard coordinates:
\begin{align*}\label{Aeq}A_P=\begin{bmatrix}  -\zeta^{-1} & 0  \\  0&  \zeta \end{bmatrix},\,
A_Q=
\begin{bmatrix} -\zeta/\epsilon & {1} \\  1/\epsilon & -\zeta^{-1}/\epsilon \end{bmatrix},\,
\text{ and } 
A_R= 
\begin{bmatrix}
1/\epsilon & 
-\zeta^{-1} \\
\zeta/\epsilon
& -1/\epsilon 
\end{bmatrix}. 
\end{align*}
\end{definition}

\begin{proposition}\label{DeltaPQR}
The matrices in Definition~\ref{DdefmatrixA} have the following properties:
 $A_P,A_Q,A_R\in \GL_2(\mathcal{O}_F)$ have orders $10$, $3$ and $2$ respectively, and satisfy $A_PA_QA_R=\mathrm{Id}$.
 %Save: determinants $-1, 1, -1$.
 \end{proposition}
 
\begin{proof}
This can be verified computationally with Lemma~\ref{Lepsilonfacts}.
We found these matrices by fixing $A_P$ and finding conditions on $A_Q$ such that $A_Q^{2}=A_Q^{-1}$ and 
$A_PA_Q$ has order $2$.
\end{proof} 

In Section~\ref{S5bunitarysimilitude}, we show that $A_P,A_Q,A_R \in \GU_2(\mathcal{O}_{F_0})$.

\subsection{The unitary similitude group when $m=5$} \label{S5bunitarysimilitude}

In this section,
we work with the unitary similitude group from Section~\ref{S4unitarysimilitude} to obtain additional information.

\begin{notation}
For $M[11]$, set $m=5$.
We fix the values $\bb_1, \bb_2$ defined in Notation~\ref{Dbeta0Hermitian}.
Set \[\bb_1=1 \text{ and } \bb_2= (1-\sqrt{5})/2.\] 
Then $\bb_1$ and $\bb_2$ satisfy the positivity conditions:
$\tau_1(\bb_1)>0$, $\tau_1(\bb_2)<0$, and $\tau_2(\bb_1\bb_2)>0$.

Recall $\beta_0$ from \eqref{Ebeta0}, $\omega$ from \eqref{omega_eq}, and
$\epsilon$ and $\alpha$ from Notation~\ref{Nepsilonalpha}.
Let $r_\circ=-1$ and $s_\circ = \alpha/\epsilon$.
\end{notation}

\begin{lemma}
Set $m=5$.
Then $\beta_0 = \sqrt{5} \alpha$ and $\omega^2 = \epsilon = (-1+\sqrt{5})/2$.
 
Also $\omega^2(3r_\circ^2 + s_\circ^2) = -4$,
and $s_\circ=\alpha u$. 
\end{lemma}

\begin{proof}
We compute 
\[\beta_0 = \left(\frac{5}{\zeta^3 - \zeta^2}\right) \left( \frac{\zeta - \zeta^4}{\zeta - \zeta^4} \right)
= \frac{5 \alpha}{\zeta^4 - \zeta^3 - \zeta^2 + \zeta} = \frac{5 \alpha}{\sqrt{5}} = \sqrt{5} \alpha.\]
The second claim is true since $\omega^2 = - \bb_2$.
The third claim follows from $\omega^2(3r_\circ^2 + s_\circ^2) =  \epsilon (3 + \alpha^2/\epsilon^2)
= 4(\epsilon^2 -1)/\epsilon = -4$.
The fourth is a short calculation.
\end{proof}

\begin{proposition} \label{PdefXmatrixPQR}
Let $m=5$.  The three matrices $A_P$, $A_Q$, and $A_R$ from Definition~\ref{DdefmatrixA} are in $\GU_2(\mathcal{O}_{F_0})$.
In isotropic coordinates, these matrices are given by:
\begin{equation}\label{EXP5fix}
	X_P= \frac{\alpha}{2}
	\begin{bmatrix} 1  & (5-\sqrt{5})/2  \\  (-3+\sqrt{5})/10& 1
	\end{bmatrix};
\end{equation}

\begin{equation}\label{EXQ5fix}
	X_Q= \frac{1}{2}
	\begin{bmatrix} -1+2/\omega  & -\sqrt{5}(5+3\sqrt{5})/2 \\  (5+\sqrt{5})/10& -1 - 2/\omega
	\end{bmatrix}; \ \mbox{ and }
\end{equation}

\begin{equation} \label{EXR5fix}
	X_R= \frac{\alpha}{2\omega}
	\begin{bmatrix} 1  &  \sqrt{5}(- \epsilon - \omega(1+\sqrt{5})) \\  
	\frac{1-\sqrt{5}}{10} (\epsilon -\omega(1+\sqrt{5})) & -1
	\end{bmatrix}.
\end{equation}
\end{proposition}

\begin{proof} 
The matrices $A_P$, $A_Q$, and $A_R$ from Definition~\ref{DdefmatrixA} are in $\GL_2(F)$.
The formulas for $X_P$, $X_Q$, and $X_R$ follow from Lemma~\ref{Lmatrix}. 
Note that $X_P$, $X_Q$, and $X_R$ are in $\CC^* \GL_2(\RR)$ since $\omega, \epsilon \in \RR^+$.
Thus $A_P$, $A_Q$, and $A_R$ are in $\GU_2(\RR)$.
Since $\epsilon$ is a unit, the entries of $A_P$, $A_Q$, and $A_R$ are in $\mathcal{O}_{F}$.
Thus $A_P$, $A_Q$, and $A_R$ are in $\GU_2(\mathcal{O}_{F_0})$.
\end{proof}

\subsection{Vertices of fundamental triangles} \label{S5cstabilizers}

We determine the vertices of a fundamental triangle for the action of $\GU_2$ on $\bH$. 

Let $m=5$.
Recall the Deligne--Mostow Shimura variety $\Sh =M[11]$ from Section~\ref{M11asSh}.
Recall from Sections~\ref{SdefQR} and \ref{SdefP} that $P$ (resp.\ $Q$, $R$) is the point on $\Sh$ 
that represents the curve in the family having an extra automorphism of order $10$ (resp.\ $3$, $2$).

\begin{remark} \label{RdefPQRtilde}
We find the fixed points $\tilde{P}$, $\tilde{Q}$, and $\tilde{R}$ in $\bH$ of $X_P$, $X_Q$, and $X_R$:
\begin{eqnarray}\label{EtildePQR}
\tilde{P} & := & \beta_0 \approx  0 + 4.253 i; \\
\tilde{Q} & := & (5-\sqrt{5})\left(\frac{1}{\omega} + \frac{\sqrt{-3}}{2}\right)
\approx  3.516 + 2.394 i; \\ 
\tilde{R} & := & -5u \frac{(1-2\omega/\alpha)}{(\epsilon - \omega(1+\sqrt{5}))}
\approx  4.200 + 3.472 i.
\end{eqnarray}
\end{remark}

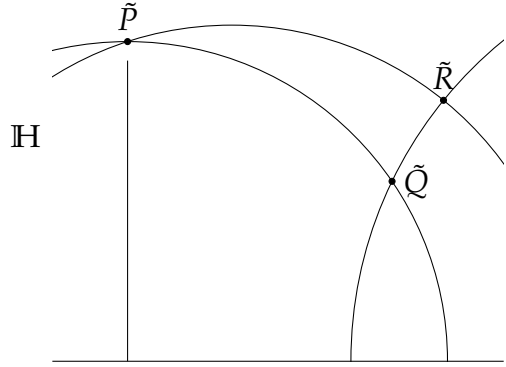
\begin{figure}[h!]
\begin{center}
\begin{tikzpicture}
\draw (-1,0)--(5,0);
\draw (0,0)--(0,4);

\node[above] at (0,4.253) {$\tilde{P}$};
\draw [fill] (0,4.253) circle [radius=0.04];
\node[right] at (3.516,2.394) {$\tilde{Q}$};
\draw [fill] (3.516,2.394) circle [radius=0.04];
\node[above] at (4.200,3.472) {$\tilde{R}$};
\draw [fill] (4.200,3.472) circle [radius=0.04];

\node[left] at (-1,3) {$\mathbb{H}$};

\clip (-1,0) rectangle (5,5);
\draw (0,0) circle(4.253);
\draw (8.478,0) circle(5.509);
\draw (1.382,0) circle(4.472);

\end{tikzpicture}
\end{center}
\caption{The hyperbolic triangle with vertices $\tilde{P}$, $\tilde{Q}$, and $\tilde{R}$}
    \label{fig:hyperbolictriangle}
\end{figure}

\begin{proposition} 
The images of the points $\tilde{P}$, $\tilde{Q}$, and $\tilde{R}$ under $\pi: \bH \to \Sh(\CC)$ are $P$, $Q$, and $R$ respectively.
A fundamental domain in $\bH$ for the action of $\Delta$ 
is given by the union of two adjacent copies of the hyperbolic triangle 
$\triangle$ in $\bH$, 
whose vertices are the points $\tilde{P}$, $\tilde{Q}$, and $\tilde{R}$.
\end{proposition}

\begin{proof}
Recall that $\pi: \bH \to \Sh(\CC)$ is the quotient by $\Delta=(\GU_2(\cO_{F_0})/Z)\cap(\SL_2(\RR)/Z)$.
For each of $P,Q,R$, the stabilizer in $\Delta$ of a lift in $\bH$ of the point
is a finite cyclic subgroup of $(\GU_2(\RR)/Z) \cap (\mathrm{SL}_2(\RR)/Z)$.
By Proposition~\ref{DeltaPQR}, 
the vertices of a fundamental triangle are the fixed points of the three matrices 
$A_P, A_Q, A_R\in \Delta$ from Definition~\ref{DdefmatrixA}. 
Using Lemma~\ref{action-on-H}, we compute the fixed point in $\bH$ of $X_P$, $X_Q$, and $X_R$.
\end{proof}

%SAVE For example, $X_P$ stabilizes $\tilde{P}$.  Then $\pi(\tilde{P})=P$ because $A_P$ has order $10$ 
%and $P$ represents the curve having an extra automorphism of order $10$.
%The proofs for $\tilde{Q}$ and $\tilde{R}$ are similar.
%The point $\tilde{Q}$ is a root of $z^2(5+\sqrt{5})/10 + z(-4/\omega) + (15+5 \sqrt{5})/2$; and $\tilde{R}$ is a root of 
%\[((1-\sqrt{5})/10)(\epsilon - \omega(1+\sqrt{5}))z^2 - 2z + \sqrt{5}(\epsilon + \omega(1+\sqrt{5})).\]
%Note that $\tilde{Q}$ and $\tilde{R}$ have positive imaginary part.

\begin{remark}
Let $\alpha$ be the area of a fundamental region for $\Delta(2,3,10)$.
By the Gauss--Bonet theorem, 
$\alpha = \pi - (\pi/2 + \pi/3 + \pi/10)$. %SAVE $\approx .209$.
We checked that $\alpha = \mathrm{Area}(\triangle)$, by finding the hyperbolic distances between $\tilde{P}$, $\tilde{Q}$, and $\tilde{R}$, and using 
the formula for the area of a triangle with 
a right angle at $\tilde{R}$. \end{remark}
%SAVE \[d(\tilde{P}, \tilde{Q}) = 1.177; \ 
%1.17733202707834,
%SAVE d(\tilde{Q}, \tilde{R}) = .440; \text{ and }
%0.439589640433368,
%SAVE d(\tilde{R}, \tilde{P}) = 1.061.\]
%1.06127506190434.
%SAVE \[\mathrm{sin}(\alpha) = \mathrm{sinh}(d(\tilde{P}, \tilde{Q})) \mathrm{sinh}(d(\tilde{Q}, \tilde{R})) / (1 + \mathrm{cosh}(d(\tilde{P}, \tilde{R}))).\] %(.208).

%SAVE \begin{remark}
%The geodesic between $\tilde{Q}$ and $\tilde{R}$
%has center $(\mathrm{xcenter}, 0)$ where
%\begin{eqnarray*}    
%\mathrm{xcenter} & = & (5/19) \left((-12(\zeta^3 + \zeta^2) + 10)w + (-5(\zeta^3+\zeta^2) + 1) \right)\\
%& = & (5/19)\left((12\epsilon + 22) w + (5\epsilon+6) \right) \sim 8.478.
%\end{eqnarray*}
%It has radius $r \approx 5.509$.  One can find an algebraic expression for this using that
%$r^2 = (\mathrm{xcenter} - \mathrm{re}(\tilde{Q}))^2 + \mathrm{im}(\tilde{Q})^2$.
%Check:
%$r = (2 \sqrt{5}/19) \sqrt{T}$, 
%where $T=\omega(2^243(1+\epsilon) + 4) + 173(1+\epsilon) + 46$.

%The geodesic between $\tilde{P}$ and $\tilde{R}$ has center $\approx 1.382$
%and radius $\approx 4.472$.
%\end{remark}

\subsection{Stabilizing elements with trace zero}

Following Section \ref{tzsqf}, for each of $z=\tilde{P},\tilde{Q},\tilde{R}$, 
we compute $\gamma_z\in\mathrm{Stab}(z)\subset \GL_2(\RR)^+$ satisfying $\mathrm{tr}(\gamma_z)=0$.
(Here $\gamma_z$ is well-defined in $\GU_2(\RR)^+$ up to multiplication by a scalar in $\RR^*$.)
Recall that $s_\circ = \alpha/\epsilon$.

\begin{lemma} \label{gammaPQR_lem}
In standard coordinates: 
\begin{equation} \label{Egammafix}
\gamma_P=
\begin{bmatrix} -\alpha &0\\0&\alpha \end{bmatrix},\,
\gamma_Q =
\begin{bmatrix}
-s_\circ & 2 \\
2/\epsilon & s_\circ
\end{bmatrix},  \text{ and }
\gamma_R= \begin{bmatrix} s_\circ & -\zeta^{-1}\alpha \\\zeta \alpha/\epsilon& -s_\circ \end{bmatrix}.
\end{equation}
In particular, $\gamma_P, \gamma_Q, \gamma_R$ are in $\mathrm{GL}_2({\mathcal O}_{F})$.
In fact, $\gamma_P, \gamma_Q, \gamma_R$ are in ${\mathrm{GU}_2}({\mathcal O}_{F_0})\cap \GL_2(\RR)^+$.
\end{lemma}

\begin{proof}
For any $z\in\bH$, if $A\in\CC^* \cdot\mathrm{Stab}(z)$ with $\mathrm{tr}(A) \not = 0$, 
then $\gamma=2A-\mathrm{tr}(A)\mathrm{Id}\in \CC^* \cdot \mathrm{Stab}(z)$,  
and $\mathrm{tr}(\gamma)=0$. 

From Definition~\ref{DdefmatrixA} and Proposition~\ref{DeltaPQR}, we compute that
$\mathrm{tr} (A_P)= \alpha$.  Thus we compute $\gamma_P= 2A_P-\alpha\mathrm{Id}=
\epsilon \cdot \mathrm{Diag}(-1,1)$.
This has determinant $-\epsilon^2<0$; we obtain the given representative for $\gamma_P$ in 
$\GU_2(\RR)^+$ by scaling the above matrix by $\alpha/\epsilon \in \CC^*$.

%SAVE In isotropic coordinates, this fixes $\beta_0$.
%$X_{\gamma_P} = \begin{bmatrix}
%0 & \beta_0 \alpha \\
%\beta_0^{-1} \alpha & 0
%\end{bmatrix},  

Similarly, $\mathrm{tr}(A_Q) = -1$ and we compute 
$\gamma_Q = 2A_Q-\mathrm{tr}(A_Q)\mathrm{Id}$, which has positive determinant $3$.
Note that $\mathrm{tr}(A_R)=0$ and $\mathrm{det}(A_R) = -1$.
We obtain the given representative for $\gamma_R$ in $\GU_2(\RR)^+$ by scaling $A_R$ by $\alpha \in \CC^*$.

The last assertion follows from Lemma \ref{detgamma}.
\end{proof}

Write $u=(1+\sqrt{5})/2$.
By Lemma~\ref{gammaPQR_lem}:
\begin{eqnarray} \label{Edeterminantgamma}
\det(\gamma_P) & = & -s_\circ^2\epsilon^2 = - \alpha^2 = \sqrt{5} u >0, \\ 
\det(\gamma_Q) & = & -(s_\circ^2+4/\epsilon)=3>0, \mbox{ and } \\
\det(\gamma_R) & = & -s_\circ^2 + \alpha^2/\epsilon= \sqrt{5} u  >0.
\end{eqnarray}

\subsection{Computation of quadratic forms} \label{Scomputequadraticform}

We find the geodesics that determine the edges of the fundamental triangle.
Recall, for a pair $\gamma_1,\gamma_2$ of $\gamma_P,\gamma_Q,\gamma_R$, that 
\begin{equation} \label{Erepeatdet}
q_{1,2}(x,y)=\det (x\gamma_1 + y \gamma_2).
\end{equation}

\begin{corollary} \label{Lquadm5sec5}
Let $u = (1+\sqrt{5})/2$.  
The three quadratic forms for $M[11]$ are:
\begin{enumerate}
\item $q_{Q,R}(x,y)= 
3 x^2  - 2 \sqrt{5} u xy  + \sqrt{5} u y^2$,
with discriminant $\Delta_{Q,R}= 
4\sqrt{5}$.

\item $q_{Q,P}(x,y)= 
3 x^2 + 2\sqrt{5} u^2 xy +  \sqrt{5} u y^2$, with discriminant 
$\Delta_{P,Q}=16\sqrt{5} u^2$.

\item 
$q_{P,R}(x,y) = \sqrt{5}u(x^2 - 2uxy + y^2)$, with discriminant 
$\Delta_{P,R} = 20u^3$.  
\end{enumerate}
\end{corollary}

\begin{proof} This follows from Lemma~\ref{gammaPQR_lem} and \eqref{Edeterminantgamma} - \eqref{Erepeatdet}.
For example:
\begin{equation} \label{EqQP}
q_{Q,P}(x,y)  =  \mathrm{det}(x \gamma_Q + y \gamma_P)
=\mathrm{det}(\gamma_Q) x^2 - 2xy( s_\circ\alpha) + \mathrm{det}(\gamma_P) y^2; \text{ and}
\end{equation}
%SAVE\begin{eqnarray*} 
%SAVE q_{Q,P}(x,y) & = & \mathrm{det}(x \gamma_Q + y \gamma_P) \\
%SAVE & = & - (s_\circ x + y\alpha)^2 - 4x^2/\ep) \\
%& = & \mathrm{det}(\gamma_Q) x^2 - 2xy( s_\circ\alpha) + \mathrm{det}(\gamma_P)y^2. 
%SAVE & = &  3 x^2 -2 (\alpha^2/\ep) xy - \alpha^2 y^2
%\end{eqnarray*}
\begin{equation} \label{EqQR}
q_{Q,R}(x,y) = \mathrm{det}(x \gamma_Q + y \gamma_R)
=\mathrm{det}(\gamma_Q) x^2 + 2xy( s_\circ^2 - \alpha s_\circ) + \mathrm{det}(\gamma_R) y^2. \qedhere
\end{equation}
%SAVE \begin{eqnarray*} 
%q_{Q,R}(x,y) & = & \mathrm{det}(x \gamma_Q + y \gamma_R) \\
%SAVE & = & - (s_\circ x-ys_\circ)^2 - ((2/\epsilon)x + \zeta(\alpha/\epsilon)y)(2x - \zeta^{-1} \alpha y) \\
%& = & \mathrm{det}(\gamma_Q) x^2 + 2xy( s_\circ^2 - \alpha s_\circ) %+ \mathrm{det}(\gamma_R) y^2. 
%SAVE & = &  3 x^2 + 2 \alpha^2 xy - \alpha^2 y^2
%\end{eqnarray*}
\end{proof}

%SAVE This follows from
%\begin{eqnarray*} 
%q_{P,R}(x,y) & = & \mathrm{det}(x \gamma_P + y \gamma_R) \\
%& = & - (\alpha x - s_\circ y)^2 + (\alpha^2/\epsilon) y^2 \\
%& = & - \alpha^2 x^2 + 2 \alpha s xy +(-s_\circ^2 + (\alpha^2/\epsilon)) y^2 \\
%& = & - \alpha^2 x^2 + 2 (\alpha^2/\epsilon) xy -\alpha^2 y^2.
%\end{eqnarray*}

\begin{remark}\label{notPR}
The quadratic form $q_{P,R}(x,y)$ is not primitive, and 
we do not use it in later sections.\footnote{The reason $q_{P,R}$ is not primitive that we scaled by elements in $(\alpha)\subset \cO_F$ to obtain $\gamma_P,\gamma_R\in \GL_2(\RR)^+$. One can obtain statements analogous to Lemma~\ref{detgamma} and Proposition~\ref{prop:QuadraticFormToCMpoints} by working with elements in $i\GL_2(\RR)^+$; then we do not need to scale $2A_P-\alpha\mathrm{Id}$ and $A_R$, and can work with a primitive quadratic form.}
In particular, if $\lambda\in {\mathcal O}_{F_0}$ is a totally positive irreducible element, 
$\lambda\not \in \langle \sqrt{5} \rangle$,
then $\lambda$ is not represented by $q_{P,R}$.  
By Remark~\ref{CMonlyif}, $G_{PR}$ does not contain special points with complex multiplication by $\mathcal{O}_F[\sqrt{-\lambda}]$. 
\end{remark}

The quadratic form $q_{Q,R}$ is fundamental, in the sense of 
Zemkova \cite{zemkova}.
The quadratic form $q_{Q,P}$ is not fundamental, 
because of the power of $2$ dividing $\Delta_{P,Q}$.
We change variables to write $q_{Q,P}$ in a more simple form.

\begin{lemma} \label{solmod2QP}
Suppose $x,y \in F_0$.  
Write $x_1=2x$, $y_1 =2y$, and $d_1 = y+ux$.
\begin{enumerate}
\item The matrix $x\gamma_Q + y \gamma_P$ is in 
${\mathrm{GU}}_2({\mathcal O}_{F_0})$ if and only if 
$x_1,y_1 \in {\mathcal O}_{F_0}$ and 
$y_1 \equiv -ux_1 \bmod 2{\mathcal O}_{F_0}$, 
which is equivalent to $x_1, d_1 \in \mathcal{O}_{F_0}$.
\item With respect to this change of variables,
\begin{equation} \label{EdiagqQP}
    q_{Q,P}(x,y) = -u (x_1^2 - \sqrt{5}d_1^2). 
%   = -uN_{L_1/F_0}(x_1 + \sqrt[4]{5} d_1).
\end{equation}
\item Then $q_{Q,P}(x,y) \equiv -1 \bmod 4\mathcal{O}_{F_0}$ 
if and only if either 
(a) $(x_1,d_1) =(0,u) \bmod 2\mathcal{O}_{F_0}$ or 
(b) $(x_1,d_1) = (u^2,1) \bmod 2\mathcal{O}_{F_0}$.
\item The entries of ${\mathrm{Id}}+x\gamma_Q + y \gamma_P$ are all zero modulo $2\mathcal{O}_{F}$ in case (a) and are not all zero 
modulo $2\mathcal{O}_{F}$ in case (b).
\end{enumerate}
\end{lemma}

\begin{proof}
\begin{enumerate}
\item Since $x,y \in F_0$, these matrices are in 
${\mathrm{GU}}({\mathcal O}_{F_0})$
if and only if their entries are in $\mathcal{O}_F$.
By Corollary~\ref{Lquadm5sec5}, $2$ and $\sqrt{5}$ are the only primes dividing the discriminant of $q_{Q,P}(x,y)$,   
and the multiplicity of $\sqrt{5}$ in the discriminant is odd.
Thus it suffices to check integrality at $2$. 
The coefficients of $x \gamma_Q + y \gamma_P$
are $\pm \alpha(ux+y)$, $2x$, and $2xu$.  Thus the integrality
condition is equivalent to $2x \in {\mathcal O}_{F_0}$
and $ux+y \in {\mathcal O}_{F_0}$.

\item We compute that
    \begin{eqnarray*}
        q_{Q,P}(x,y) & = & 3x^2 + 2\sqrt{5}u^2 xy + \sqrt{5}uy^2 \\
        %SAVE & = &  
        %\left(3x_1^2 + 2\sqrt{5}u^2x_1(-ux_1 + 2t_1)
        %+ \sqrt{5} u(-ux_1 + 2t)^2\right)/4\\
        & = & 
        \left(x_1^2(3-\sqrt{5}u^3) + x_1d_1(0) + d_1^2(4\sqrt{5}u)\right)/4\\
        & = & -u(x_1^2 - \sqrt{5} d_1^2).
    \end{eqnarray*}
    \item Note that $\Omega = \{0,1,u,u^2\}$ is a set of representatives
for the cosets of $\cO_{F_0}$ modulo $2\cO_{F_0}$.
By part (2), the congruence of 
$q_{P, Q}(x,y) \bmod 4\cO_{F_0}$ is determined by the 
congruences of $x_1$ and $d_1$ modulo $2\cO_{F_0}$.
We check the 16 pairs $(x_1,d_1)$ with $x_1,d_1 \in \Omega$
to determine if $q_{P, Q}(x,y) \equiv -1 \bmod 4\cO_{F_0}$, 
leading to the 2 listed pairs.

\item 
Note that 
$s_\circ = \alpha u = 2 \zeta^2 + 2\zeta + 1 \equiv 1 \bmod 2{\mathcal O}_{F_0}$.
%$s_\circ =\zeta + \zeta^2 - \zeta^3 - \zeta^4$.
So the entries of ${\mathrm{Id}}+x\gamma_Q + y \gamma_P$
are $0$ and $1 + x \pm y \alpha \bmod 2{\mathcal O}_{F_0}$.
It suffices to check the case $(x,y) = (0,u)$ in case~(a), 
and the case $(x,y) = (u^2/2, 1-u^3/2)$ in case~(b).
    \end{enumerate}
\end{proof}

\subsection{More information about the geodesic}

\begin{lemma}
The geodesic $G_{PQ}$ is the half circle centered at $0$ with radius $r:=\beta_0 (-i)$.  
\end{lemma}

\begin{proof}
This is true because the point $\tilde{Q}$ is on the circle with radius $r$. 
\end{proof}

%SAVE Note that
%$\tilde{Q} = \beta_0 \frac{1}{\omega \alpha u}(2 + \sqrt{3u_\tau})$.

%SAVE To parametrize $G_{PQ}$, 
%one can use the fractional linear transformation that takes the endpoint $r$ to $0$, the endpoint $-r$ to $\infty$, and $\tilde{P}$ to $i$.  This transformation is 
%$f(z) = (z+\beta_0 i)/(z - \beta_0 i)$.

\begin{proposition} \label{PfixedpointQP}
Let $M_{x,y} = x \gamma_Q + y \gamma_P$ for $x,y \in \RR$ such that 
$\mathrm{det}(M_{x,y}) >0$.
Let $t=x_1/d_1 =2x/(y+ux)$.
Then the fixed point $z \in \bH$ of $M_{x,y}$ is
\begin{equation} \label{EfixedpointQP}
z=\frac{\beta_0}{\omega \alpha} 
\left(t + \sqrt{t^2 - \sqrt{5}}\right).
\end{equation}
\end{proposition}

Note that $\mathrm{det}(M_{x,y}) >0$ if and only if 
$t^2 < \sqrt{5}$.
%Also $\mathrm{Re}(z) >0$ if and only if $\frac{x_1}{d_1} >0$. 
Note that $\beta_0/(\alpha \omega)\in \RR^+$.

\begin{proof}
In terms of the coordinates $x_1$ and $d_1$, then 
\begin{equation} \label{EparametrizeGQPmatrix}
M_{x,y} =
%x \gamma_Q + y \gamma_P = 
\begin{bmatrix} -\alpha (ux+y) & 2x \\
2xu &\alpha (ux+y) \end{bmatrix} =
\begin{bmatrix} -\alpha d_1 & x_1 \\
ux_1 &\alpha d_1 \end{bmatrix}.
\end{equation}
Note that $\alpha^2/u = - \sqrt{5}$.
%SAVE The condition $\left( \frac{x_1}{d_1} \right)^2 < \sqrt{5}$
%is equivalent to the condition $\mathrm{det}(M_{x,y}) >0$.
By Lemma~\ref{Lmatrix},
$M_{x,y}$ is given in isotropic coordinates by
\[X_{x,y} = \frac{1}{2\omega} 
\begin{bmatrix} 2x_1 & \beta_0(2 \omega \alpha d_1) \\
\beta_0^{-1}(2 \omega \alpha d_1) & -2x_1 \end{bmatrix}.\]
The fixed point $z$ is the root in $\bH$ of 
$f_{x,y} = \beta_0^{-1} (2 \omega \alpha d_1)z^2
-4x_1 z - \beta_0 (2 \omega \alpha d_1)$.
The quadratic formula implies \eqref{EfixedpointQP}. 
%SAVE \[z=\frac{\beta_0}{\omega \alpha}\frac{1}{d_1} 
%\left(x_1 \pm \sqrt{x_1^2 + \omega^2 \alpha^2d_1^2} \right), \]
%which simplifies to the given expression.

\end{proof}

%SAVE \begin{remark}
%The previous result is compatible with Remark~\ref{RdefPQRtilde}.
%Specifically, the root of $f_{0,1}$ in $\bH$ is $\tilde{P}=\beta_0$. 
%The root of $f_{1,0}$ in $\bH$ is
%\[\frac{\beta_0}{\omega \alpha u} 
%\left(2 \pm \sqrt{4 - \sqrt{5}u^2}\right)=
%\frac{\beta_0}{\omega \alpha u}(2 \pm \sqrt{3u_\tau}) = \tilde{Q}.\]
%SAVE: the quadratic equation
%$0 = \beta_0^{-1}(2 \omega \alpha u) z^2 - 8z - \beta_0(2\omega \alpha u)$ is the same as $4\omega$ multiplied by the 
%quadratic equation 
%$(1/10)(5 + \sqrt{5}) z^2 - (4/\omega) z + (15+5 \sqrt{5})/2$.
%\end{remark}

%%%%%%%%%%%%%SECTION 6

\section{Existence of real CM points}  \label{SrealCM}

In this section, for $\Sh=M[11]$, 
we identify totally positive irreducible elements $\lambda\in \mathcal{O}_{F_0}$ 
for which there exist two points of $\Sh(\RR)$, 
one having complex multiplication by $\mathcal{O}_F[\sqrt{-\lambda}]$ and the other having complex multiplication 
by ${\mathcal O}_E$, where $E=F(\sqrt{-\lambda})$.
Furthermore, we show the existence of (infinitely many)
such $\lambda$ satisfying the congruence conditions 
that guarantee the uniqueness of such points, by Theorem~\ref{Tuniquecong}. 

From Remark~\ref{RdefPQRtilde}, recall that 
$\tilde{P}$, $\tilde{Q}$, $\tilde{R}$ are the vertices of the chosen fundamental triangle $\mathfrak T$ for the action of 
$\Delta$ on $\bH$. 
They are the fixed points of the matrices 
$X_P$, $X_Q$, and $X_R$ from Proposition~\ref{PdefXmatrixPQR}, 
and their images under $\pi$ in $\Sh(\QQ)$ 
are $P$, $Q$, and $R$ respectively.
Let $G_{QP}$, $G_{QR}$, and $G_{PR}$ denote the
three geodesics in $\bH$ that form the edges of $\mathfrak T$.
By Proposition~\ref{edges}, the images under $\pi$ of these geodesics cover $\Sh(\RR)$.
We focus on the geodesic 
$G_{QP}$ which contains $\tilde{Q}$ and $\tilde{P}$.

\subsection{Geodesics covering two arches of $\Sh$}

\begin{notation}
There is a continuous map between $\Sh(\RR)$ and a circle.  
Then $\Sh(\RR) - \{P,R\}$ has two connected components,
$C_1$ and $C_2$, where $C_1$ contains $Q$.
We define two arches covering $\Sh(\RR)$, 
namely $\overset{\frown}{PQR} = C_1 \cup \{P, R\}$ 
and $\overset{\frown}{PR} = C_2 \cup \{P,R\}$.
Let $V = \pi^{-1}\{P,Q,R\}$.
\end{notation}

\begin{lemma}\label{geopartition}
The restriction of $\pi: \mathbb{H} \to \Sh(\mathbb{C})$
to the geodesic $G_{QP}$ (resp.\ $G_{QR}$) maps onto the arch $\overset{\frown}{PQR}$ in $\Sh(\RR)$.
The restriction of $\pi$ to $G_{PR}$ maps onto the arch $\overset{\frown}{PR}$ instead.
\end{lemma}

\begin{proof}
Consider the geodesic $G_{QP}$ containing $\tilde{P}$ and $\tilde{Q}$.
Let $z \in V \cap G_{QP}$.
Let $z_l$ (resp.\ $z_r$) be the point on $G_{QP}$ to the left (resp.\ right) 
of $z$, which is the closest point to $z$ in $V$.

Suppose $\pi(z)=Q$.   
The number of geodesics in $\pi^{-1}(\Sh(\RR))$ passing through $z$ equals the 
order of $A_Q$ (which is $3$).  There are 6 hyperbolic edges emanating from $z$
and the points in $V$ closest to $z$ on these edges 
alternate between pre-images of $P$ and $R$.  So the two of these 
points on $G_{QP}$ satisfy $\pi(z_l) = P$ and $\pi(z_r) = R$ or vice-versa.

Suppose $\pi(z)=P$ (resp.\ $\pi(z)=R$). 
The number of geodesics in $\pi^{-1}(\Sh(\RR))$ passing through $z$ 
equals $10$ (resp.\ $2$).
The points in $V$ closest to $z$ on these edges alternate between pre-images of $Q$ and $R$ (resp.\ $Q$ and $P$).
So the two of these 
points on $G_{QP}$ satisfy $\pi(z_l) = \pi(z_r) = Q$. 

Thus $G_{QP}$ is the union of pre-images of the arch $\overset{\frown}{PQR}$.
Similar arguments apply for the geodesic $G_{QR}$.
In contrast, $G_{PR}$ is the union of preimages of the arch $\overset{\frown}{PR}$.
\end{proof}

\begin{notation}
Define $R_1$ (resp.\ $Q_1$), (resp.\ $P_1$) 
to be the point in $\bH$ fixed by 
$A_{R_1}:=A_Q^{-1} \gamma_R A_Q$, 
%SAVE 
%$A_{R_1} = \begin{bmatrix} 
%\zeta^3 + 3 \zeta^2 + 4 \zeta + 2 & \zeta^3 + \zeta^2-2 \\
%3 \zeta^3 + 3 \zeta^2 -1 & -\zeta^3 - 3 \zeta^2 - 4 \zeta -2 \end{bmatrix}$.
(resp.\ $A_{Q_1}:=(A_Q^{-1} A_R A_Q) \gamma_Q (A_Q^{-1} A_R A_Q)^{-1}$), 
%previous version: (resp.\ $A_{Q_1}:=A_{R_1} \gamma_Q A_{R_1}^{-1}$), these two matrices are the same, but the current version is more directly related to the geometric nature; see proof of the lemma below.
%SAVE
%$A_{Q_1} = \begin{bmatrix} 
% - 2\zeta^3 - 4 \zeta^2 - 6 \zeta -3 & 2u+2 \\
%4u+2 & 2\zeta^3 + 4 \zeta^2 + 6 \zeta +3 \end{bmatrix}$.
(resp.\ $A_{P_1}:= (A_Q^{-1} A_R A_Q)\gamma_P (A_Q^{-1} A_R A_Q)^{-1}$).
\end{notation}

\begin{lemma}
The points $R_1$, $Q_1$, and $P_1$ are on the geodesic $G_{QP}$, 
with the parameters $t_{R_1} = -u + 3$, $t_{Q_1} = (2u+4)/5$, and 
$t_{P_1} = (4u-2)/3$. 
They lie to the right of $\tilde{Q}$, and are the closest points on the right of $\tilde{Q}$ that lie above $R$, $Q$, and $P$ respectively.
\end{lemma}

\begin{proof}
We compute that $A_{R_1}$ is of the form in 
\eqref{EparametrizeGQPmatrix}, for $x_1 = -u-2$ and $d_1=-u-1$.
Thus $R_1$ is on the geodesic $G_{QP}$, with the parameter $t_{R_1} = -u + 3$. 

Also $A_{Q_1}$ is of the form in 
\eqref{EparametrizeGQPmatrix}, for $x_1 = 2u+2$ and 
$d_1=(1/\alpha)(2\zeta^3 + 4 \zeta^2 + 6 \zeta +3)$.
Thus $Q_1$ is on the geodesic $G_{QP}$, with the parameter $t_{Q_1} = (2u+4)/5$.

Also $A_{P_1}$ is of the form in 
\eqref{EparametrizeGQPmatrix}, for $x_1 = 8u+6$ and 
$d_1=3(2u+1)$.
Thus $P_1$ is on the geodesic $G_{QP}$, with the parameter $t_{P_1} = (4u-2)/3$.

Since $t_{\tilde{Q}} < t_{R_1} < t_{Q_1} < t_{P_1}$, these points lie to the 
right of $\tilde{Q}$.
There are three hyperbolic geodesics passing through $\tilde{Q}$ lying in 
$\pi^{-1}(\Sh(\RR))$. A direct computation shows that $A_Q$ acts on the tangent space of $\bH$ at $\tilde{Q}$ as $e^{2 \pi i /3}$, thus $A_Q^{-1}(\tilde{R})$ is the point lying above $R$ that is closest to $\tilde{Q}$ on the right side. Since $\gamma_R$ is the stabilizer of $\tilde{R}$, the stabilizer of this point is $A_Q^{-1} \gamma_R A_Q$. Then by definition, this point is $R_1$.

The matrix $A_Q^{-1} A_R A_Q$ has order $2$ and stabilizes $R_1$. 
Thus $A_{R_1}$ takes 
the point $\tilde{Q}$ to the point which is the closest point 
to the right of $\tilde{Q}$ lying above $Q$. Since $\gamma_Q$ is the stabilizer of $\tilde{Q}$, thus the stabilizer of this point is $A_{R_1} \gamma_Q A_{R_1}^{-1}$. Then by definition, this point is $Q_1$. Moreover, $A_Q^{-1} A_R A_Q$ also takes 
the point $\tilde{P}$ to the point which is the closest point 
to the right of $\tilde{Q}$ lying above $P$. Since $\gamma_P$ is the stabilizer of $\tilde{P}$, thus the stabilizer of this point is $A_{R_1} \gamma_Q A_{R_1}^{-1}$. Then by definition, this point is $P_1$.
\end{proof}

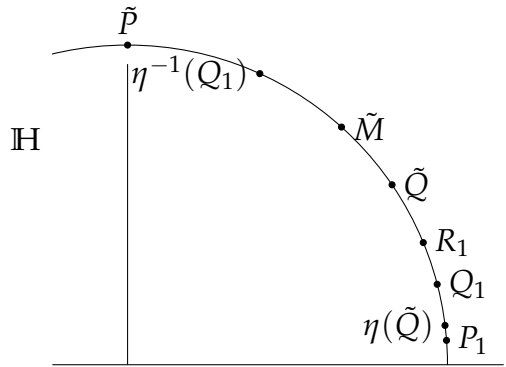
\begin{figure}[h!]
\begin{center}
\begin{tikzpicture}
\draw (-1,0)--(5,0);
\draw (0,0)--(0,4);

\node[above] at (0,4.253) {$\tilde{P}$};
\draw [fill] (0,4.253) circle [radius=0.04];
\node[right] at (3.516,2.394) {$\tilde{Q}$};
\draw [fill] (3.516,2.394) circle [radius=0.04];
%\node[above] at (4.200,3.472) {$\tilde{R}$};
%\draw [fill] (4.200,3.472) circle [radius=0.04];

\node[left] at (1.758, 3.873) {$\eta^{-1}(Q_1)$};
\draw [fill] (1.758, 3.873) circle [radius=0.04];

\node[right] at (2.844, 3.162)  {$\tilde{M}$};
\draw [fill] (2.844, 3.162)  circle [radius=0.04];

\node[right] at (3.931, 1.625)  {$R_1$};
\draw [fill] (3.931, 1.625) circle [radius=0.04];

\node[right] at (4.116, 1.070)   {$Q_1$};
\draw [fill] (4.116, 1.070)  circle [radius=0.04];

\node[left] at (4.220, .523)  {$\eta(\tilde{Q})$};
\draw [fill] (4.220, .523) circle [radius=0.04];

\node[right] at (4.240, .324) {$P_1$};
\draw [fill] (4.240, .324) circle [radius=0.04];

\node[left] at (-1,3) {$\mathbb{H}$};

\clip (-1,0) rectangle (5,5);
\draw (0,0) circle(4.253);
%\draw (8.478,0) circle(5.509);
%\draw (1.382,0) circle(4.472);
\end{tikzpicture}
\end{center}
\caption{Analytic position of points on the geodesic $G_{Q,P}$}
    \label{fig:pointsongeodesic}
\end{figure}

\subsection{Field extensions of $F_0$}
We restrict to $F_0=\QQ(\sqrt{5})$
with the implicit choice of two real embeddings: $\tau_1(\sqrt{5}) = \sqrt{5}$ and 
$\tau_2(\sqrt{5}) = -\sqrt{5}$. 
Let $\tau=\tau_2$ denote the non-trivial element of $\mathrm{Gal}(F_0/\QQ)$.
For $z \in F_0$, let $z^\tau$ denote its Galois conjugate.   
Recall that $u= (1+\sqrt{5})/2$ and $u^\tau=(1-\sqrt{5})/2$.
So $uu^\tau=-1$.

Direct computations with SAGE yield the following statements.
\begin{lemma} \label{LclassgroupQR}
Let $L=\QQ[t]/\langle t^4-5 \rangle$ which is a degree $2$ extension of $F_0$ with the nontrivial Galois action $t \mapsto -t$. 
\begin{enumerate}
\item $L$ has class number $1$;
%SAVE \item As an ${\mathcal O}_{F_0}$-module, ${\mathcal O}_L$ has basis $\{1, t\}$;
%SAVE \item the prime $2$ ramifies in the extension $L/F_0$; 
\item $\mathcal{U}_L\simeq \{\pm 1\} \times \ZZ^2$,
is generated by $\{-1, u,\eta\}$ where
$\eta= (t^3+t^2+t+3)/2$.

\item $N_{L/F_0} (\eta) = 1$
%SAVE $\cU_{F_0}^+=N_{L/F_0}(\cU_L)$.
and $\pm u,\pm u^\tau\not\in N_{L/F_0}(\cU_L)$.
\end{enumerate}
\end{lemma}

%SAVE \begin{proof} Follows from SAGE calculations. Note that
%\[\langle 2 \rangle = \langle(t+1)(t^2+1)/2 \rangle^2.\]
%SAVE For the last part, write $\eta = t(t^2+1)/2 + (t^2+3)/2$.  
%Using that $t^4 = 5$, we see that
%SAVE \[N_{L/F_0}(\eta) =  
%(t(t^2+1)/2 + (t^2+3)/2)(-t (t^2+1)/2$ + (t^2+3)/2)
%-t^2 (t^2+1)^2/4 + (t^2+3)^2/4 
%SAVE -t^2(3+t^2)/2 + (7+3t^2)/2 = 1. \qedhere \] \end{proof}

%If $\delta \in \tilde{L}$ is non-zero, we note that $N_{\tilde{L}/F_0}(\delta)$ is totally positive. 
%SAVE note that $F_0(i)$ is an intermediate field of $\tilde{L}/F_0$. 
%The norm of any element of $F_0(i)$ is of the form $a^2+b^2$ with 
%$a,b \in F_0$. All non-zero squares of $F_0$ are totally positive.

\begin{lemma}\label{tildeL}
Let $\tilde{L}$ be the splitting field of $t^4-5\in\QQ[t]$.  Then $\tilde{L}$ is a degree 4 extension of $F_0$.
\begin{enumerate}
\item $\tilde{L}$ has class number $2$;
%\item as an $\mathcal{O}_{F_0}$-module, $\mathcal{O}_{\tilde{L}}$ has basis \[\{ 1, \sqrt[4]{5}, z_1:=(1+\sqrt[4]{5})(1+i)/2, \ z_2:=\sqrt[4]{5} \cdot z_1 \}.\]
%SAVE ${\mathcal O}_{F_0}$ span of $\{1, \sqrt[4]{5}, i, i \sqrt[4]{5}\}$ has index 4 in ${\mathcal O}_{\tilde{L}}$.	
%\item the prime $2$ is totally ramified in $\tilde{L}/F_0$;
%\item a set of generators for $\mathcal{U}_{\tilde{L}} \simeq C_4 \times \ZZ^3$ is given by $\{v_0,v_1,v_2,v_3\} $ where\begin{eqnarray*}v_0 & := & i, \  v_1:=i u, \ v_2:=-(i+1)(1+\sqrt[4]{5})/2, \ {\rm and } \\v_3 & := & u^\tau ((i-1) - \sqrt[4]{5}(i+1))/2.\end{eqnarray*}$$v_1:=i(1 + t^2)/2(1/4*y^3 - 1/4*y^2 - 1/4*y + 1/4)*a + 1/4*y^3 + 1/4*y^2 - 1/4*y - 1/4\}$$
%\item $-1\not\in N_{\tilde{L}/L}(\cU_{\tilde{L}})$, and $[\cU_L:N_{\tilde{L}/L}(\cU_{\tilde{L}}]=2$;
\item %$u^2\in N_{\tilde{L}/F_0}(\cU_{\tilde{L}})$, and  
$[\cU_{F_0}^+:N_{\tilde{L}/F_0}(\cU_{\tilde{L}})]=1$.
\end{enumerate}
\end{lemma}

Let $L_1=\QQ(\sqrt[4]{5})$.
We fix an isomorphism $L \simeq L_1$ by sending 
$t \mapsto \sqrt[4]{5}$; this isomorphism sends 
$\eta$ to $\eta_1 := u(\sqrt[4]{5}+u)$.
Note that 
$L_1 = F_0(\sqrt{\Delta_{Q,P}}) = F_0(\sqrt{\Delta_{Q,R}})$.

Let $L_2 = F_0(i\sqrt[4]{5})$.
We fix an isomorphism $L \simeq L_2$ by sending 
$t \mapsto i \sqrt[4]{5}$; this isomorphism sends  
$\eta$ to $\eta_2:= u^\tau (i\sqrt[4]{5}+u^\tau)$.
%SAVE: There are two natural embeddings of $F_0$ into $L_2$, one fixes $\sqrt{5}$, which is the one used here; the other is the one compatible with $F_0\subset L \simeq L_2$, which is not the one used here.

%SAVE
%\[N_{\tilde{L}/L_1}(v_3)  & = & (u^\tau)^2 (1/4)((-1-\sqrt[4]{5})^2 - (i-\sqrt[4]{5}i)^2).\]

\subsection{Adjusting by units}

We use multiplication by a unit $\eta_1$ of $L_1$ to switch between 
points having complex multiplication 
by the maximal and non-maximal order.

Define a linear transformation 
$\psi_{QP}: F_0^2 \to L_1$ 
by $(x_1,d_1) \mapsto \sigma_{QP} := x_1 + \sqrt[4]{5}d_1$.
Let $[\times \eta_1]: L_1 \to L_1$ be multiplication by $\eta_1$.
Let $F_{QP}:F_0^2  \to F_0^2$ denote the composition
$F_{QP}:= \psi_{QP}^{-1} \circ [\times \eta_1] \circ \psi_{QP}$.

\begin{lemma} \label{lintransQP}
The composition 
$F_{QP}:F_0^2 \to F_0^2$ is given by $(x_1,d_1) \mapsto (\underline{x}_1,\underline{d}_1)$, where
\begin{equation} \label{Edefxunderdunder}
\underline{x}_1:= u (u x_1 + \sqrt{5} d_1) 
\text{ and } \underline{d}_1:=u (x_1 + u d_1).
\end{equation}
Thus, 
\begin{equation} \label{EQPrationalmap}
F_{QP}\left(\frac{x_1}{d_1}\right)= \frac{\underline{x}_1}{\underline{d}_1}
= \frac{u (\frac{x_1}{d_1}) + \sqrt{5}}{(\frac{x_1}{d_1}) + u}. 
\end{equation}

\begin{enumerate}
\item Also $(x_1,d_1) \in \cO_{F_0}^2$ if and only if $(\underline{x}_1,\underline{d}_1)\in \cO_{F_0}^2$.
\item For $(x_1,d_1) \in \cO_{F_0}^2$: 
if $(x_1,d_1) \equiv (0,u) \bmod 2\cO_{F_0}$, 
then $(\underline{x}_1,\underline{d}_1) \equiv (u^2,1) \bmod 2\cO_{F_0}$; 
and 
if $(x_1,d_1) \equiv (u^2,1) \bmod 2\cO_{F_0}$, 
then $(\underline{x}_1,\underline{d}_1) \equiv (0,u) \bmod 2\cO_{F_0}$.
\end{enumerate}
\end{lemma}

\begin{proof}
    The statement about $\underline{x}_1,\underline{d}_1$ in 
    \eqref{Edefxunderdunder} follows from this computation:
    \begin{eqnarray*}
\eta_1 \sigma_{QP} & = &
u (u + \sqrt[4]{5})(x_1 + \sqrt[4]{5} d_1) \\
& = & u(u x_1 + \sqrt{5} d_1) + 
u (x_1 + u d_1) \sqrt[4]{5}.
\end{eqnarray*}
\begin{enumerate}
\item The rational function $F_{QP}$ is given by following matrix:
\[F_{QP} = \begin{bmatrix} u & \sqrt{5} \\
1 & u \end{bmatrix}.\]
The result follows since $F_{QP}$ has integral entries and unit determinant $(u^\tau)^2$.

\item We omit this proof.
\end{enumerate}
\end{proof}

One can view $F_{QP}$ as a rational linear map on $t\in [-\sqrt[4]{5}, \sqrt[4]{5}]$
which fixes the endpoints.

%SAVE \sqrt[4]{5} \sim 1.495$.
Via Proposition~\ref{PfixedpointQP}, we identify $t\in [-\sqrt[4]{5}, \sqrt[4]{5}]$
with the geodesic $G_{QP}$.

\begin{lemma} \label{Lhypmidpoint}
The action of $\eta$ on the geodesic $G_{QP}$ (i.e., the action induced by $F_{QP}$ on $G_{QP}$) is a hyperbolic isometry.

The value $t=1$ yields the hyperbolic 
midpoint $\tilde{M}$ of the geodesic segment between
$\tilde{P}$ and $R_1$.
\end{lemma}

\begin{proof}
%SAVE Note that the explicit map $\rho$ is NOT a hyperbolic isometry as one is mixing up the real line with coordinate $t$ and the geodesic with coordinate $z$ -- the metric is on $z$ so we have a $GL_2(\RR)^+$ action bringing $z$ to $ic$; it is not too bad as the correct $c$ below is found by taking the square root, thus all the conclusions still hold}
We consider the hyperbolic isometry 
$\rho:G_{QP} \to \RR^{>0} i$ to a vertical half-ray which takes
$t=-\sqrt[4]{5}$ to $0 \cdot i$, $t=0$ to $1 \cdot i$, and $t=\sqrt[4]{5}$
to $\infty \cdot i$.
By \Cref{PfixedpointQP}, we have $z:=t+i \sqrt{\sqrt{5}-t^2}\in \frac{\omega\alpha}{\beta_0}G_{QP}$; we compute that $\rho: z \mapsto c\in \RR^{>0}$ is given by the matrix 
$\rho = -i \begin{bmatrix} 1 & \sqrt[4]{5} \\
-1 & \sqrt[4]{5} \end{bmatrix}$.
The action of $\eta$ maps $t$ to $t':=\frac{ut+\sqrt{5}}{t+u}$. 
Thus on ${\mathbb R}^{>0}$, we have
\[c:=\rho(t)=-i \frac{z+\sqrt[4]{5}}{-z+\sqrt[4]{5}}=\frac{\sqrt{\sqrt{5}-t^2}}{\sqrt[4]{5}-t}=\sqrt{\frac{\sqrt[4]{5}+t}{\sqrt[4]{5}-t}}\]
\[c':=\rho(t')=\sqrt{\frac{\sqrt[4]{5}+t'}{\sqrt[4]{5}-t'}}=\sqrt{\frac{\sqrt[4]{5}(t+u)+ut+\sqrt{5}}{\sqrt[4]{5}(t+u)-ut-\sqrt{5}}}=\sqrt{\frac{\sqrt[4]{5} + u}{-\sqrt[4]{5} +u}} \cdot c.\]

Let $d_h(z_1,z_2)$ denote the hyperbolic distance between 
two points $z_1,z_2 \in \bH$.
If $c_1, c_2 \in \mathbb{R}^{>0}$, 
then $d_h(c_1 \cdot i, c_2\cdot i) = |\mathrm{log}(c_2/c_1)|$.
Write $m=\sqrt{(\sqrt[4]{5} + u)(-\sqrt[4]{5} +u)^{-1}}$.
The first claim follows since $d_h(T(c_1 \cdot i), T(c_2 \cdot i)) =
|\mathrm{log}(m c_2/ m c_1)| = d_h(c_1\cdot i, c_2 \cdot i)$.

The hyperbolic midpoint of $c_1 \cdot i$ and $c_2 \cdot i$
is $\sqrt{c_1c_2} \cdot i$.
Note that $t(\tilde{P})=0$ and $t(R_1) = -u+3$.
Then $c(\tilde{P}) = \rho(0)=1$ and 
$c(R_1) = \sqrt{(-u+3 + \sqrt[4]{5})/(u-3 + \sqrt[4]{5})}$.
The hyperbolic midpoint $\tilde{M}$ has parameter $c_{\tilde{M}} = \sqrt{c(R_1)}$.
To show that $t_{\tilde{M}}=1$, it suffices to show that $c_{\tilde{M}} = \rho(1)$, 
or, equivalently, that $c(R_1) = (1+\sqrt[4]{5})/(-1 + \sqrt[4]{5})$, which is true.
%SAVE: 
%$c(R_1) = (-\sqrt{5} + 5 + 2 \sqrt[4]{5})(\sqrt{5} - 5 + 2 \sqrt[4]{5})$.
\end{proof}

We remark that $\tilde{M}$ is also the hyperbolic midpoint of the 
geodesic segment between $\eta^{-1}(Q_1)$ and $\tilde{Q}$.
%SAVE 
%$\eta^{-1}(Q_1)$ occurs at $t_1=u-1$ %about .618 
%and $\tilde{Q}$ occurs at $t_2 =2/u$. %about 1.236
%SAVE Thus $\rho(\eta^{-1}(Q_1))$ occurs at $c_1 \cdot i$ where 
%$c_1 =(1/3) (2 \sqrt[4]{5} + 2u +1)$
%and $\rho(\tilde{Q})$ occurs at $c_2 \cdot i$ where 
%$c_2 = (1/3) (4u+4) \sqrt[4]{5} + (1/3)(8u+3)$.
%SAVE The hyperbolic midpoint occurs at $c \cdot i$
%where $c = \sqrt{c_1 c_2}$.
%SAVE We compute that $c_1c_2=(2 + \sqrt{5})(2 \sqrt[4]{5} + 3)$.
%Then $\rho^{-1}(\sqrt{c_1c_2}) = 1$.

%SAVE At $\tilde{P}$, $(x,y)=(0,1)$ and $(x_1,d_1) = (0,1)$ and $F(\tilde{P}) = 
%-u^\tau \sqrt{5} \sim 1.382$.

%SAVE At $\tilde{Q}$, $(x,y) = (1,0)$ and $(x_1,d_1) = (2,u)$ and 
%$F(\tilde{Q})=(4+2\sqrt{5})/(3\sqrt{5} -1) \sim 1.484$.

Let $M := \pi(\tilde{M}) \in \Sh(\RR)$.
Let $C_M$ be the curve represented by $M$.

\begin{proposition} \label{MisCM}
The Jacobian of the curve $C_M$ has complex multiplication by $\QQ(\sqrt{-2})$.
\end{proposition}

\begin{proof}
The value $t=1$ occurs when $x_1=d_1=1$.  
By \eqref{EdiagqQP}, $q_{Q,P}(x,y)= - u(1-\sqrt{5}) = 2$.
The result follows by the same ideas as for 
Proposition~\ref{prop:QuadraticFormToCMpoints}, with the odd norm 
requirement in Assumption~\ref{def:good} being unnecessary.
\end{proof}

We divide $\overset{\frown}{PQR}$ into $\overset{\frown}{PM}$ and 
$\overset{\frown}{MR}$.

\begin{lemma}
Suppose $z_1, z_2 \in G_{QP}$ and $\eta(z_1) = z_2$.  
Then $\pi(z_1)$ is in $\overset{\frown}{PM}$ if and only 
if $\pi(z_2)$ is in $\overset{\frown}{MR}$.
\end{lemma}

\begin{proof}
Using \eqref{EQPrationalmap}, 
we compute that $\eta(\tilde{P})=R_1$ and $\eta(R_1)=P_1$. 
By Lemma~\ref{Lhypmidpoint}, $\tilde{M}$ is the hyperbolic midpoint.
Thus $\eta$ exchanges points in $\pi^{-1}(\overset{\frown}{PM})$
and points in $\pi^{-1}(\overset{\frown}{MR})$.
\end{proof}

\subsection{Quadratic forms as norms}

\begin{remark}
In \cite{zemkova}, Zemkova studies quadratic forms over a totally real number field $K$ with narrow class number 1.
Using an oriented relative class group, 
she gives necessary and sufficient conditions on an irreducible element $\lambda\in \mathcal{O}_{K}$ to be representable by a given quadratic form of discriminant $d$ in terms of the behavior of $\lambda$ in the extension $L=K(\sqrt{d})/K$. 
For the quadratic forms $q_{Q,R}$ and $q_{Q,P}$ in 
Corollary~\ref{Lquadm5sec5},   
the number field $L$ has class number $1$.
This allows us to provide an explicit description of the representability of these quadratic forms using norms.
\end{remark}

From Lemma~\ref{Lquadm5sec5} and equation \eqref{EdiagqQP}, recall that
\[q_{Q,P}(x,y)= 
3 x^2 + 2\sqrt{5} u^2 xy +  \sqrt{5} u y^2
=-u (x_1^2-\sqrt{5}d_1^2),\]
%with discriminant $\Delta_{P,Q}=16\sqrt{5} u^2$.
and
\[q_{Q,R}(x,y)= 
3 x^2  - 2 \sqrt{5} u xy  + \sqrt{5} u y^2.\]
%with discriminant $\Delta_{Q,R}= 4\sqrt{5}$.

\begin{proposition}\label{qasN}
Recall that $L_1=F_0(\sqrt[4]{5})$. 
Then:
\begin{enumerate}
\item $q_{Q,P}(x,y) = -u N_{L_1/F_0}(x_1 + \sqrt[4]{5} d_1)$; and 
\item $q_{Q,R}(x,y)=u^\tau N_{L_1/F_0}(x+\sqrt[4]{5}u(y-x))$.
\end{enumerate}
\end{proposition}

\begin{proof}
\begin{enumerate}
\item This is clear since 
$N_{L_1/F_0}(x_1 + \sqrt[4]{5} d_1)=x_1^2 - \sqrt{5}d_1^2$.
\item This is true because $uu^\tau=-1$ and
\[N_{L_1/F_0}(x+\sqrt[4]{5}u(y-x)) = (1-\sqrt{5}u^2)x^2 + 2 \sqrt{5}u^2 xy - \sqrt{5}u^2 y^2= -u q_{Q,R}(x,y). \qedhere\] 
\end{enumerate}
\end{proof}

\begin{proposition}\label{maxsol}
With notation as in \eqref{Edefxunderdunder}, let
$\underline{x} = \underline{x}_1/2$ 
and $\underline{y} = \underline{d}_1 - u\underline{x}_1/2$. Then
\begin{equation}
q_{Q,P}(x,y)=
q_{Q,P}(\underline{x}, \underline{y}).
%q_{Q,P}(x_1/2,d_1-ux_1/2)=
%q_{Q,P}(\underline{x}_1/2,\underline{d}_1 - u\underline{x}_1/2) 
\end{equation}
\end{proposition}

\begin{proof}
By Lemmas~\ref{lintransQP} and \ref{qasN},
\[q_{Q,P}(x,y) = u N_{L_2/F_{0}}(\sigma_{QP}) \text{ and }
q_{Q,P}(\underline{x},\underline{y}) =  u N_{L_2/F_{0}}(\eta_2 \sigma_{QP}).\]
By Lemma~\ref{LclassgroupQR}, 
$N_{L_2/F_0}(\eta_2)=1$.
So $N_{L_2/F_0}(\sigma_{QP})=N_{L_1/F_0}(\eta_2\sigma_{QP})$.
\end{proof}

\subsection{Complex multiplication and quadratic forms when $m=5$} \label{SCMquadformsm5}

We continue building on the material from Section~\ref{SCMquadforms}.

\begin{corollary} \label{LmaxQP} 
Under Assumption~\ref{def:good}, suppose also that
$\lambda \equiv -1 \bmod 4 {\mathcal O}_{F_0}$. 
Suppose $q_{Q,P}(x,y) = \lambda$
for some $x, y \in F_0$ such that 
$x\gamma_Q + y \gamma_P \in {\mathrm{GU}_2}({\mathcal O}_{F_0})$. 

Then $\Sh(\RR)$ contains a point with complex multiplication by ${\mathcal O}_E$ and another point with complex multiplication 
by $\mathcal{O}_F[\sqrt{-\lambda}]$.
\end{corollary}

\begin{proof}
Let $z$ be the point of $\bH$ fixed by $x\gamma_Q + y \gamma_P$.
By Lemma~\ref{solmod2QP}, the hypotheses
$q_{Q,P}(x,y) = \lambda$ and 
$\lambda \equiv -1 \bmod 4 {\mathcal O}_{F_0}$
imply that 
either (a) $(x_1,d_1) \equiv (0,u) \bmod 2{\mathcal O}_{F_0}$
or (b) $(x_1,d_1) \equiv (u^2,1) \bmod 2{\mathcal O}_{F_0}$.

Let $\underline{x}$ and $\underline{y}$ be as in Proposition~\ref{maxsol}.
Let $\underline{z}$ be the point of $\bH$ fixed 
by $\underline{x}\gamma_Q + \underline{y} \gamma_P$. 
By Proposition~\ref{maxsol}, 
$q_{Q,P}(\underline{x},\underline{y}) = \lambda$ as well.
By Lemmas~\ref{solmod2QP} and \ref{lintransQP}, 
$\underline{x}\gamma_Q + \underline{y} \gamma_P \in {\mathrm{GU}}_2({\mathcal O}_{F_0})$, and 
$(\underline{x}_1,\underline{d}_1)$ has case (a) exactly when
$(x_1, d_1)$ has case (b).

By Proposition~\ref{prop:QuadraticFormToCMpoints},
$z$ has CM by ${\mathcal O}_E$
when $\frac{1}{2}({\mathrm{Id}}+x_0\gamma_Q + y_0 \gamma_P) 
\in {\mathrm{GU}}_2({F_0})\cap M_2(\mathcal{O}_{F})$;
otherwise, it has CM by 
$\mathcal{O}_F[\sqrt{-\lambda}]$.
By Proposition~\ref{solmod2QP}(4), the former happens in case (a) and the latter in case (b).
Thus exactly one of $z$ and $\underline{z}$
has CM by ${\mathcal O}_E$
and the other has CM by 
$\mathcal{O}_F[\sqrt{-\lambda}]$.
\end{proof}

\subsection{Existence of real CM points on $M[11]$}

\begin{proposition}\label{qrep}
Let $\lambda$ be a totally positive irreducible element of $\mathcal{O}_{F_0}$, and $\tilde{L}/F_0$ as in Lemma~\ref{tildeL}.
\begin{enumerate}
\item	Then $\lambda$ is representable by both $q_{Q,P}(x,y)$ and $q^\tau_{Q,P}(x,y)$ if and only if the ideal $\langle\lambda\rangle$ of $\mathcal{O}_{F_0}$ splits completely in $\tilde{L}$  as a product of non-principal ideals. 
	
\item Assume $\lambda\equiv -1\bmod 4 \mathcal{O}_{F_0}$ and $\lambda$ is representable by both $q_{Q,P}(x,y)$ and $q^\tau_{Q,P}(x,y)$.  
Then there exist $x,y, x',y'\in F_0$ satisfying $q_{Q,P}(x, y)=\lambda$
and $q^\tau_{Q,P}(x', y')=\lambda$;
furthermore, with notation as in Lemma~\ref{solmod2QP}, 
%Note that we need to take the $\Gal(F_0/\QQ)$-conjugate for all the $x',y',d'$ conditions as the quadratic form for them is $q^\tau$; hence we deduce its property by applying $\tau$ to all properties of $q$
($x_1=2x$, $d_1 = y+ux$, $x_1' = 2x'$, and $d_1' = y' + u^\tau x'$)
then $x_1,d_1, x_1',d_1'\in\mathcal{O}_{F_0}$ and
$(x_1,d_1)\equiv (0,u) \bmod 2{\mathcal O}_{F_0}$, $(x_1',d_1') \equiv  (0,u^\tau) \bmod 2{\mathcal O}_{F_0}$.
\end{enumerate} 
\end{proposition}

\begin{proof}
\begin{enumerate}
\item Assume  $\lambda\in\mathcal{O}_{F_0}$ is representable by both $q_{Q,P}$ and $q^\tau_{Q,P}$. Equivalently, $\lambda, \lambda^\tau$ are both representable by $q_{Q,P}$. Hence,  by Lemma \ref{qasN}, $\langle\lambda\rangle$ and  $\langle\lambda^\tau\rangle$ both split in $L_1/F_0$. Equivalently, $\langle\lambda\rangle$ splits in both $L_1/F_0$ and $L_2/F_0$, where $L_2=F_0(i\sqrt{5})$. Since $\tilde{L}=L_1L_2$, we deduce that $\langle\lambda\rangle$ splits completely in $\tilde{L}$.    

If $\delta\in\tilde{L}$ is non-zero, then $N_{\tilde{L}/F_0}(\delta)$ is totally positive.  
By Lemma~\ref{tildeL}(2),
$\cU_{F_0}^+ =N_{\tilde{L}/F_0}(\cU_{\tilde{L}})$. 
We deduce that the ideal $\langle\lambda\rangle$ of $\mathcal{O}_{F_0}$ splits in $\tilde{L}$ as a product of non-principal ideals if and only if $\lambda\not\in N_{\tilde{L}/F_0}(\tilde{L})$.
By Lemma \ref{LclassgroupQR}(3),  $-u\not\in N_{L_1/F_0}(L_1)$; 
hence, by Lemma~\ref{qasN},  $\lambda \not\in N_{L_1/F_0}(L_1)$ and thus also $\lambda\not\in N_{\tilde{L}/F_0}(\tilde{L})$.
This completes the proof of the forward direction.

For the converse direction, note that $q_{Q,P}(1,0)=q^\tau_{Q,P}(1,0)=3$. 
Hence, $3 =-u N_{L_1/F_0}(2 + \sqrt[4]{5} u)$,
and the prime ideal $\langle3\rangle \subset \mathcal{O}_{F_0}$ splits completely in $\tilde{L}$ as a product of non-principal ideals. 

Assume $\langle\lambda\rangle$ splits completely in $\tilde{L}$ as a product of non-principal ideals. By Lemma \ref{tildeL}(1), the ideal class group of $\tilde{L}$ is isomorphic to $\mathbb{Z}/2\mathbb{Z}$, thus the ideal $\langle 3\lambda \rangle$ factors completely as a product of principal ideals. Since $3\lambda$ is totally positive, there exists $\omega \in\tilde{L}$ such that $N_{\tilde{L}/F_0}(\omega)=3\lambda$.
Let $\omega_1=N_{\tilde{L}/L_1}(\omega)$ and 
$\sigma_1=\omega_1 (1/u) (2+\sqrt[4]{5}u)^{-1}$. 
Then
\[-u N_{L_1/F_0}(\sigma_1)=-u (3 \lambda) (1/u^2) (u^\tau 3)^{-1}=\lambda.\]
A similar construction holds for $q^\tau_{Q,P}$.
Thus  $\lambda\in\mathcal{O}_{F_0}$ is representable by both $q_{Q,P}$ and $q^\tau_{Q,P}$.

\item 
This follows immediately from Lemma \ref{maxsol} and \Cref{solmod2QP}.
\end{enumerate}
\end{proof}

\begin{proposition}\label{CMbyorderQP}
Let $\lambda\in\cO_{F_0}$ be a totally positive irreducible element of $\cO_{F_0}$,
with $\lambda\equiv -1\bmod 4 \mathcal{O}_{F_0}$ and $N_{F_0/\QQ}(\lambda)\equiv 4\bmod 5$.  
Assume that $\lambda$ is representable by $q_{Q,P}(x,y)$.

Then, there are exactly two points in $\Sh(\RR)$ with complex multiplication by $\cO_F[\sqrt{-\lambda]}$, and they are both on the arch $\overset{\frown}{PQR}$.
\end{proposition}

\begin{proof}
The hypotheses imply that $\lambda$ satisfies Assumption~\ref{def:good}. 
By Corollary~\ref{LmaxQP}, 
$\Sh(\RR)$ contains a point $X$ with complex multiplication by ${\mathcal O}_E$ and another point $Y$ with complex multiplication 
by $\mathcal{O}_F[\sqrt{-\lambda}]$.
By Remarks~\ref{CMonlyif} and \ref{notPR},  
$X, Y \in \overset{{\frown}}{PQR}$.
By Theorem~\ref{Tuniquecong} and Remark~\ref{notmaxNT}, there are at most two points in $\Sh(\RR)$ with complex multiplication by $\cO_F[\sqrt{-\lambda}]$, so these must be $X$ and $Y$. Furthermore, the action of $\eta$ exchanges $\pi^{-1}(X)$ and 
$\pi^{-1}(Y)$ in $G_{QP}$.
\end{proof}

\section{Equidistribution of real CM points}  \label{Sequidistribution}

The main result of this section is Theorem~\ref{thm:infintely-many-lambda}
which implies that the set of points on the arch 
$\overset{\frown}{PQR} \subset \Sh(\mathbb{R})$
which have complex multiplication by the ring $\mathcal{O}_F[\sqrt{-\lambda}]$, 
for some totally positive irreducible element $\lambda\in\mathcal{O}_{F_0}$, is dense with respect to the Euclidean topology.

\subsection{Archimedean Character}\label{arch_sec}
In the following, for $K$ a number field, we denote by $J_K$ the id\`eles of $K$.
This is a locally compact topological group equipped with the restricted product topology.
Let $J_K^0$ denote the id\`eles of norm $1$.  Thus $J^0_K\supset K^*$.
We denote by $J_K^\infty$ the subgroup of $J_K$ consisting of those 
id\`eles having components which are units at the finite primes, 
and $1$ at the infinite places. 

Let $L=\QQ[t]/\langle t^4-5 \rangle$ be as in Lemma \ref{LclassgroupQR}.
Following \cite[XV \textsection 5 Example 3]{Lang},
we construct a homomorphism $ \psi:J_{L}\to \mathbb{S}^1 \sqcup \mathbb{S}^1$. 

\begin{notation}
The field $L$ has two real embeddings and one pair of complex embeddings.
We write 
\begin{equation} \label{ELinfinity*}
L^*_{ \infty} \simeq \RR^* \times \RR^* \times \CC^*,
\end{equation}
where the first real embedding is $t\mapsto+\sqrt[4]{5}$,
the second real embedding
is $t\mapsto-\sqrt[4]{5}$, and the complex embedding is $t\mapsto \pm i\sqrt[4]{5}$ (and we pick $t\mapsto i\sqrt[4]{5}$ for the isomorphism).
In particular, we embed the field $L$ in $L^*_{ \infty}$ via the diagonal.

Let $\tilde{L}$ be the splitting field of $t^4-5$ as in Lemma \ref{tildeL}. We identify $\tilde{L}=\QQ(\sqrt[4]{5}, i)$.
We write 
\begin{equation} \label{EtildeLinfinity*}
    \tilde{L}_\infty^*\simeq \CC^*\times \CC^* \times \CC^*\times \CC^*,
\end{equation} 
where the first embedding is $i \mapsto i$ and $\sqrt[4]{5}\mapsto \sqrt[4]{5} $, the second embedding is $i\mapsto i$ and $\sqrt[4]{5}\mapsto -\sqrt[4]{5}$, the third embedding is $i\mapsto i$ and $\sqrt[4]{5}\mapsto i\sqrt[4]{5}$, and the last embedding is $i\mapsto i$ and $\sqrt[4]{5}\mapsto -i\sqrt[4]{5}$. 

We identify $L_1=\QQ(\sqrt[4]{5})$ and $L_2=\QQ(i\sqrt[4]{5})$ as two subfields of $\tilde{L}$ both isomorphic to $L$. We fix isomorphisms $L \simeq L_1$ (resp.\ $L \simeq L_2$) by setting $t\mapsto \sqrt[4]{5}$ (resp.\ $t \mapsto i\sqrt[4]{5}$). 
We implicitly use these isomorphisms to identify $L$ with $L_1$ and $L_2$.
\end{notation}

Consider the homomorphism 
\begin{equation} 
\varphi:L_{\infty}^*\to \RR^*, \, (d_1,x_1,z)\mapsto x_1/d_1.
\end{equation}

Note that $\varphi(-1)=1$ and $\varphi(u)=1$.
Let $\varepsilon=\varphi (\eta)$, 
with $\eta$ as defined in Lemma~\ref{LclassgroupQR}.
Then $\varepsilon \in \RR^+$, 
%SAVE $\varepsilon \sim .0394
and we identify $\mathbb{S}^1 \sqcup \mathbb{S}^1\simeq\RR^*/\varepsilon^\mathbb{Z}$, 
as topological spaces with the Euclidean topology.  
The homomorphism $\varphi$ induces a surjective homomorphism \begin{equation}\label{phi2} \phi:L_{\infty}^*/\U_{L}\to \mathbb{S}^1 \sqcup \mathbb{S}^1.\end{equation}

Because $L$ has class number 1, the natural injection $L_{\infty}^* \hookrightarrow J_{L}$ induces a canonical isomorphism $j:L_{\infty}^*/\U_{L}\simeq J_{L}/L^*J_{L}^{\infty}$.
Let $\pi_{\infty}:J_{L}\to J_{L}/L^*J_{L}^{\infty}$ denote the natural projection.

\begin{definition}
Define $\psi:J_{L}\to \mathbb{S}^1 \sqcup \mathbb{S}^1$ as 
\begin{equation}
\label{psi}
\psi=\phi\circ j^{-1}\circ \pi_{\infty}.
\end{equation}

\begin{center}
\begin{tikzcd}
\psi: &    J_L \arrow[r, two heads ,"\pi_{\infty}" '] & J_L/L^*J_L^{\infty} \arrow[r, "j^{-1}"', "\simeq"] & L^*_\infty /\U_L \arrow [r, two heads ,"\phi" '] & \mathbb{S}^1 \sqcup \mathbb{S}^1.
\end{tikzcd}
\end{center}

\end{definition}

Define $\Psi: J_{\tilde{L}} \to (\mathbb{S}^1\sqcup \mathbb{S}^1)\times (\mathbb{S}^1\sqcup \mathbb{S}^1)$ as
\begin{equation}\label{Psi}
\Psi= \left(\psi\times \psi\right)\circ ({ N}_{\tilde{L}/L_1}\times { N}_{\tilde{L}/L_2}): J_{\tilde{L}} \to J_{L}\times J_{L}\to (\mathbb{S}^1\sqcup \mathbb{S}^1)\times  (\mathbb{S}^1\sqcup \mathbb{S}^1).
\end{equation}

\begin{lemma}\label{arch}
\begin{enumerate}
    \item 
Let $\psi: J_{L}\to \mathbb{S}^1\sqcup \mathbb{S}^1$ be as in \eqref{psi}. Then $\psi$ is a continuous homomorphism such that $\psi(J_{L}^0)=\mathbb{S}^1\sqcup \mathbb{S}^1$ and $\mathrm{Ker}(\psi)$ contains $L^*$. 

\item 
Let $\Psi: J_{\tilde{L}} \to (\mathbb{S}^1\sqcup \mathbb{S}^1)\times (\mathbb{S}^1\sqcup \mathbb{S}^1)$ be as in \eqref{Psi}. 
Then $\Psi$ is a continuous homomorphism such that $\Psi(J_{\tilde{L}})=\Psi(J_{\tilde{L}}^0)=\mathbb{S}^1\times \mathbb{S}^1$ (here both $\mathbb{S}^1$ correspond to $\RR^+$ in the identification $\mathbb{S}^1\sqcup \mathbb{S}^1\simeq \RR^*/\varepsilon^\ZZ$) and 
$\mathrm{Ker}(\Psi)$ contains $\tilde{L}^*$.
\end{enumerate}
\end{lemma}

\begin{proof}
All the statements are clear from the construction except for the equality
$\Psi(J_{\tilde{L}})=\Psi(J_{\tilde{L}}^0)=\mathbb{S}^1\times \mathbb{S}^1$.
To verify it, 
it suffices to observe the surjectivity of the map   $${\Phi}=\left( \varphi\times \varphi\right)\circ ({N}_{\tilde{L}/L_1}\times { N}_{\tilde{L}/L_2}):\tilde{L}^*_{\infty}\to L^*_{\infty}\times L^*_{\infty}\to \mathbb{S}^1\times \mathbb{S}^1.$$
Let $h$ denote complex conjugation on $\CC$. 
Under the identifications in \eqref{ELinfinity*} and \eqref{EtildeLinfinity*},
\[
\tilde{L}^*_{\infty}\simeq \CC^*\times \CC^*\times \CC^*\times \CC^*
\text{ and }
L^*_{\infty}\times L^*_{\infty} \simeq (\RR^*\times\RR^*\times \CC^*)\times (\RR^*\times\RR^*\times\CC^*),\] 
the map $({N}_{\tilde{L}/L_1}\times {N}_{\tilde{L}/L_2})$
is given, for $a,b,c,d\in \CC^*$, by
\[({N}_{\tilde{L}/L_1}\times {N}_{\tilde{L}/L_2}) (a,b,c,d)
=((a h({a}), bh({b}), ch({d})), (dh({d}),ch({c}),ah({b}))).\] Hence 
\[{\Phi}(a,b,c,d)= \left(\frac{bh({b})}{ah({a})}, \frac{ch({c})}{d h({d})}\right).\qedhere \]
\end{proof}

From now on, we use $\Psi$ to denote the map $J_{\tilde{L}} \to \mathbb{S}^1\times \mathbb{S}^1$.

\subsection{Non-archimeadean Character}

We consider some quadratic extensions of $F_0$, namely 
$K_1=F$, $K_2=F_0(\sqrt{u})$, $K_3=F_0(i)$, $K_4=L_1$, 
and
$K_\pi=F_0(\sqrt{\pi})$ for a totally positive irreducible element
$\pi\in\cO_{F_0}$.
Let $C_2=\{1,-1\}$ be the cyclic group of order two.
We denote Artin's reciprocity map for $K_?$ by
$r_?:J_{F_0}\to \mathrm{Gal}(K_?/F_0)\simeq C_2$. 

For $S$ a finite set of totally positive irreducible elements of $\cO_{F_0}$,  
we define
\begin{equation} \label{ChiS}
\Chi_{S}=\chi \times\chi_{S}:J_{F_0}\to C_2^4\times C_2^{|{S}|}\simeq C_2^{|{S}|+4},
\end{equation}
where $\chi=r_1\times r_2 \times r_3 \times r_4$ and 
$\chi_{S}=  \prod_{s\in {S}} r_s.$

\begin{lemma}\label{narch}
Let $\lambda \in \mathcal{O}_{F_0}$ be a totally positive irreducible element, 
such that $\lambda\neq 2,  u\sqrt{5}$. 
Assume $\lambda,  u\sqrt{5} \not\in {S}$, and ${S}={S}^\tau$.
Denote by $\kappa_0(\lambda)\in J_{F_0}$ the element with entry $\lambda$ at the place $\langle \lambda\rangle$, and $1$ everywhere else. Then,

\begin{enumerate}
\item $\chi(\kappa_0(\lambda))=(-1,-1,1,1)$ if and only if 
$\lambda \equiv -1, -(3\pm \sqrt{5})/2 \bmod 4 \cO_{F_0}$, 
 $N_{F_0/\QQ}(\lambda)\equiv 4\bmod5$, and
$\lambda$ is completely split in $\tilde{L}$.

\item  $\Chi_{S}(\kappa_0(\lambda))=((-1,-1,1,1), \mathbb{1} )$ if and only if $\chi(\kappa_0(\lambda))=(-1,-1,1,1)$ and all $s\in {S}$ split in both $F_0(\sqrt{-\lambda})/F_0$ and $F_0(\sqrt{-\lambda^\tau})/F_0$.
\end{enumerate}
\end{lemma}

\begin{proof}
\begin{enumerate}
\item 
By Lemma \ref{quadrec5},
the first three entries of $\chi(\kappa_0(\lambda))$ are $(-1,-1,1)$ 
if and only if 
$\lambda \equiv -1, -(3\pm \sqrt{5})/2 \bmod 4 \cO_{F_0}$ and
$N_{F_0/\QQ}(\lambda)\equiv 4\bmod5$. 
By definition, the last two entries of $\chi(\kappa_0(\lambda))$ are $(1,1)$ if and only if $\lambda$ is completely split in $\tilde{L}/F_0$.

\item By definition, $\Chi_S(\kappa_0(\lambda))=((-1,-1,1,1), \mathbb{1} )$ if and only if $ \phi(\kappa_0(\lambda))=(-1,-1,1,1)$ and  
the prime $\langle \lambda \rangle$ splits in $F_0(\sqrt{s})/F_0$, 
for each $s\in {S}$. 

Since $\lambda \equiv -1, -(3\pm \sqrt{5})/2 \bmod 4 \cO_{F_0}$,
by Lemma \ref{QR}, for each $s\in{S}$, the ideal $\langle \lambda \rangle$ splits in $F_0(\sqrt{s})/F_0$ if and only if 
 $s$ splits in the extension $F_0(\sqrt{-\lambda})/F_0$. In particular, $s^\tau\in S$ splits in  $F_0(\sqrt{-\lambda})/F_0$ and thus $s$ splits in $F_0(\sqrt{-\lambda^\tau})/F_0$. 
 
Note that for 
$\lambda \equiv -1, -(3\pm \sqrt{5})/2 \bmod 4 \cO_{F_0}$,
the prime $2$ is unramified in $F_0(\sqrt{-\lambda})/F_0$; direct computations show that $2$ can be either split or inert. \qedhere
\end{enumerate}
\end{proof}

\subsection{Equidistribution theorem}

We deduce the following from \cite[XV \textsection 5, Theorem 6]{Lang}.

\begin{theorem}\label{thm:infintely-many-lambda}
	Given any finite set $\mathcal{S}$ of prime ideals of $F_0$ , with $\langle u\sqrt{5}\rangle \notin \mathcal{S}$ and $\mathcal{S}^\tau=\mathcal{S}$,
	there exists a set $\Lambda$ of totally positive irreducible elements of $\mathcal{O}_{F_0}$ such that
	\begin{enumerate}
	\item for any $\lambda \in \Lambda$:
	\begin{itemize}
		\item $\lambda \ne \lambda^\tau$;
	 \item $N_{F_0/\QQ} (\lambda) \equiv 4 \bmod 5$ and $\lambda\equiv -1\bmod 4 \cO_{F_0}$;
		\item  $\lambda$ is representable by both the quadratic forms $q_{Q,P}$ and $q^\tau_{Q,P}$ 
				 \item each prime ideal $s \in \mathcal{S}$ splits in the 
                 extensions $F_0(\sqrt{-\lambda})/F_0$ and $F_0(\sqrt{-\lambda^\tau})/F_0$;
		 \end{itemize}
		 \item 
        For any $a,b \in (-\sqrt[4]{5}, \sqrt[4]{5}) \subset \mathbb{R}$, 
        and $\epsilon>0$:
		there exists $\lambda_0 \in \Lambda$ and $d_1,x_1,d_2,x_2 \in \mathcal{O}_{F_0}$ such that 
		 \begin{itemize}
		  \item with notation from \eqref{EdiagqQP},\footnote{
          With a slight abuse of notation, we write $q_{Q,P}(d_1,x_1)$
          rather than $q_{Q,P}(x,y)$ with 
          $x=x_1/2$ and $y=d_1 - ux_1/2$.
          }
          $q_{Q,P}(d_1,x_1)=\lambda_0$, and  
          $q_{Q,P}(d_2,x_2)=\lambda_0^{\tau}$; and 
		
		 \item $|x_1/d_1 - a| \le \epsilon$, and $|x_2/d_2 - b| \le \epsilon$. 
		
		  \end{itemize}
	\end{enumerate}
\end{theorem}

\begin{proof}
Recall the maps 
$\Psi:J_{\tilde{L}} \to \mathbb{S}^1\times \mathbb{S}^1$ given in \eqref{Psi}
and 
$\Chi_\mathcal{S}: J_{F_0} \to C_2^{|\mathcal{S}|+4}$ given in \eqref{ChiS};  
let $\chi_0:J_{\tilde{L}}\to \mathrm{cl}_{\tilde{L}} \simeq C_2$ be the 
class character of $\tilde{L}$ (see Lemma \ref{tildeL}(1)).

Consider the continuous homomorphism
$$\Theta=\Psi\times  \nabla
:J_{\tilde{L}} \to \mathbb{S}^1\times \mathbb{S}^1\times C_2^{|\mathcal{S}|+5},$$
where $\nabla: J_{\tilde{L}} \to C_2^{|\mathcal{S}|+5}$ is defined as 
 $\nabla=\chi_0 \times (\Chi_\mathcal{S} \circ  N_{\tilde{L}/F_0} )$.
 Let $G=\mathrm{Im}(\Theta)$ and regard $\Theta: J_{\tilde{L}}\to G$.  
By Lemma \ref{arch}(2), we have $\Theta(J^0_{\tilde{L}})=G$ and $\Theta(\tilde{L}^*)=1$.   

We identify the set $P_{\tilde{L}}$ of primes of $\tilde{L}$ with a subset of  $J_{\tilde{L}}$, by choosing a map
$\kappa:P_{\tilde{L}}\to J_{\tilde{L}}$. 
For $\mathfrak{q}$ a prime of ${\tilde{L}}$, 
define $\kappa(\mathfrak{q})\in J_{\tilde{L}}$ 
such that 
its entry at the place $\mathfrak{q}$ is a uniformizer of ${\tilde{L}}_{\mathfrak{q}}$, and its entry is $1$ everywhere else. 
By \cite[XV \textsection 5, Theorem 6]{Lang}, we deduce that the prime ideals  of  ${\tilde{L}}$ are equidistributed, with respect to the map $\theta=\Theta\circ\kappa:P_{\tilde{L}}\to G$. 

Suppose a prime $\mathfrak{q}$ of ${\tilde{L}}$ satisfies 
$\nabla(\kappa(\mathfrak{q}))=(-1, (-1,-1,1,1), \mathbb{1})$
Consider the totally positive irreducible element $\lambda\in \cO_{F_0}$ given by $\langle \lambda\rangle=N_{\tilde{L}/F_0}(\mathfrak{q})$ 
(which is unique up to multiplication by the square of a unit in $\cO_{F_0}$).
By Proposition~\ref{qrep}, combined with Lemma~\ref{narch},
$\lambda \equiv -1, -(3\pm \sqrt{5})/2 \bmod 4 \cO_{F_0}$, 
$N_{F_0/\QQ} (\lambda) \equiv 4 \bmod 5$, 
and $\lambda$ is representable by both quadratic forms $q_{Q,P}$ and $q^\tau_{Q,P}$. 
If we multiply $\lambda$ by the square of a unit, all the discussions above still hold; thus we may choose $\lambda \equiv -1 \bmod 4 \cO_{F_0}$. Furthermore, $\lambda\neq \lambda^\tau$ if and only if the prime $\langle\lambda\rangle$ has degree 1; that is,
the condition $\lambda\neq \lambda^\tau$ removes a set of primes of density zero.

To deduce the statement from the $\theta$-equidistribution of the primes of ${\tilde{L}}$, it suffices to: 
\begin{enumerate}
\item verify the inclusion $G\supset  \mathbb{S}^1\times \mathbb{S}^1\times
\langle (-1, (-1,-1,1,1), \mathbb{1}) \rangle$;
%\langle(-1, -1,-1,1,1, \dots ,1)\rangle$, 
\item show that equidistribution of $\Psi(\kappa(\mathfrak{q}))\in  \mathbb{S}^1\times \mathbb{S}^1$,  
for $\mathfrak{q}\in P_{\tilde{L}}$ satisfying 
$\nabla(\kappa(\mathfrak{q}))=(-1, (-1,-1,1,1), \mathbb{1})$,
%(-1, -1,-1,1,1, \dots ,1), 
implies the density of   
\[(x_1/d_1, x_2/d_2)\in (-\sqrt[4]{5}, \sqrt[4]{5}) \times 
(-\sqrt[4]{5}, \sqrt[4]{5}),\]
for $d_1,x_1,d_2,x_2 \in \mathcal{O}_{F_0}$ satisfying		
$q_{Q,P}(d_1,x_1)=\lambda$,  $q_{Q,P}(d_2,x_2)=\lambda^{\tau}$, 
for $\langle \lambda \rangle =N_{\tilde{L}/F_0}(\mathfrak{q})$.
 \end{enumerate}
 
{\bf Proof of claim (1):} Let $p_1:G\to  \mathbb{S}^1\times \mathbb{S}^1$ denote the projection such that $p_1\circ \Theta=\Psi$, and
 $p_2:G\to C_2^5$ denote the projection such that $p_2\circ\Theta=\chi_0\times (\chi \circ N_{\tilde{L}/F_0})$.
 
 By Lemma~\ref{narch}, 
 the equality $q_{Q,P}(1,0)=3$ implies that 
 $(\chi_0\times (\chi \circ N_{\tilde{L}/F_0}))(\kappa(\mathfrak{q}_3))=
 (-1, (-1,-1,1,1))\in p_2(G)$,
 %(-1,-1,-1,1,1)\in p_2(G)$, 
 for $\mathfrak{q}_3$ a non-principal prime ideal of $\tilde{L}$ above 3.  
 To conclude, note that the extension $F_0(\sqrt{s})/F_0$, for $s\in\mathcal{S}-\{2\}$, is disjoint from $F$, $F_0(\sqrt{u})$, and the Hilbert class field extension of $\tilde{L}$, since the former is ramified at $s$ and the latter are only ramified at $2$ and $\sqrt{5}$. 
 The prime $\langle 2\rangle$ is principal 
 in $\mathcal{O}_{\tilde{L}}$, and direct computations show that it
 can be either split or inert in $F_0(\sqrt{-\lambda})/F_0$. Therefore 
 $\langle (-1, (-1,-1,1,1)) \rangle \subset p_2(G)$.
 %$\langle(-1, -1,-1,1,1, \dots ,1)\rangle\subset p_2(G)$.

For any $\xi \in \langle (-1, (-1,-1,1,1), \mathbb{1}) \rangle$, 
%\langle(-1, -1,-1,1,1, \dots ,1)\rangle$, 
$\nabla^{-1}(\xi)\cap J^0_{\tilde{L}}$ is a finite index subgroup of $J^0_{\tilde{L}}$.
Lemma~\ref{arch}(2) implies that $p_1$ is surjective. 
Claim (1) follows since any finite index subgroup of $\mathbb{S}^1\times \mathbb{S}^1$ is itself.

{\bf Proof of claim (2):}
Suppose  $\mathfrak{q}\in P_{\tilde{L}}$.
Since $L$ has class number $1$, 
there exists an (irreducible) element $\sigma_i \in\mathcal{O}_{L_i}$, 
satisfying $N_{\tilde{L}/L_i}(\mathfrak{q})=\langle \sigma_i\rangle$, for $i=1,2$.
Let $\iota$ denote the natural map $L^*\hookrightarrow L_\infty^*$.
Let $\gamma$ denote the nontrivial element in $\Gal(L/F_0)$. 
We view elements in $L$ as elements in $\RR$ 
via the embedding given by $t\mapsto \sqrt[4]{5}$. 
Recall that we fixed isomorphisms $L\simeq L_i$ for $i=1,2$.

Using the character $\phi :L_{\infty}^*/\mathcal{U}_{L_i}\to \mathbb{S}^{1}$ 
given in \eqref{phi2},
$\phi(\iota(\sigma_i^{-1})) = \sigma_i/\sigma_i^\gamma \in \RR^*/\varepsilon^\ZZ$, for $i=1,2$. 
Using the definition of $\Psi$ and the fact that $\psi(L^*)=1$, we compute 
\begin{align*}
\Psi(\kappa(\mathfrak{q})) &=
\left(\psi({ N}_{\tilde{L}/L_1}(\kappa(\mathfrak{q})),  \psi ({ N}_{\tilde{L}/L_2}(\kappa(\mathfrak{q}))\right)= (\phi(\iota(\sigma_1^{-1})),\phi(\iota(\sigma_2^{-1})))\\
&=(\sigma_1/\sigma_1^\gamma, \sigma_2/ \sigma_2^\gamma )\in \RR^*/\varepsilon^\ZZ \times \RR^*/\varepsilon^\ZZ.\end{align*}

For $i=1,2$, write $\sigma_i=x_i+\sqrt[4]{5}d_i\in L_i$, 
with $d_1,x_1,d_2,x_2\in F_0$. 
We deduce
$$\Psi(\kappa(\mathfrak{q}))= 
\left(\frac{x_1+\sqrt[4]{5}d_1}{x_1-\sqrt[4]{5}d_1} , \frac{x_2+\sqrt[4]{5}d_2}{ x_2-\sqrt[4]{5}d_2}   \right)\in \RR^*/\varepsilon^\ZZ \times \RR^*/\varepsilon^\ZZ.$$

Let $\lambda\in\mathcal{O}_{F_0}$ be a totally positive (irreducible) element satisfying $\langle \lambda \rangle =N_{\tilde{L}/F_0}(\mathfrak{q})$.  
By  Lemma \ref{narch} and Proposition \ref{qrep}, $\mathfrak{q}\in P_{\tilde{L}}$ satisfies $\nabla(\kappa(\mathfrak{q}))=
(-1, (-1,-1,1,1), \mathbb{1})$
%(-1, -1,-1,1,1, \dots ,1)$ 
if and only if $\lambda$ satisfies the conditions in assumption (1) in the statement.

Assume $\nabla(\kappa(\mathfrak{q}))=(-1, (-1,-1,1,1), \mathbb{1})$.
%(-1, -1,-1,1,1, \dots ,1)$. 
Then, by Lemma~\ref{qasN}, there exists a  totally positive unit $v\in\U^+_{F_0}$ such that $-u N_{L_1/F_0}(\sigma_1)=v\lambda$.
Since $\U_{F_0}^+ =\U_{F_0}^2$, after multiplying $\sigma_1$ by a suitable element in $\U_{F_0}$ (which does not affect the value $\varphi(\iota(\sigma^{-1}_1))\in \RR^*$), we have
$-u N_{L_1/F_0}(\sigma_1)=\lambda$. 
Similarly, we can adjust $\sigma_2$ so that
$-u N_{L_2/F_0}(\sigma_2)=\lambda$,
without changing $\varphi(\iota(\sigma_2^{-1}))\in \RR^*$.

That is, for $d_1, x_1, d_2, x_2\in \mathcal{O}_{F_0}$ satisfying $q_{Q,P}(d_1,x_1)=\lambda$ and  $q_{Q,P}(d_2,x_2)=\lambda^{\tau}$, the values $(\varphi(\iota(\sigma_1^{-1})), \varphi(\iota(\sigma_2^{-1})))$ are equidistributed in $ \RR^+\times \RR^+$ (because $\lambda$ is totally positive), and hence also 
the values
\[(x_1/d_1, x_2/d_2)\in (-\sqrt[4]{5}, \sqrt[4]{5}) \times 
(-\sqrt[4]{5}, \sqrt[4]{5}).\qedhere\]
\end{proof}

%%%%%%%%%%%%%SECTION 8

\section{Reduction modulo $5$}

On the $M[11]$ family, we consider the abelian varieties we constructed with complex multiplication.  We study their reduction modulo $5$.

\subsection{Reduction modulo 5 of curves} \label{Sreductionlehr}

By a result of Lehr, we can determine whether a curve in the $M[11]$ family has good reduction at $5$.
Here is some notation needed to state this.
Let $R$ be a complete discrete valuation ring with mixed characteristic $(0,5)$, let $\mathfrak{m}$ be the maximal idea of $R$, let $K=\mathrm{Frac}(R)$, and let $v$ be the valuation of $K$. 
Suppose $R$ is a $\cO_{F_0}$-algebra; with abuse of notation, we also denote $u\sqrt{5}\in R$ the image of  $u\sqrt{5}\in\cO_{F_0}$. 
\footnote{In \cite{lehrred2}, Lehr further assumes $\zeta_5, \sqrt[4]{-5}\in R$; in our setting, these conditions are satisfied if $R$ is a $\cO_F$-algebra; indeed note that in $\overline{\ZZ}_5$, we have $(\sqrt[4]{-5})=(\zeta_5^2-\zeta^3_5)$ and it is easy to check that these two numbers differ by an element in $\ZZ_5^\times$; since $R$ is $5$-adic, we have $\sqrt[4]{-5} \in R$.
%\rachel{Equation typo: $\QQ(\zeta_5)$ does not contain $\sqrt[4]{-5}$.}
Proposition \ref{Predlehr} holds with the weak assumption.}

Given $t \in K-\{0,1\}$, consider the cover $f: C_t \to {\mathbb P}^1_K$ given by $y^5 = x(x - 1)(x -t)$. 
Set \begin{equation}\label{Eqlittlej}
j_t=(u\sqrt{5})^{-5} \frac{(t^2-t+1)^3}{t^2 (t-1)^2}\in K,\end{equation} 
the Klein j-function from \eqref{Ejfunction} in Section~\ref{Sjfunction}, normalized at $5$. 

\begin{proposition} \label{Predlehr} (\cite[Theorem 2.1]{lehrred2}; \cite[Corollary 2]{lehrred1}, for $p=5$)
Suppose $t \in K -\{0,1\}$.
Then 

\begin{enumerate}
\item $C_t$ has potentially good reduction if  $v(j_t) \geq 0$;
\item $C_t\bmod \mathfrak{m}$ is geometrically isomorphic to $  C_P \bmod \mathfrak{m}$ if  $v(j_t) < 0$.
\end{enumerate}
\end{proposition}

\begin{proof}   
By \cite[Theorem 2.1]{lehrred2},  these are the only two 
possibilities for the special fiber of the stable model $C_t \bmod \mathfrak{m}$ (in this instance, case (2) of \cite[Theorem~2.1]{lehrred2} does not occur).\end{proof}

\subsection{Reduction modulo 5 of CM points}

We prove that the CM points we constructed on $M[11]$ have non-degenerate reduction modulo $5$.

Denote the Jacobian of $C_P$ (resp.\ $C_R$) as $A_P$ (resp.\ $A_R$). Let $\lambda \in {\mathcal O}_{F_0}$ be irreducible, totally positive, and relatively prime to $2 \sqrt{5}$. Let $E=F(\sqrt{-\lambda})$, and $(E,\Phi)$ be the CM type  defined in Section~\ref{CMtype}.

\begin{proposition}\label{CMmod5}
With the notations above,
let $A$ be a principally polarized CM abelian variety, of CM type  $(E,\Phi)$.  Let $L'$ be a field of definition for $A$ containing $F$, and let $v_5$ be a place of $L'$ with characteristic $5$.
Assume  $A$ has complex multiplication by $\cO_E$ or $\cO_F[\sqrt{-\lambda}]$.

Suppose $N_{F_0/\QQ}(\lambda) \equiv 4 \bmod 5$.
Then $A \bmod v_5$ and $A_P \bmod v_5$ (resp.\ $A \bmod v_5$ and $A_R \bmod v_5$) 
are not isomorphic as 
principally polarized abelian varieties over $\overline{\FF}_5$.
\end{proposition}

\begin{proof}
Recall that $A$ is simple by Lemma \ref{LsimpleCMtype}. 
An endomorphism of $A$ is called $\cO_F$-linear if it commutes with $\cO_F$; we
denote by $\End_F{A}$ the geometric $\cO_F$-linear endomorphism ring of $A$. 
The endomorphism $\sqrt{- \lambda}\in \End_{F}(A)$ 
satisfies that $(\sqrt{-\lambda})^{\dagger} = -\sqrt{-\lambda}$, where $\dagger$ denotes the Rosati involution.

By Section \ref{subsec-special-points}, there is a principally polarized abelian surface $A_0$ with CM by $\cO_F$, such that 
$A_P \bmod v_5$ is geometrically isomorphic to $A_0^2 \bmod v_5$, with the product polarization; this isomorphism is compatible with the polarizations and $\cO_F$-action.\footnote{By Section \ref{subsec-special-points}, $A_P$ is isomorphic to $A_0^2$ as a polarized abelian variety, but with non-compatible $\cO_F$-action due to signature. To have a compatible $\cO_F$-action, we need to twist the $\cO_F$-action on the second copy of $A_0$ by a suitable element in $\Gal(F/\bQ)$. Note that $A_0 \bmod v_5$ is geometrically isomorphic, compatibly with the $\cO_F$-action, 
to the twisted one.}
%SAVE: indeed, $A_0$ is the Jacobian of the genus $2$ curve is given by $y^5-y=x^2$ with $\ZZ/5\ZZ$-action given by $y\mapsto y+a$, then isomorphism between twists by $b$ is given by $y\mapsto by, x\mapsto \sqrt{b} x$.
The geometric $\cO_F$-linear endomorphism ring of $A_P$ is 
$\End_{F}(A_P) \cong M_2(\cO_F)$.
The Rosati involution $\dagger$ acts via the composition of matrix transposition and complex conjugation on $F$.

Since $\End(A_0) \cong \cO_F$, the $\ell$-adic Tate module $T_\ell(A_0)$ is an $\cO_F$-module of rank $1$. Hence, after the reduction, 
the endomorphisms commuting with $F$ are $\End_{F}(A_{0}\bmod v_5) \cong  \cO_F$ and thus $\End_{F}(A_{P}\bmod v_5) \cong M_2(\cO_F)$.

Assume that $A \bmod v_5$ is isomorphic to $A_P \bmod v_5$ as polarized abelian varieties. 
Then $\sqrt{-\lambda} \in \End_F(A_0^2\bmod v_5)$.  Write
\[\sqrt{-\lambda} = M = \begin{pmatrix}  a & b \\ c & d \end{pmatrix} \in M_2( \cO_F).\] 
Thus $\lambda = \det M = ad - bc$ and $\mathrm{tr} M=0$, that is $d = -a$. Since $\sqrt{-\lambda} \in \End_F(A_0^2)$ anticommutes with $\dagger$, we deduce
that $a=-\bar{a}$ and $c=-\bar{b}$ (where $z\mapsto\bar{z}$ denotes complex conjugation on $F$).

Since $a\in \cO_F$ is totally imaginary, we can write $a$ as a 
$\ZZ$-linear combination of $\zeta_5 - \zeta_5^4$ and $\zeta_5^2 - \zeta_5^3$.
Hence $a \in v_5 \cap\mathcal{O}_F = \langle 1 - \zeta_5 \rangle$, and $a^2 \in v_5^2 \cap \mathcal{O}_{F_0}=\langle \sqrt{5}\rangle$.
It follows that $\lambda=-a^2 + N_{F/F_0} (b)\equiv N_{F/F_0} (b) \mod \sqrt{5}$, and hence $\lambda$ is a square modulo $\sqrt{5}$.  
(To see this, write $b = b_1 + b_2(\zeta_5-\zeta_5^{-1})$ with $b_1,b_2 \in \cO_{F_0}$.
Then $N_{F/F_0} (b) \equiv b_1^2 \bmod\sqrt{5}$.)
This is a contradiction since by assumption $\lambda \equiv \pm 2\bmod \sqrt{5}$. 
\end{proof}

We remark that by Proposition \ref{Predlehr} the reductions modulo $5$ of $C_R $ and $C_P$ are isomorphic; hence, $A_P\bmod v_5$ and $A_R\bmod v_5$ are isomorphic as principally polarized abelian varieties.

Combining Propositions \ref{Predlehr} and \ref{CMmod5}, we deduce the following result.

\begin{corollary}\label{Cnot5}
With notation as in Propositions \ref{Predlehr} and \ref{CMmod5}: Let $t\in K-\{0,1\}$ and $j_t\in K$ be as in \eqref{Eqlittlej}. Suppose $\mathrm{Jac}(C_t)$  has complex multiplication by $\cO_E$ or $\cO_F[\sqrt{-\lambda}]$, 
and $N_{F_0/\bQ}(\lambda)\equiv 4 \bmod 5$. 
Then $v(j_t)\geq 0$.
\end{corollary}

%%%%%%%%%%%%%SECTION 9

\section{CM cycles in M[11]}

\subsection{CM cycles in characteristic 0}

Over $\cO_{F_0}[1/5]$, consider the family of curves over $\bP^1$ given by the affine equation $C_t:\ y^5=x(x-1)(x-t)$, with $t\in \bP^1\setminus \{0,1,\infty\}$,
and the map $j:\bP^1\to \bP^1$, given by $t\mapsto j_t=(u\sqrt{5})^{-5}(t^2-t+1)^3/(t^2 (t-1)^2)$ (see \eqref{Eqlittlej}). 

Let $\Sh=\Sh(\mathcal{D})/\cO_F[1/5]$ denote  the  PEL type moduli space defined in Section \ref{M11asSh}. \footnote{The moduli space $\Sh$ is a proper Deligne--Mumford stack defined over $F$. By the theory of canonical integral models, $\Sh$ has a smooth canonical integral model over $\mathcal{O}_F[1/5]$, which is a proper Deligne--Mumford stack, and is given by the moduli interpretation away from $5$; with abuse of notation, we also denote it by $\Sh$.}  
Recall that $\Sh$ is connected by \cite{shimuratranscend}.
%SAVE: More precisely, as we work with no level structure, 
%the number of connected components of $\Sh$ is equal to the number of  isometry classes in the same genus (for this argument, see for instance, BHKRY-I Prop 2.1.1), which is $1$ in this case by work of Shimura, as  the field $F$ has class number $1$. 

By Lemma \ref{Lfomfod}, the map which associates the isomorphism class of the curve $C_t$ to $j_t$ defines an isomorphism between the coarse moduli space associated to $\Sh$ and $\bP^1_{\cO_F[1/5]}$. 
\footnote{Since $\Sh$ is a smooth Deligne--Mumford stack of relative dimension $1$ over $\cO_F[1/5]$, its coarse moduli space is also smooth. The isomorphism between the coarse moduli space associated to $\Sh$ and  $\bP^1$ over ${F}$ extends over $\mathcal{O}_F[1/5]$.)} 

In the following, we use this isomorphism to identify the coarse moduli space associated to $\Sh$ and the $j$-line $\bP^1_{\cO_F[1/5]}$. Note that the special points $Q,R,P$ from Section~\ref{subsec-special-points} map respectively to 
$j_Q:=0$, $j_R:=c:=\frac{27}{4} (u\sqrt{5})^{-5}$, and $j_P:=\infty$ in $\bP^1(F_0)$. 
 
\begin{notation}\label{lambda} Let $\lambda\in\mathcal{O}_{F_0}$ be a totally positive irreducible element,  satisfying $N_{F_0/\QQ}(\lambda)\equiv 4\bmod 5$, $\lambda\equiv -1\bmod 4 \cO_{F_0}$, 
and $\lambda\neq \lambda^\tau$.
Assume the ideal $\langle\lambda\rangle$ splits  completely as a product of non-principal ideals in  the splitting field $\tilde{L}$ of $x^4-5$ over $\QQ$ (see Lemma \ref{tildeL}). 

By assumption,  $\lambda$ is inert in $F/F_0$; 
we denote by $F_\lambda$ the completion of $F$ at $\lambda$ and by $\mathcal{O}_{F,\lambda}$ its ring of integers. 

We write $\bF_{\lambda}$ for the residue field of $\mathcal{O}_{F_0}$ modulo $\langle \lambda \rangle$; it has characteristic  $p=N_{F_0/\QQ}(\lambda)$.  Let $\overline{\bF}_{\lambda}=\overline{\bF}_{p}$ denote an algebraic closure of $\bF_{\lambda}$.
We identify the residue field of $\cO_{F}$ modulo $\langle \lambda \rangle$ with the field $\bF_{p^2}$ of size $p^2$ inside $\overline{\bF}_{\lambda}$, and denote by $\iota:\cO_{F,\lambda}\to \bZ_{p^2}=W(\bF_{p^2})$  the induced isomorphism and $\iota_\lambda:\cO_F\to \bZ_{p^2}$ its restriction to $\cO_F$. 

In this section, by an automorphism of an abelian variety, we always mean an automorphism compatible with the given polarization;
also, an endomorphism of an abelian variety $A$ over $k$ means a geometric endomorphism, namely an endomorphism of $A_{\overline{k}}$.
\end{notation}

\begin{definition}\label{poly}
Let $\lambda$ be as in Notation \ref{lambda}.
Let $E=F(\sqrt{-\lambda})$,
and let $\Phi$ denote the CM type of $E$ defined in Example~\ref{Ecmtypem5}.

Define $Z(\lambda)$ (resp.\ $ \tilde{Z}(\lambda)$) to be the divisor of the $j$-line $\bP^1_F$ whose support consists of abelian varieties of CM type $(E,\Phi)$, with complex multiplication by $\cO_E$ 
(resp.\ $\cO_F[\sqrt{-\lambda}]$) (each point has multiplicity $1$).  
We denote by $\cP_\lambda(x)$ (resp.\ $\tilde{\cP}_\lambda(x)$) 
the unique monic separable polynomial in $F[x]$ satisfying 
$Z(\cP_\lambda(x))=Z(\lambda)$ 
(resp.\ $Z(\tilde{\cP}_\lambda(x)) =\tilde{Z}(\lambda)$).

By definition, $Z(\lambda)\subseteq \tilde{Z}(\lambda)$ and hence $\cP_\lambda(x)$ divides $\tilde{\cP}_\lambda(x)$. We write $W(\lambda)=\tilde{Z}(\lambda)\setminus Z(\lambda)$ and $\cQ_\lambda(x)\in F[x]$ the unique monic separable polynomial satisfying $Z(\cQ_\lambda(x))=W(\lambda)$. Thus $\tilde{\cP}_\lambda(x)=\cP_\lambda(x)\cQ_\lambda(x)$.
\end{definition}

\begin{lemma}\label{LinF0}
Notation and assumptions as in Definition \ref{poly}.

   The CM cycles $Z(\lambda), W(\lambda)$ are  defined over ${F_0}$, and hence $\cP_\lambda(x), \cQ_\lambda(x)\in F_0[x]$.
\end{lemma}
\begin{proof}
    A point $z$ representing $A_z$ is in the cycle $\tilde{Z}(\lambda)$ if $A_z$ admits an endomorphism $s$ such that $s\circ s=-\lambda \in \End(A_z)$ and $s$ commutes with the $\cO_F$-action on $A_z$;
    %equivalent definition: $s\circ s^\dagger\lambda$ and trace $0$ (here $\dagger$ denotes the Rosati involution), 
    furthermore, $z$ is in $Z(\lambda)$ if $A_z$ admits an endomorphism $s$ as above such that $(1/2)(\Id+s)\in \End(A_z) \subset \End^0(A_z)$. 
    %equivalent definition: an endomorphism $s'$ such that $s\circ s^\dagger \frac{1+\lambda}{4}$ and trace $1$ (here we view cohomology of $A_z$ as $2$-dimensional space over $F\otimes_\bQ \bQ_\ell$ and take trace here; same for the trace above). 
    Any element in $\Gal(\overline{\bQ}/F_0)$ fixes $\lambda$, thus fixes these two cycles. \end{proof}
 %SAVE: indeed, the moduli interpretation indicates that the cycle is defined over $\bQ[\zeta_5]$ and then it is easy to check that taking complex conjugation also fix the cycle (as a set) so the cycle is defined over $\bQ[\sqrt{5}]$.

\begin{lemma}\label{Hilbert} 
Notation and assumptions as in Definition \ref{poly}.

Each closed point of $ Z(\lambda)$ is defined over the Hilbert class field  of $E=F(\sqrt{-\lambda})$. 

Each closed point of $W(\lambda)$ is defined over the ring class field  of the order $\cO_F[\sqrt{-\lambda}]$ of $E$.
\end{lemma}

\begin{proof}
By the theory of complex multiplication (see \cite[Chapter 5, Theorem 4.1]{LangCM}), the field of moduli\footnote{In our setting, the field of moduli is indeed the field of definition of these polarized CM abelian varieties by Lemma~\ref{Lfomfod}.} of the polarized CM abelian varieties (here part of the data is the embedding of $E$ into $\End^0$) in $Z(\lambda)$ (resp. $W(\lambda)$) are defined over the Hilbert class field $H_{\lambda}$ (resp. the ring class field  of the order $\cO_F[\sqrt{-\lambda}]$) of $E$. Thus we obtain the desired statements by the moduli interpretation of $\Sh$.  
%\yunqing{below are the old discussion on uniqueness of principal polarization because I thought the CM theory needs to take into account all polarizations of a fixed degree; however indeed the reference to Lang gives field of moduli for polarized CMAVs, thus we no longer need the discussion below.}
%By \cite[Proposition 4.5]{LMPTshimuradata} (whose hypotheses are satisfied by Proposition \ref{Poddclassuniquepp} and Lemma \ref{LsimpleCMtype}),  our assumptions on $\lambda$ imply that for each such CM abelian variety, there exists a unique principal polarization. %SAVE: here I put a unique one rather than at most one because we already talk about points in the CM cycles, i.e., it has already have a polarization compatible with the PEL structure.
%Since it is unique, the definition field  of the polarization is also $H_\lambda$ by Galois descent.
%SAVE: Under our assumption for $\lambda$, we have that each ideal class (i.e., each isomorphism class of the CMAV) admits at most one principal polarization. Hence we use Galois descent (see Lang for instance) to conclude that if $B$ admits a principal polarization, it must be defined over the definition field of $B$. 
\end{proof}

We deduce the following statement from Propositions~\ref{qrep} and \ref{CMbyorderQP}, Theorem~\ref{Tuniquecong}, and Remark~\ref{notmaxNT}.

\begin{lemma}\label{Poddrealroot}
Notation and assumptions as in Definition \ref{poly}.

The polynomial $\cP_\lambda(x)$ has a unique real root and 
 odd degree. That is, 
 $\# Z(\lambda)({\RR})=1$ and  $\# Z(\lambda)({\overline{\bQ}})$ is odd. 

The polynomial $\cQ_\lambda(x)$ has a unique real root and 
 odd degree. That is, $\# W(\lambda)({\RR})=1$ and  $\# W(\lambda)({\overline{\bQ}})$ is odd.
\end{lemma}

\subsection{The supersingular polarized $\cO_F$-module over $\overline{\bF}_p$}\label{mathcalH}

Let $p$ be an odd rational prime.
Let $\bQ_{p^2}$ denote the unique unramified quadratic extension of $\QQ_p$ in  
$\overline{\bQ}_{p}$, let $\ZZ_{p^2}$ be the ring of integers of $\QQ_{p^2}$, let $\bF_{p^2}$ be its residue field, and 
let $\overline{\bF}_{p^2}$ be an algebraic closure of $\bF_{p^2}$.
Let $\gamma$ be the non-trivial element in $\Gal(\bQ_{p^2}/\bQ_p)$.

\begin{notation} \label{NH11}
Let $\mathcal{H}/\overline{\bF}_{p^2}$ be a polarized (of degree prime to $p$) supersingular $\ZZ_{p^2}$-module of signature $(1,1)$; 
(in particular, $\mathcal{H}$ is of dimension $4$ and height $2$).
Note that $\mathcal{H}$ is unique up to isogeny 
(see \cite[Proposition 1.15]{Vollaard}). 
\end{notation}

 We compute the ring $\mathrm{End}_{\ZZ_{p^2},\pol} ^0(\mathcal{H})$ of quasi-isogenies of $\mathcal{H}$ which commute with the $\ZZ_{p^2}$-action and with the polarization, up to a similitude factor.

Let $D$ be the quaternion algebra over $\QQ_p$ ramified at $p$ and let $\varpi$ be a uniformizer of $\QQ_p$ (we will later take $\varpi=\lambda$). Then 
$D$ is the algebra over $\QQ_{p^2}$ generated by an element $\Pi$ satisfying $\Pi^2=-\varpi$ and $\Pi y=y^\gamma\Pi$, for all $y\in\QQ_{p^2}$. We  realize $D\hookrightarrow M_2(\bQ_{p^2})$ via $y\mapsto   \begin{bmatrix} y&0\\0&y^\gamma \end{bmatrix}$ and  $\Pi\mapsto \begin{bmatrix} 0 & -\varpi\\1&0\end{bmatrix}$.

\begin{lemma}\label{EndH}
   With the above notation, $$\mathrm{End}_{\ZZ_{p^2},\pol} ^0(\mathcal{H})^\times\simeq \bQ_{p^2}^\times D^\times.$$
\end{lemma}

\begin{proof}
By \cite[Lemma 1.13, Proposition 1.15, Remark 1.16 (1)]{Vollaard}, 
we have $\mathrm{End}_{\ZZ_{p^2},\pol} ^0(\mathcal{H})^\times\simeq \GU(W ,\{\cdot ,\cdot\}) ,$
where $W$ is a $\QQ_{p^2}$-vector space of dimension $2$, and $\{ \cdot,\cdot \}$ is a perfect skew $\gamma$-hermitian form on $W$ given by the matrix $t\begin{bmatrix} 1&0\\0&\varpi\end{bmatrix}$,\footnote{In \cite{Vollaard}, the $\gamma$-hermitian form is given by $t\begin{bmatrix} 1&0\\0&p\end{bmatrix}$; since $p/\varpi \in (\bQ_{p^2}^\times)^2$; we obtain the matrix $t\begin{bmatrix} 1&0\\0&\varpi\end{bmatrix}$ by a suitable change of basis.} 
for $t\in\ZZ_{p^2}^\times$ satisfying $t^\gamma=-t$.  Concretely,
$$%\mathrm{End}_{\ZZ_p^2} ^0(\mathcal{H})^\times
\GU(W ,\{\cdot ,\cdot\})\simeq
\QQ^\times_{p^2}D^\times \subset \GL_2(\QQ_{p^2})\simeq \GL(W),$$
where $\QQ^\times_{p^2}\subset \GL_2(\QQ_{p^2})$ denotes the subgroup of diagonal matrices, and  $D^\times$  is the subgroup $$D^\times=\left\{ \begin{bmatrix} y&-\varpi x\\x^\gamma&y^\gamma \end{bmatrix}\in \GL_2(\QQ_{p^2}) %\begin{bmatrix} 0 & -p\\1&0\end{bmatrix} ,\begin{bmatrix} 0&tp\\t&0\end{bmatrix}
\mid x, y\in\QQ_{p^2}^\times\right\}.\qedhere $$
\end{proof}

By direct computations, we deduce the following lemma.
\begin{lemma}\label{trx}
Let $x\in \QQ_{p^2}^\times D^\times \subset \GL_2(\QQ_{p^2})$. Then $x\Pi=-\Pi x$ if and only if $$x\in \bQ_{p^2}^\times \cdot \left(\bQ_p \begin{bmatrix} t & 0 \\ 0 &-t \end{bmatrix} + \bQ_p \begin{bmatrix} 0 & \varpi t \\ t &0 \end{bmatrix}\right)^\times.$$ 
In particular, if $x\Pi=-\Pi x$, then $\mathrm{tr} (x)=0$.
\end{lemma}

\subsection{Reduction modulo $\lambda$ of CM cycles I}

Recall the notation and assumptions from Definition \ref{poly}.
The goal of this section is the proof of the following statement.
Note that since  $\lambda \ne 2, u\sqrt{5}$, then $A_P\bmod \lambda$ and $A_R\bmod \lambda$ are not isomorphic.
\begin{proposition}\label{RootInPairs}
Notation and assumptions as in Definition \ref{poly}.
There exists a unique point in ${\bP}^1 ({\overline{ \mathbb F}}_\lambda)$ such that the preimage of this point under the reduction map contains an odd number of geometric points in the support of $Z(\lambda)$ (resp. $W(\lambda)$). Moreover, this unique point is $A_P\bmod \lambda$ for $W(\lambda)$ and is $A_R \bmod \lambda$ for $Z(\lambda)$.

For any other point $x\in {\bP}^1({\overline{ \mathbb F}}_\lambda)$, the geometric points in the support of $Z(\lambda)$ (resp. $W(\lambda)$) which are in the preimage of $x$ under the reduction map occur in conjugate pairs.
\end{proposition}
% Note that the above proposition implies the roots in pairs property only fails for one of $A_P, A_R$.  In particular, either $A_P$ or $A_R$ (or both) has CM by $\cO_E$ modulo $\lambda$.  

For $A$ an abelian variety corresponding to a point of $\Sh$, we denote by $\Aut_{\cO_F} (A)$ the group of automorphisms of $A$ which commute with the action of $ \mathcal{O}_F$ and preserve the polarization.

\begin{lemma}\label{aut} 
Let $\lambda$ be as in Notation \ref{lambda}, and $A/\overline{\bF}_\lambda$ be an abelian variety corresponding to a point in $\bP^1(\overline{\bF}_\lambda)$.
If $A$ is not geometrically isomorphic to $A_P\bmod \lambda$, 
% over ${\overline{\mathbb F}}_\lambda$, 
then $\Aut_{\cO_F} (A) \simeq \{\pm 1\} \times G_0$ where the group $G_0$ is isomorphic to 
$\bZ/5 \bZ$, $\bZ/10 \bZ$, or $\bZ/15 \bZ$.
Furthermore, the last two cases occur if and only if 
$A$ is geometrically isomorphic to $A_R\bmod \lambda$ or $A_Q\bmod \lambda$, respectively.
\end{lemma}

\begin{proof}
The hypothesis that $A$ is not geometrically isomorphic to $A_P\bmod \lambda$ implies that 
 $A=\mathrm{Jac}(C)$, where $C/\overline{\bF}_\lambda$ is a smooth curve.
 Note that $C$ is not hyperelliptic.
 By \cite[Appendice]{lauterappendix}, $\Aut(A) \simeq \{\pm 1\} \times \Aut(C)$.
 By Proposition~\ref{Pnomorebigaut}, $\Aut(C) \simeq \bZ/5 \bZ$, $\bZ/10 \bZ$, or $\bZ/15 \bZ$, 
 with the last two cases occurring exactly for $C_R$ and $C_Q$ respectively. 
 
 In each case, the action of the unique subgroup of order 5 yields the action of $\cO_F$ on $A$.  Hence, in particular, all of the 
 automorphisms above commute with the action of $\cO_F$.
 \end{proof}
%SAVE: we can obtain the same conclusion for $\lambda |p \gg 1$ without quoting results of curves. Namely we pick a suitable level structure so that the corresponding SV is a fine moduli space. Away from the primes dividing the level structure, the points with extra automorphisms are exactly the ones where the Hecke correspondence (of the level) collapse. For $p\gg 1$, these are exactly the ones corresponding to the ones in characteristic $0$ (namely the elliptic points).

\begin{proof}[Proof of Proposition \ref{RootInPairs}]
For $\lambda$ as in Notation \ref{lambda}, let $p=N_{F_0/\QQ}(\lambda)$. 
Let $A/\overline{\bF}_\lambda$ be an abelian variety which is the reduction modulo $\lambda$ of a point in $Z(\lambda)(\overline{\bQ})$. 
Then $A$ has complex multiplication by $\mathcal{O}_E$, for $E=F(\sqrt{-\lambda})$. By Proposition \ref{ST}, $A$ is basic.

The action of $\mathcal{O}_F$ on $A$ induces a decomposition of the $p$-divisible group $A[p^\infty]$, as $$A[p^\infty]=A[(\lambda)^\infty]\oplus A[(\lambda^\tau)^\infty],$$
where $A[(\lambda)^\infty]$ and $A[(\lambda^\tau)^\infty]$ are two polarized $p$-divisible groups, of height 4, with multiplication by $\mathcal{O}_{F, \lambda}$ and $ \mathcal{O}_{F,{\lambda^\tau}}$ respectively, of signature $(1,1)$ and $(2,0)$. 
Let $\CH_\lambda$ denote $A[(\lambda)^\infty]$, the polarized $p$-divisible subgroup of signature $(1,1)$. By  Proposition \ref{ST}, $\CH_\lambda$ is supersingular. Via the isomorphism $\iota:\mathcal{O}_{F,\lambda}\to \ZZ_{p^2}$ from Notation \ref{lambda}, we regard $\CH_\lambda$ as a polarized supersingular $\ZZ_{p^2}$-module. 

Recall, from Notation~\ref{NH11}, that
$\mathcal{H}/\overline{\bF}_{p^2}$ is the (unique up to isogeny) polarized supersingular $\ZZ_{p^2}$-module of signature $(1,1)$.
Thus
%by the uniqueness of $\mathcal{H}$, 
there exists an isogeny $\rho:\mathcal{H}\to \CH_\lambda$, of polarized $\bZ_{p^2}$-modules, defined over $\overline{\bF}_\lambda\simeq \overline{\bF}_{p^2}$.  
%By definition, $\langle\iota_\lambda(\lambda)\rangle=\langle p\rangle$, hence there exists $u\in \ZZ_{p^2}^\times$ such that $u\Pi\in D$ satisfies $(u\Pi)^2=-\iota(\lambda)$. 

Let $B$ be the quaternion algebra over $F_0$ ramified at 
$\{\sqrt{5},\lambda,\infty_1, \infty_2\}$. A direct computation shows that $\End^0_{\cO_F,\pol}(A)\cong FB$, where $\End^0_{\cO_F,\pol}(A)$ denotes the $F_0$-algebra of quasi-isogenies of $A$ which commute with the $\cO_F$-action and preserve the skew-hermitian form (on all $\ell$-adic and crystalline cohomologies) associated to the polarization up to a scalar in $F_0$. (Compare to Lemma \ref{EndH} for the local computation at $\lambda$.) The isogeny $\rho$ induces an isomorphism $B\otimes \QQ_p\simeq D$. Let $\mathcal{O}_B$ be the order in $B$ given by $\End_{\cO_F,\pol}(A)\cap B$;\footnote{Here by a slight abuse of notation, $\cO_B$ may not necessarily be a maximal order.} we then have $\mathcal{O}_B\otimes \ZZ_p\subset \mathcal{O}_D$ under the isomorphism $B\otimes \QQ_p\simeq D$, where $\mathcal{O}_D$ denotes the maximal order of $D$. %SAVE: because $D$ over $\bQ_p$ is ramified thus exists unique maximal order.

%SAVE: \yunqing{The idea of normalized embedding is that we want to rule out the ambiguity of the CM type (i.e., choice of embedding of $E$ into $\End^0(\CA) \subset \End^0(A)$); this can be done by studying the action of $\CO_{F_0}[\sqrt{-\lambda}]$ on the tangent space of $\CH$; it doesn't matter whether we view $\Lie \CH$ as a $2$-dimensional $\CO_{F_0}/(\lambda)\otimes_{\bZ,\lambda}\FF_p \cong \FF_p$-vector space or a $1$-dimensional $\FF_{p^2}$-vector space, as the normalized condition is always that $\sqrt{-\lambda}$ maps to the reduction of $\sqrt{-\lambda}$ (positive imaginary part) on the entire Lie algebra.}

Recall Example~\ref{Ecmtypem5}.
Let $\iota_E$ denote the injective homomorphism $\cO_{F_0}(\sqrt{-\lambda}) \rightarrow \overline{\bZ}_{p}$ induced by 
the embedding $F_0(\sqrt{-\lambda})\hookrightarrow \CC$ given by $(\sigma_1,+)$ (and this is the same embedding given by $(\sigma_4,+)$).\footnote{There is a natural map $\overline{\bZ}_{p}\hookrightarrow\CC$, given by identifying CM abelian varieties with CM type $\Phi$ (i.e., points on $Z(\lambda), W(\lambda)$) as $\overline{\bQ}_p$-points on the Shimura curve.} We use $\overline{\iota_E}$ to denote the homomorphism $\cO_{F_0}(\sqrt{-\lambda}) \rightarrow \overline{\bF}_{p}$ obtained by the composition of the reduction map $\overline{\bZ}_{p} \rightarrow \overline{\bF}_{p}$ and $\iota_E$.
We call a homomorphism $ \mathcal{O}_{F_0}[\sqrt{-\lambda}]\to \mathcal{O}_B$ {\em normalized} if the induced action on the tangent space of $\mathcal{H}$ via $\cO_B \subset \cO_D$ agrees with that induced by $(\overline{\iota_E}, \overline{\iota_E})$ on $\overline{\bF}_{p}^2$.

{\bf Claim:} There is a bijection between the set of points in $Z(\lambda)\cup W(\lambda)$ whose reduction is $A$ and isomorphism classes of normalized homomorphisms $\theta: \mathcal{O}_{F_0}[\sqrt{-\lambda}]\to \mathcal{O}_B$ of $\cO_{F_0}$-algebras extending $\iota_A:\cO_{F_0} \rightarrow \cO_B$ given by the $F_0$-action on $A$. \footnote{Here we say two such homomorphisms are isomorphic if they are conjugate by an element in $\Aut_{\cO_F}(A) \subset F^\times B^\times$.}

{\bf Proof of claim:}
Indeed, let $\mathcal{A}$ be a lifting of $A$ to characteristic $0$ lying in $Z(\lambda)$ or $W(\lambda)$; in particular, $\mathcal{A}$ is an abelian variety 
with CM by $(E, \Phi)$ in \Cref{Ecmtypem5}, and the reduction map to $A$ is compatible with the $F$-action.
The action of $\mathcal{O}_{F}[\sqrt{-\lambda}]$ on $\mathcal{A}$ and the reduction map to $A$ define a homomorphism of algebras $\iota_\mathcal{A}:\mathcal{O}_{F}[\sqrt{-\lambda}]\to \End^0_{\cO_F,\pol}(A)\cong FB$, which extends $\iota_A$. An argument similar to \Cref{CMonlyif} shows that image of $\iota_\cA$ lies in $B$, 
and hence in $\cO_B$ by definition. 
%SAVE: more precisely, let $s$ denote the endomorphism corresponding to $\sqrt{-\lambda}$, since $s\circ s^\dagger$ is a scalar map, and $\daggar$ corresponds to complex conjugation on $F$ and standard convolution on $B$, thus this restriction forces $s$ either in $B$ or a totally imaginary element in $F$ times $B$. Then by positivity $s\circ s^\dagger=\lambda$ and positivity of reduced norm on $B$, we conclude that $s\in B$. Since $s$ is an endomorphism and by definition $\cO_B$ is defined to be the order given by endomorphisms, thus $s\in \cO_B$ by definition.
We denote the restriction homomorphism by $\theta_{\cA}: \mathcal{O}_{F_0}[\sqrt{-\lambda}]\to \cO_B$.
Since $\cA$ has CM type $\Phi$, by definition $\theta_{\cA}$ is normalized. 
%SAVE: indeed, on the Lie algebra of $\cA[(\lambda)^\infty]$, the action of $\mathcal{O}_{F_0}[\sqrt{-\lambda}]$ is given by $(\iota_E,\iota_E)$; thus the reduction is normalized. 

Conversely, for any normalized embedding 
$\theta: \mathcal{O}_{F_0}[\sqrt{-\lambda}]\to \mathcal{O}_B$, we use Lubin--Tate theory to construct a lifting of $A$ to characteristic $0$ in $Z(\lambda)\cup W(\lambda)$.  More precisely, by Serre--Tate theory, we only need to construct a lifting $\CG$ of the polarized $p$-divisible group $A[p^\infty]$ such that $\cO_F[\sqrt{-\lambda}]\subset \End_{\mathrm{pol}} (\CG)$ with CM type $\Phi$. 
We use \cite[Proposition~2.1]{Gross}. Let $\CG_{0}$ denote the supersingular $p$-divisible group, of dimension $1$ and height $2$, 
over $\overline{\FF}_\lambda$; (it is unique up to isomorphism).
Note that $H_\lambda$ is isomorphic to $\CG_0 \otimes_{\cO_{F_0,\lambda}} \cO_{F,\lambda}$.
The group $\CG_0$ gives a formal $\cO_{F_0,\lambda}$-module/group of dimension $1$ and height $2$ with endomorphism ring $\cO_D$. The normalized embedding $\theta$ induces an embedding $\cO_{F_0, \lambda}[\sqrt{-\lambda}]\rightarrow \cO_D$, which makes the formal module/group associated to $\CG_0$ a formal $\cO_{F_0, \lambda}[\sqrt{-\lambda}]$-module/group of height $1$, which admits a unique lifting to a formal $\cO_{F_0, \lambda}[\sqrt{-\lambda}]$-module/group of height $1$ over $W(\overline{\FF}_\lambda)$. We use $\CG_1'$ to denote the corresponding $p$-divisible group of dimension $1$ and height $2$ with $\mathcal{O}_{F_0}[\sqrt{-\lambda}]$-action; in particular, $\CG_1:=\CG_1'\otimes_{\cO_{F_0,\lambda}} \cO_{F,\lambda}$ is a lifting of $H_\lambda$. Moreover, since the image of $\theta$ lies in $\End_{\cO_F}(H_\lambda)\cong \End(\CG_0)$, i.e., the $\cO_F$-action, we then obtain an $\cO_F[\sqrt{-\lambda}]$-action on $\CG_1$. 

Furthermore, recall $A[(\lambda^{\tau})^\infty]$ has signature $(2,0)$ and has dimension $2$ and height $4$; it is isomorphic to the direct sum of $\CG_0$ (equipped with $\cO_F$-action, with the induced $\cO_F$-action on $\Lie \CG_0$ being induced by $\sigma_2$) with itself. %SAVE: YT: I think $A$ is superspecial for all points on the Shimura curve as we are at good reduction prime and A has endomorphism by $\cO_F$.
Let $\CG_2'$ denote the unique $p$-divisible group of dimension $1$ and height $2$ over $W(\overline{\FF}_\lambda)$ lifting $\CG_0$ along with the $\cO_F$-action. Define $\CG_2:= \CG_2' \otimes_{\cO_{F_0, \lambda^\tau}} \cO_{F_0, \lambda^\tau}[\sqrt{-\lambda}]$, which is a $p$-divisible group of dimension $2$ and height $4$ equipped with $\cO_F[\sqrt{-\lambda}]$-action. This is our desired lift of $A[(\lambda^{\tau})^\infty]$, up to $M_2(F_{0,\lambda^\tau})$-conjugacy; we pick the lift such that the induced $\cO_F[\sqrt{-\lambda}]$-action agrees with $\theta$ localized at $\lambda^\tau$. Therefore, $\CG_1 \oplus \CG_2$ is our desired lift of the $p$-divisible group $A[p^\infty]$ with $\cO_F[\sqrt{-\lambda}]$-action compatible with $\theta$. Since the image of $\theta$ lies in $\cO_B$, which preserves the polarization on $A$, we can lift it to a polarization on $\CG$ compatible with $\cO_F[\sqrt{-\lambda}]$-action.
%SAVE: there are two ways to see this. One way is that the construction above is canonical and when one applies to the dual of $A$ and polarization as an isomorphism from $A$ to its dual, we obtain the lifting. The other way is that the Mumford--Tate group of the CMAV is the Nm in Q part of E and it lies in $B$, which preserves the polarization; in other words, the image of the polarization in char p is preserved by all char 0 Galois action and hence is indeed a Tate cycle and in this case a polarization.
Such a polarization is unique by \Cref{Tuniqueconganym} (its proof also applies to the non-maximal order case); thus we associate a point in $Z(\lambda)\cup W(\lambda)$ to $\theta$ and by the construction, it is exactly the inverse to the map $\cA \mapsto \theta_\cA$ in the paragraph above.
This ends the proof of the claim.

For a normalized embedding $\theta: \cO_{F_0}[\sqrt{-\lambda}]\to\mathcal{O}_B$ with $\theta(\sqrt{-\lambda})=\alpha$, we note that the conjugate embedding $\theta'$ given by the $\cO_{F_0}$-algebra homomorphism with $\theta'(\sqrt{-\lambda}):= -\alpha$ is also normalized. Indeed $\overline{\iota_E}(\sqrt{-\lambda})=0$ and so $\alpha$, and hence also $-\alpha$, acts as the $0$-map on $\Lie H$; in other words, $\theta'$ is also normalized. 
By Serre--Tate theory and the above bijection, $\theta$ corresponds to a point in $Z(\lambda)$ if and only if $(1/2)(1+\theta(\sqrt{-\lambda}))\in \cO_B$. This condition holds for $\theta$ if and only if it holds for $\theta'$. Thus the points corresponding to $\theta, \theta'$ are either both in $Z(\lambda)$ or both in $W(\lambda)$.

If $\theta, \theta'$ above give rise to the same point in $Z(\lambda)\cup W(\lambda)$, then by the above bijection, there exists $\epsilon \in \Aut_{\cO_F}(A)$ such that $\epsilon \alpha \epsilon^{-1} =-\alpha$. Consider the images of $\epsilon, \alpha$ under the injective homomorphism 
${\mathrm{End}}^0_{\cO_F,{\mathrm{pol}}}(A)\hookrightarrow {\mathrm{End}}^0_{\cO_F}(\CH_\lambda)\cong \QQ_{p^2} D$.
Since our discussion is up to conjugacy, by the Noether--Skolem Theorem, 
we may assume $\alpha=\Pi$.
By Lemma~\ref{trx} (taking $\varpi=\lambda$), we deduce $\mathrm{tr} (\epsilon)=0$. 

By Lemma~\ref{aut}, if there exists $\epsilon \in \Aut_{\cO_F}(A)$ of trace $0$, then $A$ is either $A_P\bmod \lambda$ or $A_R \bmod \lambda$. 
Hence,  if $A$ is neither $A_P\bmod \lambda$ nor $A_R \bmod \lambda$, then the roots of $\cP_\lambda (x)$  (resp. $\cQ_\lambda (x) $) show up in conjugate pairs (i.e., $\theta, \theta'$) in the $\lambda$-adic neighborhood of $A$.

Since the degree of $\cP_\lambda (x)$  (resp. $\cQ_\lambda (x) $) is odd by Lemma~\ref{Poddrealroot}, the number of points in $\mathbb{P}^1(\overline{\mathbb{F}}_\lambda)$ whose number of preimages under the reduction map is odd is exactly one and the point is either $A_P\bmod \lambda$ or $A_R\bmod \lambda$.
The final claims about
the unique exceptional point (i.e., the point with an odd number of preimages)
are proved in Lemma~\ref{L1}.
%It remains to show that the unique exceptional point (i.e., the point with odd number of preimages) is $A_P\bmod \lambda$ for $W(\lambda)$ and is $A_R \bmod \lambda$ for $Z(\lambda)$ and the number of points of $Z(\lambda)$ (resp. $W(\lambda)$) in the $\lambda$-adic neighborhood of $A_P \bmod \lambda$ (resp. $A_R \bmod \lambda$)  is even and the points occur in conjugate pairs. These assertions are proved in  Lemma \ref{L1}.
\end{proof}

\begin{lemma}\label{L1}
The unique exceptional point is $A_P\bmod \lambda$ for $W(\lambda)$ and is $A_R \bmod \lambda$ for $Z(\lambda)$.
The number of points of $Z(\lambda)$ (resp.\ $W(\lambda)$) in the $\lambda$-adic neighborhood of $A_P \bmod \lambda$ (resp.\ $A_R \bmod \lambda$)  is even and these points occur in conjugate pairs.
\end{lemma}
\begin{proof}
We use the notation from the proof of \Cref{RootInPairs}.

By \Cref{aut}, if $A=A_P\bmod \lambda$ or $A_R\bmod \lambda$,
then there are exactly five elements in $ \Aut_{\cO_F}(A)/\{\pm 1\}$ of trace $0$, they are $\epsilon_i=\zeta_5^i\epsilon_0$, for $0\leq i\leq 4 $ and $\epsilon_0^2=1$, $\epsilon_0\neq 1$. 
In other words, modulo the center $\cO_F$,\footnote{Note that elements that differ by an element in the center give rise to exactly the same conditions on $\alpha$; thus we only need to work with elements module $\cO_F$.} there is only one possible element $\epsilon_0\in \mathrm{Aut}_{\cO_F}(A)$ which satisfies $\alpha \epsilon_0 = -\epsilon_0\alpha$ for some $\alpha\in \cO_B$ satisfying $\alpha_0^2=-\lambda$. 

We first prove that all preimages of $A=A_P\bmod \lambda$ in $Z(\lambda)$ occur in conjugate pairs, which implies the assertions for $A_P$. Write $A_P=A_1\times A_2$, where $A_1$ and $A_2$ are abelian surfaces with CM by $\cO_F$. Since $2$ is inert in $F/\QQ$, the $2$-adic Tate module is $T_2(A)\cong \cO_{F,2}\oplus \cO_{F,2}$, equipped with the natural $\cO_F$-action by multiplication on each part; and by \Cref{aut}, $\epsilon_0=\begin{bmatrix}
    1 & 0\\ 0& -1
\end{bmatrix}\in M_2(\cO_{F,2})=\End_{\cO_F}(T_2(A))$.
%SAVE: here we are not claiming $(A_1)_{\overline{\FF}_\lambda} \cong (A_2)_{\overline{\FF}_\lambda}$; the fact that they are isogenous as abelian surfaces with $\cO_F$-action is compatible with the Tate module description. To see they are isogenous, note that they are isogenous as abelian surfaces because both are supersingular. Then $\cO_F$-action is just an embedding of $\cO_F$ to $M_2(B_{\infty, p\})$; all embeddings are conjugate by Noether--Skolem, which means the two abelian surfaces equipped with $\cO_F$-actions are isogenous.
The condition $\alpha \epsilon_0 = -\epsilon_0\alpha$ shows that the $\alpha$-action on $T_2(A)$ must be of the form $\begin{bmatrix}
    0 & b \\ c & 0
\end{bmatrix}$ for some $b,c \in \cO_{F,2}$. Then we observe that $(1/2) (1+\alpha) \notin M_2(\cO_{F,2})=\End_{\cO_F}(T_2(A))$; thus the points corresponding to $\theta=\theta'$ with $\theta(\sqrt{-\lambda})=\alpha$ do not lie in $Z(\lambda)$.

To prove the rest of the assertions, we give an explicit description of all $\alpha \in{\mathrm{End}}_{\cO_F,{\mathrm{pol}}}(A)$  satisfying  $\alpha \epsilon_0 = -\epsilon_0\alpha$  and $\alpha^2=-\lambda$ for $A=A_P\bmod \lambda$ and $A_R\bmod \lambda$. 
Fix $\alpha_0$ satisfying these properties; we claim that $r:=\alpha \alpha_0^{-1}\in \End^0_{\cO_F,{\mathrm{pol}}}(A)$ actually lies in $\cO_B^\times$; we will also prove that the set of all such $r$ form the group $\{\pm 1\} \times \ZZ/5\ZZ$ for $A=A_P\bmod \lambda$, and the group $\{\pm 1\}$ for $A=A_R\bmod \lambda$.

Since $\alpha, \alpha_0 \in \cO_B$ and $\alpha^2=\alpha_0^2=-\lambda$, we have $r\in (\cO_B[1/\lambda])^\times$. Since $A=A_P \bmod \lambda$ or $A_R \bmod \lambda$ and $\lambda \nmid 2$, we have $A[(\lambda)^\infty]=A_1[(\lambda)^\infty] \times A_2[(\lambda)^\infty]$ and a direct computation shows that $\cO_{B,\lambda}=\cO_D$. Thus we only need to work locally at $\lambda$ and show that $r\in \cO_D^\times$ to conclude that $r\in \cO_B^\times$.

Recall $\beta_0\in F$ from \eqref{Ebeta0} and set $\epsilon_1:=\beta_0\epsilon_0$; then $\epsilon_1^2=\beta_0^2\in F_0$; so $\epsilon_1^2$ is totally negative and the same argument for $\sqrt{-\lambda}$ in the proof of \Cref{RootInPairs} implies that $\epsilon_1\in \cO_B$.
By the Noether--Skolem Theorem, we may assume that the image of $\epsilon_1$ under the injective homomorphism $\mathrm{End}_{\cO_F}(A)\to \QQ_{p^2} D$ is $\begin{bmatrix}
    \beta_0 & 0 \\ 0 & -\beta_0
\end{bmatrix} \in \cO_D$,\footnote{Although the Noether--Skolem Theorem only implies uniqueness up to conjugacy by $D^\times$, the maximal order $\cO_D$ in the ramified quaternion algebra $D$ is stable under conjugation by $D^\times$; thus this reduction step is valid for our purposes.} where we view $\beta_0\in \ZZ_{p^2}$ via $\iota$ and we use the coordinate of $D^\times \subset \GL_2(\bQ_{p^2})$ as in the proof of \Cref{EndH}. The condition $\alpha \epsilon_0 = -\epsilon_0\alpha$ is equivalent to $\alpha \epsilon_1 = -\epsilon_1\alpha$.
By direct computation, 
if $\alpha \epsilon_1 = -\epsilon_1\alpha$ and $\alpha^2=-\lambda$, 
then the image of $\alpha$ in $\cO_D$ is of the form $$\alpha=\begin{bmatrix}
    0 & -\lambda x^\gamma \\
   x  & 0\\
\end{bmatrix},$$
for some $x\in\QQ_{p^2}^\times$ satisfying $xx^\gamma=1$; thus $x\in \bZ_{p^2}^\times$. We write $\alpha_0=\begin{bmatrix}
    0 & -\lambda x_0^\gamma \\
   x_0  & 0\\
\end{bmatrix}$. Then
 $r=\alpha\alpha_0^{-1}= \begin{bmatrix}
     x x_0^\gamma & 0 \\ 0 & x^\gamma x_0
 \end{bmatrix}$, which lies in $\cO_D^\times$. Thus we conclude that $r\in \cO_B^\times \subset \Aut_{\cO_F}(A)$.
 %SAVE: heuristically, for $r=\begin{bmatrix}   x  & 0 \\ 0 & x^\gamma \end{bmatrix}$, we define $r':= \begin{bmatrix}   x^\gamma  & 0 \\ 0 & x \end{bmatrix}$; one may justify that such an action is defined on the intersection of the diagonal torus in $D^\times$ with $\cO_B$ (one way to do so is that $r\mapsto r'$ here is indeed the Rosati involution. Thus all desired $r$ are elements in $\cO_B$, commuting with $\epsilon_1$ (hence lie in the diagonal torus) and  $rr'=1$. One can verify that if $r\alpha_0$ satisfies the condition for $\alpha$.

We apply Lemma~\ref{aut} to find these $r$.
For $A=A_R\bmod \lambda$, each automorphism in $\Aut_{\cO_F}(A)=\{\pm 1\}\times \mu_5 \subset F$ acts via scalar multiplication; thus $\Aut_{\cO_F}(A)\cap \cO_B^\times =\{\pm 1\}$. Given $\alpha_0$, the element $-\alpha_0$ also satisfies the conditions for $\alpha$. Thus we have exactly two conjugate embeddings $\theta, \theta'$ corresponding to a unique point in $Z(\lambda)\cup W(\lambda)$.
For $A=A_P\bmod \lambda $, $\Aut_{\cO_F}(A)=\{\pm 1\}^2\times \mu_5^2$ and its image in $\bZ_{p^2}\cO_D$ using the above coordinates is $\begin{bmatrix}
 \pm   \zeta_5^i & 0 \\ 0 & \pm \zeta_5^j
\end{bmatrix} $ and the only $r$'s of the form $\begin{bmatrix}
     y& 0 \\ 0 & y^\gamma
 \end{bmatrix}$ are $\pm \begin{bmatrix}
 \zeta_5^i & 0 \\ 0 & \zeta_5^{-i}
\end{bmatrix} $.  One can check directly that $r\alpha_0$ satisfies the condition for $\alpha$ for these ten values of $r$. By direct computations, these pairs are conjugate to each other.
%SAVE: for $r=- \begin{bmatrix}\zeta_5^i & 0 \\ 0 & \zeta_5^{-i}$, the conjugation is given by $\epsilon_0 \begin{bmatrix}\zeta_5^i & 0 \\ 0 &1$.
In other words, we also have exactly one point in $Z(\lambda)\cup W(\lambda)$ corresponding to $\alpha$'s.

Recall we proved that all points in $Z(\lambda)$ occur in pairs 
for preimages of $A=A_P\bmod \lambda$,
and thus the unique exceptional point (not in a pair) lies in $W(\lambda)$. 
Also $\# W(\lambda)$ is odd and there is only one other exceptional point (not in a pair), and the reduction of this exceptional point is $A_R\bmod \lambda$.
Thus we conclude that this exceptional point lies in $Z(\lambda)$, and all preimages of $A_R\bmod \lambda$ in $W(\lambda)$ occur in pairs.
\end{proof}

\subsection{Switching the roles of $P$ and $R$} \label{SswitchPR}

Later, when applying Proposition \ref{RootInPairs}, it is convenient to change the coordinate system, to switch the points $R$ and $P$ while fixing $Q$. Here, we introduce the relevant notation.

Let $\circ$ be the unique automorphism of $\mathbb{P}^1_{F_0}$ which fixes $j_Q$ and switches $j_R$ and $j_P$. 
Concretely, $\circ$ is the fractional linear transformation $x\mapsto x^\circ=\frac{cx}{x-c}$, where $c=\frac{27}{4}(u\sqrt{5})^{-5}$. 
Then,
\begin{equation}\label{circfacts}
    x^\circ-y^\circ=\frac{-c^2(x-y)}{(x-c)(y-c)}, \text{ and } 
x^\circ-c= \frac{c^2}{x-c}.\end{equation}

In particular, if $j\in F_0$, then $j^\circ\in F_0$ and $(j^\tau)^\circ=(j^\circ)^\tau$. 
Also, 
if $v$ is a prime of $F_0$ satisfying $\mathrm{val}_v(c)=0$ 
(concretely, $v$ is relatively prime to $2,3,u\sqrt{5}$), 
then $\mathrm{val}_v (j-c)=-\mathrm{val}_v(j_\circ-c)$ 
since $(j-c)(j_\circ-c) =%(j-c) \frac{c^2}{j-c}=
c^2$.  We omit the proof of the following lemma. 

\begin{lemma} \label{Lcirc}
With the above notation, let $f(x)\in F_0[x]$ be monic of degree $n$, and denote by $f^\circ (x)$ the monic polynomial of degree $n$, 
whose roots are the images under $\circ$ of the roots of $f(x)$. Then
    $$f^\circ(x)=\frac{1}{f(c)}(x-c)^{n} f (x^\circ) \in F_0[x] \text{ and } f(x)f^\circ(x^\circ)=\frac{1}{f(c)}(x^\circ-c)^{n} f(x)^2.$$
\end{lemma}

%SAVE: $f(x) = \prod (x-r_i)$ and $f(c) = \prod(c-r_i)$.
%Now $f(\sigma(x)) = \prod (\frac{cx}{x-c} - r_i)$.
%So $\frac{1}{c} (x-c)^n f(\sigma(x)) = \prod (x + \frac{cr_i}{c-r_i})$.

\subsection{Reduction modulo $\lambda$ of CM cycles II}\label{denominator}

Recall notation and assumptions from Definition \ref{poly}.
By Lemmas \ref{LinF0} and \ref{Lcirc},  we deduce $\cP_\lambda(x), \cP_\lambda^\circ(x), \cQ_\lambda(x), \cQ_\lambda^\circ(x)\in F_0[x]$.  

Define $a_\lambda \in \cO_{F_0}$ (resp.\ $b_\lambda \in \cO_{F_0}$) 
to be the totally positive least common multiple of the denominators of the coefficients of $\cP_\lambda(x) \in F_0[x]$, (resp.\ $\cP_\lambda^\circ (x)\in F_0[x]$).
 Then $a_\lambda, b_\lambda  \in \cO_{F_0}$ are uniquely defined up to multiplication by totally positive units, that is up to squares of units since $ \mathcal{U}^+_{F_0}= \mathcal{U}^2_{F_0}$. 

\begin{proposition}\label{asquare}
With notation as above, 
$\mathrm{val}_v(a_\lambda)$ is even for all primes $v$ of $F_0$, 
with $v\neq \lambda$. 
In particular,  $a_\lambda\bmod \lambda$ is a square (possibly $0$).
%SAVE: indeed if $\lambda \mid a_\lambda$, then the reduction is $0$; if not, then $val_{\lambda}(a_\lambda)=0$ also even thus a square from the first assertion.
\end{proposition}

\begin{proof}
By \Cref{L1}, the number of geometric points of $Z(\lambda)$ in the $\lambda$-adic neighborhood of $A_P\bmod \lambda$ is even. Let $\beta$ be the $j$-invariant of a point on $Z(\lambda)$; that is, $\beta$ is a root of ${\mathcal{P}}_\lambda(x)$. By Lemma \ref{Hilbert}, $\beta\in H_\lambda$, the Hilbert class field of $E=F(\sqrt{-\lambda})$. 

Let $v$ be a prime of $F_0$, and $\nu$ a prime of $H_\lambda$ dividing $v$. Assume $v\neq \lambda, u\sqrt{5}$. Then $v$ is unramified in $H_\lambda$, and $\mathrm{val}_v(a)=\mathrm{val}_\nu(a)$, for all $a\in F_0$. To prove that $\mathrm{val}_v(a_\lambda)$ is even,
%for $v\neq \lambda, u\sqrt{5}$, 
it suffices to show that if $\mathrm{val}_\nu(\beta)<0$ then $\mathrm{val}_\nu(\beta)$ is even.

Choose a local parameter $t$ around $P$ on the smooth Deligne--Mumford stack $\Sh$ above the coarse moduli space (i.e., the $j$-line ${\mathbb P}_F^1$). 
In a neighborhood of $P$ (localized at $v$), 
\begin{equation}\label{eqt}
1/j=\prod_{\gamma\in \Gamma} t^\gamma, \end{equation} where 
$\Gamma=\mathrm{Aut}_{{\cO}_F}(A_P)/ (\{\pm 1\} \times \mu_5)$. 

Since the $\Gamma$-action is \'etale, $\mathrm{val}_\nu(t_\beta )=\mathrm{val}_\nu(t_\beta^\gamma )$.
We deduce that if $\mathrm{val}_\nu(\beta )<0$, then $\mathrm{val}_\nu(\beta )=\# \Gamma\cdot\mathrm{val}_\nu(t_\beta )$, 
%and in particular %$\mathrm{val}_\nu(\beta )$
which is even since  $\# \Gamma$ is even.  

Assume $v=u\sqrt{5}$. By Corollary \ref{Cnot5}, $\mathrm{val}_{u\sqrt{5}} (\beta)\geq 0 $ and hence 
$\mathrm{val}_{u\sqrt{5}} (a_\lambda)= 0 $.
We conclude that $\mathrm{val}_v(a_\lambda)$ is even for all primes $v$
of $F_0$, with $v\neq \lambda$.
\end{proof}

Similarly, 
define $c_\lambda \in \cO_{F_0}$ (resp.\ $d_\lambda  \in \cO_{F_0}$) to be the totally positive least common multiple of the denominators of the coefficients of $\cQ_\lambda(x) \in F_0[x]$ (resp.\ $\cQ_\lambda^\circ (x)\in F_0[x]$).

\begin{proposition}\label{asquare2}
With notation as above, $\mathrm{val}_v(d_\lambda)$ is even for all primes $v$ of $F_0$, with $v\neq 2, \lambda$. 
In particular, either $d_\lambda\bmod \lambda$ or $2d_\lambda\bmod \lambda$ is a square (possibly $0$).
\end{proposition}

Our assumptions do not determine whether $2\bmod \lambda$ is a square. In fact,
 by \cite[Theorem 12.14(3)]{lemmermeyer}, given $\lambda \equiv -1\bmod 4 \cO_{F_0}$, then $2\bmod \lambda$ is a square if $\lambda\equiv -1, 3\bmod 8 \cO_{F_0}$, and is not a square if $\lambda\equiv -1+4u, -1+4u^\tau\bmod 8 \cO_{F_0}$.

\begin{proof}
    The proof is analogous to that of Proposition \ref{asquare}. 
    We use \Cref{aut} to conclude that $\#\Gamma$ for $A_R \bmod \lambda$ is even.
    By Proposition \ref{Hilbert}, the roots of 
   $\cQ^\circ_\lambda(x)$ are in the ring class field of $E$ associated to the order $\cO_F[\sqrt{{-\lambda}}]$, where the prime $v=2$ of $F_0$ is ramified.
    \end{proof}

By Proposition~\ref{asquare2}, 
we can choose $d'_\lambda\in \{d_\lambda, 2d_\lambda\}$ 
which is a square modulo $\lambda$.
By definition, $a_\lambda \cP_\lambda(x) \in \cO_{F_0}[x]$ and $d'_\lambda \cQ^\circ_\lambda(x)\in \cO_{F_0}[x]$.
Recall $c=\frac{27}{4}(u\sqrt{5})^{-5}\in \cO_{F_0}$.

\begin{proposition}\label{coroSQ}
Denote the reduction of $c$ modulo $\lambda$ by $\bar{c} \in\bF_\lambda$,
and the reduction of $a_\lambda \cP_\lambda(x) \in \cO_{F_0}[x]$
modulo $\lambda$ by $\bcP_{\lambda}(x) \in\bF_\lambda[x]$.
%$\bar{c}, \bcP_{\lambda}(x) \in\bF_\lambda[x]$ the reduction modulo $\lambda$ of $c, a_\lambda \cP_\lambda(x) \in \cO_{F_0}[x]$.
Then ${(x-\bar{c}) \bcP_{\lambda}(x)}$ is a square in $\bF_\lambda[x]$.  

Similarly, denote the reduction modulo $\lambda$ of  
$d'_\lambda \cQ^\circ_\lambda(x)\in \cO_{F_0}[x]$
by $ \bcQo (x) \in\bF_\lambda[x]$. 
Then ${(x-\bar{c}) \bcQo(x)}$ is a square in $\bF_\lambda[x]$.  
\end{proposition}
\begin{proof}
We will prove that $(x-\bar{c})\bcP_{\lambda}(x)$ is a square in $\bF_\lambda[x]$; the proof of the assertion for ${(x-\bar{c}) \bcQo(x)}$ is the same (with Proposition \ref{asquare2} replacing Proposition \ref{asquare}).  

By Proposition \ref{asquare},
$a_\lambda\bmod\lambda$ is a  square. 
Hence, if $\mathrm{val}_\lambda (a_\lambda)=0$, then  the statement follows from %since  $a_\lambda\bmod\lambda$ is a square, 
Proposition \ref{RootInPairs}.

Assume $\mathrm{val}_\lambda (a_\lambda)>0$. Then, the statement follows from Proposition \ref{RootInPairs} combined with Lemmas  \ref{L1} and \ref{L2}. Indeed, let $\beta_1, \beta_2\in H_\lambda$ (the Hilbert class field of $E$)
be a conjugate  pair of geometric points in the support of $Z(\lambda)$ which lie the $\lambda$-adic neighborhood of $A_P\bmod \lambda$. Write $\delta_1=\varpi^n\beta_1$ and $\delta_2=\varpi^n\beta_2$, where $\varpi=\sqrt{-\lambda}$ and $n=\mathrm{val}_\nu(1/\beta_1)$. By Lemma \ref{L2}, $\mathrm{val}_\nu(1/\beta_2)=n$ and $\mathrm{val}_\nu(1/\beta_1-1/\beta_2)>n$. Hence, modulo $\sqrt{-\lambda}$,
\begin{eqnarray*}
(\varpi^nx-\delta_1)(\varpi^nx-\delta_2) & \equiv & 
 \delta_1\delta_2 \equiv
  \delta_1(\delta_1+( \delta_2- \delta_1)) \\
  & \equiv & 
 \delta_1( \delta_1+( \varpi^n \beta_1\beta_2(1/\beta_1- 1/\beta_2))\equiv
 \delta_1^2.
\end{eqnarray*}
Since $\cP_{\lambda}\in F_0[x]$, we apply the above computation to the entire Galois orbit of $\beta_1$ under $\Gal(H_\lambda/F_0(\sqrt{-\lambda}))$, to obtain the desired assertion.
%SAVE: note that the pair $\beta_1, \beta_2$ cannot be in the same Galois orbit because $\sqrt{-\lambda}$ is unramified in $H_\lambda/F_0(\sqrt{-\lambda}$; thus we obtain congruence between Galois orbits and hence get a square in the base field $\bF_\lambda$.
\end{proof}

\begin{lemma}\label{L2}
  Let $\beta, \beta'$ denote a conjugate pair of geometric points in the support of $Z(\lambda)$ (resp. $W(\lambda)$) which are in the $\lambda$-adic neighborhood of $A_P\bmod \lambda$ (resp. $A_R\bmod \lambda$).

  Let $\fp$ be a prime of $H_\lambda$ (resp. the ring class field of $E$ associated to $\cO_F[\sqrt{-\lambda}]$) above $\lambda$.
Then  
$$\mathrm{val}_\fp(1/\beta)=\mathrm{val}_\fp(1/\beta') \text{ and } \mathrm{val}_\fp(1/\beta- 1/\beta')> \mathrm{val}_\fp(1/\beta).$$
\end{lemma}

\begin{proof}
We shall prove the assertion for $Z(\lambda)$; the same proof holds for $W(\lambda)$. 
%SAVE: because the proof uses the order $2$ automorphism and the $p$-divisible group structure and $A_P[p^\infty]_{\overline{\bF}_\lambda}\cong A_R [p^\infty]_{\overline{\bF}_\lambda}$; both properties are shared by $P,R$.

We start by showing $\mathrm{val}_\fp(1/\beta)=\mathrm{val}_\fp(1/\beta')$. 
As in the proof of Proposition \ref{asquare}, denote 
$\Gamma=\mathrm{Aut}_{{\cO}_F}(A_{P})/(\{\pm 1\}) \times \mu_5)$.
Then, as in the proof of \Cref{asquare}, we have $\mathrm{val}_\fp(1/\beta)=\# \Gamma\cdot\mathrm{val}_\fp (t)$ and $\mathrm{val}_\fp(1/\beta')=\# \Gamma\cdot \mathrm{val}_\fp (t')$, where  
$t, t'$ are the corresponding values of a local parameter around $P$ on the smooth Deligne--Mumford stack $\Sh$. 
Then it suffices to show that $\mathrm{val}_\fp(t)=\mathrm{val}_\fp(t')$. 

We denote by $A_\beta, A_{\beta'}$ the CM abelian varieties corresponding to $\beta,\beta'$ respectively. 
Recall notation from the proof of Proposition \ref{RootInPairs}.
By construction, the pair $(\beta,\beta')$
corresponds to a conjugate pair of normalized embeddings $\cO_{F_0}[\sqrt{-\lambda}] \to \cO_B$, mapping 
$\sqrt{-\lambda}\mapsto \pm\alpha\in\cO_B$.
By Serre--Tate theory, our construction of $\beta, \beta'$, 
and \cite[Proposition~3.3]{Gross}, we have
\[\End_{\cO_F,{\mathrm{pol}}}(A_\beta \bmod \fp^n)= \End_{\cO_F,{\mathrm{pol}}}(A_{\beta'} \bmod \fp^n) =\cO_F\left(\cO_{F_0}[\pm \alpha] + \fp^{n-1} \cO_{B}\right).\] 

By definition, $\mathrm{val}_\fp(t)$ (resp. $\mathrm{val}_\fp(t')$) is the largest integer $n$ such that there exists a non-scalar (trace $0$) element of order $2$ in $\End_{\cO_F,{\mathrm{pol}}}(A_\beta \bmod \fp^n)$ (resp. $\End_{\cO_F,{\mathrm{pol}}}(A_{\beta'} \bmod \fp^n)$). 
%SAVE: admitting such an endomorphism directly construct a isogeny (of index $2$ -- here we mean lie in $2T_2$) from the given point to product of two abelian surfaces CM by $\cO_F$; thus such a point is isomorphic to $A_P \bmod \fp^n$ or $A_R \bmod \fp^n$; these two points have distinct reductions $\bmod \lambda$, thus we know which one we get depending on which $\lambda$-adic disc/neighborhood we start from.
Since the endomorphism algebras of $A_\beta\bmod \fp^n$ and $A_\beta' \bmod \fp^n$ agree, we deduce the equality $\mathrm{val}_\fp(t)=\mathrm{val}_\fp(t')$. 

By the definition of a normalized embedding, the element $\alpha\in\cO_B$  acts as $\sqrt{-\lambda}$ on ${\mathrm{Lie}} (A_\beta)\bmod\fp^2$ and as  $-\sqrt{-\lambda}$ on ${\mathrm{Lie}} (A_{\beta'})\bmod\fp^2$. We deduce that $t\not\equiv t'\bmod \fp^2$, hence $\mathrm{val}_\fp(t-t')=1$, and $\mathrm{val}_\fp(t)=\mathrm{val}_\fp(t')=1$.

It remains to show $\mathrm{val}_\fp(1/\beta'- 1/\beta)> \mathrm{val}_\fp(1/\beta)$, where $\mathrm{val}_\fp(1/\beta)=\# \Gamma$.
From \eqref{eqt}, we deduce\footnote{Indeed, by interpreting the valuations in terms of local intersection numbers between corresponding divisors, we can prove that $\mathrm{val}_\fp(\prod_{\gamma\in \Gamma} (t')^\gamma-\prod_{\gamma\in \Gamma} t^\gamma)  = \sum_{\gamma\in \Gamma} \mathrm{val}_\fp(t'-t^\gamma)$; from this equality, we can also deduce that this value is no less than $\max_{\gamma\in \Gamma} \mathrm{val}_\fp(t'-t^\gamma) + \#\Gamma -1$.} 
$$\mathrm{val}_\fp(1/\beta'- 1/\beta)=\mathrm{val}_\fp(\prod_{\gamma\in \Gamma} (t')^\gamma- \prod_{\gamma\in \Gamma} t^\gamma)  \geq \max_{\gamma\in \Gamma} \mathrm{val}_\fp(t'-t^\gamma) + \#\Gamma -1.$$
One way to see this inequality is to use the triangle inequality, and the fact that for any $I\subset \Gamma$, any $\gamma_0 \in \Gamma \setminus I$, 
and any $\gamma_1 \in \Gamma$, we have 
\begin{align*}
   & \mathrm{val}_\fp(\prod_{\gamma\in I \cup \{\gamma_0\}} (t')^\gamma \prod_{\gamma\in \Gamma\setminus (I \cup \{\gamma_0\})} t^{\gamma_1\gamma}- \prod_{\gamma\in I} (t')^\gamma \prod_{\gamma\in \Gamma\setminus I} t^{\gamma_1\gamma})\\
    = &\sum_{\gamma\in I} \mathrm{val}_\fp((t')^\gamma) + \sum_{\gamma \in \Gamma\setminus (I \cup \{\gamma_0\})} \mathrm{val}_\fp(t^{\gamma_1\gamma}) +  \mathrm{val}_\fp((t')^{\gamma_0}-t^{\gamma_1\gamma_0}) \\
    = &\#\Gamma -1 + \mathrm{val}_\fp(t'-t^{\gamma_1}). 
\end{align*}
Above, we used that the $\Gamma$-action preserves valuations.
In particular, $\mathrm{val}_\fp(1/\beta'- 1/\beta)\geq \# \Gamma$,
since $\mathrm{val}_\fp(t'-t^\gamma)\geq 1$, for all $\gamma\in\Gamma$. Furthermore, to establish the inequality 
$\mathrm{val}_\fp(1/\beta'- 1/\beta)> \# \Gamma$, it suffices to show
$\mathrm{val}_\fp(t'-t^\gamma)> 1$ for some $\gamma\in \Gamma$.

As in the proof of Lemma \ref{L1}, let
$\epsilon_0\in {\mathrm{Aut}}_{{\cO}_F}(A_{P})$ be the non-scalar element of order $2$ (which is unique modulo the center).  We may assume $\epsilon_0\mapsto \begin{bmatrix}
    1 & 0 \\ 0 & -1
\end{bmatrix}\in \bZ_{p^2}\cO_D$.
We claim that $\mathrm{val}_\fp(t'-t^{\epsilon_0})> 1$. 
By definition, $t^{\epsilon_0}$ is the value of the local parameter corresponding to the normalized embedding 
$\cO_{F_0}[\sqrt{-\lambda}] \to \cO_B$, mapping 
$\sqrt{-\lambda}\mapsto \alpha^{\epsilon_0}$, where $\alpha^{\epsilon_0}=\begin{bmatrix}
    1 & 0 \\ 0 & -1
\end{bmatrix} \alpha \begin{bmatrix}
    1 & 0 \\ 0 & -1
\end{bmatrix}\in\cO_D$.
With notation as in Section \ref{mathcalH}, we write an element $\alpha\in \cO_D$ as $$\alpha=\begin{bmatrix}
    y & -\lambda x^\gamma \\
   x  & y^\gamma\\
\end{bmatrix},$$
where $x, y\in\ZZ_{p^2}$. 
From $\alpha^2=-\lambda$,  we deduce  $y^2-\lambda xx^\gamma=-\lambda$ and $y+y^\gamma=0$. From the second equality, we deduce $y=y_0 \delta$, where $\delta \in\ZZ_{p^2}^\times$ satisfies $\delta^\gamma=-\delta$ and $y_0\in \ZZ_p$. 
From the first equality, we deduce $\mathrm{val}_\fp(y)>0$. Hence, $\mathrm{val}_\fp(y_0)>0$, and since $y_0\in\ZZ_p$ then $y_0\in\fp^2$.

By direct computation, we have $$\alpha^{\epsilon_0}= \begin{bmatrix}
    y_0\delta & \lambda x^\gamma \\
   -x  & -y_0 \delta\\
\end{bmatrix}.$$
We deduce $-\alpha\equiv \alpha^{\epsilon_0}\bmod \fp^2\cO_D$, and hence $\mathrm{val}_\fp(t'-t^{\epsilon_0})> 1$.
\end{proof}

%%%%%%%%%%%%%SECTION 10

\section{Proof of the main theorem}

We briefly review key notation.

\begin{notation}
\label{Nfinalnotation}
Recall that $F_0=\QQ(\sqrt{5})$ and $u=(1+\sqrt{5})/2$.
Let $c=(27/4)(u\sqrt{5})^{-5}\in F_0$.

For $t \in \CC -\{0,1\}$, recall the Klein $j$-function
$J(t) = (t^2-t+1)^3/t^2(t-1)^2$ from 
\eqref{Ejfunction}, and its normalization $j_t = (u\sqrt{5})^{-5} J(t)$
from \eqref{Eqlittlej}.
\end{notation}

%SAVE We note that $c \approx .0109$ and $c^\tau \approx 1.3391$.

\begin{theorem}\label{thm:M11}
For $t \in \CC -\{0,1\}$,
let $C =C_t$ be the smooth projective curve with affine model defined by $y^5=x(x-1)(x-t)$.  Let $j_t$ and $c$ be as in Notation~\ref{Nfinalnotation}.
Assume $j_t \in F_0$.  Assume:
	\begin{enumerate}
        \item $c-j_t$ is totally positive;
		\item $ \mathrm{val}_{u\sqrt{5}}(j_t-c) \in 2\mathbb{Z}$; and
		\item $ \mathrm{val}_{u\sqrt{5}}(j_t)< 0$. 
			\end{enumerate} 
	Then there exist infinitely many primes of $F_0$ at which the reduction of $\mathrm{Jac}(C)$ is basic.
\end{theorem}

    By Proposition \ref{Predlehr}, assumption (3) holds if and only if $C$ does not have potentially good reduction at the prime of $F_0$ above $5$. 
   
\begin{remark} \label{RswitchJj} 
Note that $J(t) \in F_0$ if and only if $j_t\in F_0$.
Also, the first two assumptions in Theorem \ref{thm:M11} are equivalent to  \begin{enumerate}
		\item $\frac{27}{4}-J(t)$ is totally positive; and
        \item $\mathrm{val}_{u\sqrt{5}}( J(t)-\frac{27}{4}) \in 2\mathbb{Z}$.\end{enumerate}
In particular, they hold true if $J(t) \in \QQ \cap (-\infty, 27/4)$.
\end{remark}

Recall that $\tau$ is the non-trivial 
automorphism in $\mathrm{Gal}(F_0/\QQ)$.

\begin{lemma}\label{Ltau}
\begin{enumerate}
\item If $j_t \in F_0$, then the isomorphism class of $C$ is defined over $F_0$.

\item The points of the $j$-line corresponding to $C$ and $ C^\tau$ are $j_t$ and $\xi_t:= u^{-10}j_t^\tau$ respectively.

\item The value $c-j_t$ is totally positive if and only if 
these two points lie on the
arch $\overset{\frown}{PQR}$.
\end{enumerate}
\end{lemma}

\begin{proof}
By Notation~\ref{Nfinalnotation}, $J(t)= j_t(u\sqrt{5})^5\in \CC$. 
\begin{enumerate}
\item If $j_t \in F_0$, then $J(t)\in F_0$. 
By Lemma \ref{Lfomfod}, since $J(t)\in\overline{\QQ}$, the isomorphism class of $C$ is defined over $F_0$.

\item The isomorphism class of $C^\tau$ is given by $J(t)^\tau$. 
Hence, the points of the $j$-line representing $C$ and $ C^\tau$ are respectively $j_t=(u\sqrt{5})^{-5}J(t)$ and 
\[\xi_t:=(u\sqrt{5})^{-5}J(t)^\tau = (u\sqrt{5})^{-5} j_t^\tau (-u^\tau \sqrt{5})^5 = u^{-10} j_t^\tau.\]

\item By Remark~\ref{RswitchJj} and Lemma~\ref{Lfomfod}, 
the point representing $C$ lies on the arch $\overset{\frown}{PQR}$
if and only if $J(t)< 27/4$ 
(or, equivalently, $j_t < c$). 
Similarly, the point $C^\tau$ lies on the arch $\overset{\frown}{PQR}$
if and only if $J(t)^\tau < 27/4$
(or, equivalently, $j_t^\tau < c^\tau= u^{10}c$).
These two conditions are satified if and only if $c-j_t$ is totally positive
\end{enumerate}
\end{proof}

\begin{proof}[Proof of \Cref{thm:M11}]
Write $C =C_t$ and $j=j_t\in F_0$. If $C=M$, then by \Cref{MisCM}, its Jacobian has complex multiplication and hence it has basic reduction 
at infinitely many primes by Shimura--Taniyama. For the rest of the proof, we assume $C, C^\tau \neq M$.

We prove the statement by contradiction. 
Assume $\mathrm{Jac}(C)$ has basic reduction at only finitely many primes of $F_0$. 
Let $\mathcal{S}$ be a finite set of primes of $F_0$, 
such that $\mathcal{S}=\mathcal{S}^\tau$,
containing all primes for which the reduction of $\mathrm{Jac}(C)$ is basic, 
the primes $2$, $3$, and $u\sqrt{5}$, and all primes $v$ 
if either $\mathrm{val}_v(j)\neq 0$ or $\mathrm{val}_v(j-c)\neq 0$. 
    %SAVE: the reason to rule out $3$ is we do not want to take $\lambda=3$ to get $A_P$.
    
Our goal is to construct a prime $v \notin \mathcal{S}$ at which 
$\mathrm{Jac}(C)$ has basic reduction. 

Consider the points on the $j$-line associated with $C$ and $C^\tau$; 
by Lemma~\ref{Ltau}, they are $j$ and $\xi_t= u^{-10}j^\tau$. 
By hypothesis, they both lie on the arch $\overset{\frown}{PQR}$. 

By applying Theorem~\ref{thm:infintely-many-lambda} to the set $\mathcal{S}\setminus \{u\sqrt{5}\} $, we obtain a set $\Lambda$ of totally positive irreducible elements $\lambda$ in $\cO_{F_0}$, which satisfy the assumptions in Notation \ref{lambda},
and such that the two real points $C_\lambda$ and $C_{\lambda^\tau}$
having complex multiplication by $\cO_F[\sqrt{-\lambda}]$ lie on the arch $\overset{\frown}{PQR}$ with desired location to be specified below.  The condition $\lambda\neq \lambda^\tau$ implies 
that $C_\lambda, C_{\lambda^\tau} \not = M$.
There are two cases:

Case (A): $C$ and $C^\tau$ are on the same side on $M$, 
meaning they are both on $\overset{\frown}{PM}$
or both on $\overset{\frown}{MR}$; 
without loss of generality, we suppose 
that $C$ and $C^\tau$ are both on $\overset{\frown}{MR}$; 
the proof in the other case is very similar; or

Case (B): $C$ and $C^\tau$ are on the opposite sides of $M$;
without loss of generality, we suppose that $C$ is on $\overset{\frown}{MR}$ 
and $C^\tau$ is on $\overset{\frown}{PM}$.

In case (A), 
we can suppose that $C_\lambda$ (resp.~$C_\lambda^\tau$) is closer to $M$ (resp.~$R$) than  
any of $\{C, C^\tau\}$;
\footnote{To measure distance, we lift to the geodesic segment $\tilde{P} R_1$ in 
$\bH$ and use the hyperbolic distance.}
In case (B), we can suppose that $C_\lambda$, $\eta^{-1} C_\lambda$, and $C_\lambda^\tau$ are all closer to $M$ than any of $\{C, C^\tau\}$.
Note that either $C_\lambda$ or  $C_{\lambda^\tau}$, or both, might have multiplication by the maximal order $\cO_F[\frac{1+\sqrt{-\lambda}}{2}]$.
We say that $C_\lambda$ and $C_{\lambda^\tau}$ have the same multiplication type 
if both have CM by $\cO_E$ or both do not have CM by $\cO_E$. 

Let $\cP_\lambda, \cQ_\lambda^\circ$ be the polynomials given in Definition \ref{poly} and \Cref{SswitchPR}. We claim that (at least) one of $\cP_\lambda(j)\cP_{\lambda^\tau}(j_{\tau})$ and $\cQ_\lambda^\circ(j)\cQ_{\lambda^\tau}^\circ (j_{\tau})$ is negative. 
Indeed, if 
$C_\lambda$ and $C_{\lambda^\tau}$ have the same multiplication type, then by Lemma~\ref{Poddrealroot}, from the relative position of
the points in $\{C, C^\tau, C_\lambda, C_\lambda^\tau\}$,
we deduce that (at least) one of $\cP_\lambda(j)\cP_{\lambda^\tau}(j_{\tau})$ and $\cQ_\lambda^\circ(j)\cQ_{\lambda^\tau}^\circ(j_{\tau})$  is negative. More precisely, $\cP_\lambda(j)\cP_{\lambda^\tau}(j_{\tau})$ is negative if $C_\lambda, C_{\lambda^\tau}$ both have multiplication by the maximal order $\cO_F[\frac{1+\sqrt{-\lambda}}{2}]$, and $\cQ_\lambda^\circ(j)\cQ_{\lambda^\tau}^\circ(j_{\tau})$ is negative otherwise.

Therefore, it remains to consider the case when 
$C_\lambda$ and $C_{\lambda^\tau}$ have different multiplication types.
By Proposition~\ref{CMbyorderQP},
after replacing the point 
$C_\lambda$ with its image under $\eta^{-1}$ (or, equivalently, $\eta$), 
we can ensure that $C_\lambda$ and $C_{\lambda^\tau}$ have the same multiplication type, without changing their relative position with respect to $C$ and $C^\tau$.
We illustrate this modification for case (A) in Figure~\ref{fig:schematic1}
and for case (B) in Figure~\ref{fig:schematic2}; 
(the location of $C_\lambda^\tau$ does not impact the argument).

\begin{figure}[h!]
\begin{center}
\begin{tikzpicture}
        \draw (-5,0)--(5,0);

%The point $P$
        \node[below] at (-4,0) {\footnotesize{$P=\eta^{-1}R$}};
        \node[above] at (-4,0) {\footnotesize{$\infty$}};
        \draw [fill] (-4,0) circle [radius=0.04];
        
         % \node[below] at (-1,0) {\footnotesize{$\eta^{-1}Q$}};
        %\draw [fill] (-1,0) circle [radius=0.04];

%The point $M$
        \node[below] at (0,0) {\footnotesize{$M$}};
        %\node[above] at (0,0) {\footnotesize{$?$}};
        \draw [fill] (0,0) circle [radius=0.04];

%The point $Q$.  The ratio of lengths $PQ$ and $PR$ is accurate.        
        \node[below] at (1.82,0) {\footnotesize{$Q$}};
        \node[above] at (1.82,0) {\footnotesize{$0$}};
        \draw [fill] (1.82,0) circle [radius=0.04];

%The point $R$.        
        \node[below] at (4,0) {\footnotesize{$R$}};
        \node[above] at (4,0) {\footnotesize{$c$}};
        \draw [fill] (4,0) circle [radius=0.04];

        \node[below] at (2.5,0) {\footnotesize{$C^\tau$}};
        \draw[red] [fill] (2.5,0) circle [radius=0.04];

        \node[below] at (1.2,0) {\footnotesize{$C$}};
        \draw[red] [fill] (1.2,0) circle [radius=0.04];

        \node[below] at (3,0) {\footnotesize{$C_\lambda^\tau$}};
        \draw[blue] [fill] (3,0) circle [radius=0.04];

        \node[below] at (0.7,0) {\footnotesize{$C_{\lambda}$}};
        \draw[blue] [fill] (0.7,0) circle [radius=0.04];

        \node[below] at (-.7,0) {\footnotesize{$\eta^{-1}C_{\lambda}$}};
        \draw[blue] [fill] (-.7,0) circle [radius=0.04];

        \draw (0.7,0.2) .. controls (0,0.7) and (0,0.4) .. (-.7,0.2);
        \draw (-.7,0.2)--(-.6,0.4);
        \draw (-.7,0.2)--(-.6,0.1);

        \node[right] at (5,0) {$\Sh(\mathbb{R})$};
    \end{tikzpicture}
\end{center}
\caption{Schematic of the arch $\overset{\frown}{PQR}$, with modification in Case (A)}
    \label{fig:schematic1}
\end{figure}
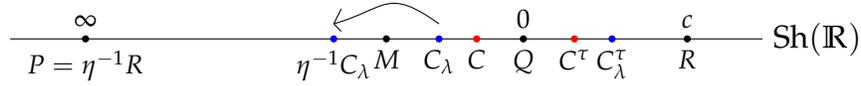

\begin{figure}[h!]
\begin{center}
\begin{tikzpicture}
        \draw (-5,0)--(5,0);

%The point $P$
        \node[below] at (-4,0) {\footnotesize{$P=\eta^{-1}R$}};
        \node[above] at (-4,0) {\footnotesize{$\infty$}};
        \draw [fill] (-4,0) circle [radius=0.04];
        
         % \node[below] at (-1,0) {\footnotesize{$\eta^{-1}Q$}};
        %\draw [fill] (-1,0) circle [radius=0.04];

%The point $M$
        \node[below] at (0,0) {\footnotesize{$M$}};
        %\node[above] at (0,0) {\footnotesize{$?$}};
        \draw [fill] (0,0) circle [radius=0.04];

%The point $Q$.  The ratio of lengths $PQ$ and $PR$ is accurate.        
        \node[below] at (1.82,0) {\footnotesize{$Q$}};
        \node[above] at (1.82,0) {\footnotesize{$0$}};
        \draw [fill] (1.82,0) circle [radius=0.04];

%The point $R$.        
        \node[below] at (4,0) {\footnotesize{$R$}};
        \node[above] at (4,0) {\footnotesize{$c$}};
        \draw [fill] (4,0) circle [radius=0.04];

        \node[below] at (-2,0) {\footnotesize{$C^\tau$}};
        \draw[red] [fill] (-2,0) circle [radius=0.04];

        \node[below] at (1.5,0) {\footnotesize{$C$}};
        \draw[red] [fill] (1.5,0) circle [radius=0.04];

        \node[below] at (1,0) {\footnotesize{$C_\lambda^\tau$}};
        \draw[blue] [fill] (1,0) circle [radius=0.04];

        \node[below] at (0.6,0) {\footnotesize{$C_{\lambda}$}};
        \draw[blue] [fill] (0.6,0) circle [radius=0.04];

        \node[below] at (-.7,0) {\footnotesize{$\eta^{-1}C_{\lambda}$}};
        \draw[blue] [fill] (-.7,0) circle [radius=0.04];

        \draw (0.6,0.2) .. controls (0,0.7) and (0,0.4) .. (-.7,0.2);
        \draw (-.7,0.2)--(-.6,0.4);
        \draw (-.7,0.2)--(-.6,0.1);

        \node[right] at (5,0) {$\Sh(\mathbb{R})$};
    \end{tikzpicture}
\end{center}
\caption{Schematic of the arch $\overset{\frown}{PQR}$, with modification in Case (B)}
    \label{fig:schematic2}
\end{figure}

Therefore, after this modification,
we also have that (at least) one of $\cP_\lambda(j)\cP_{\lambda^\tau}(j_{\tau})$ and $\cQ_\lambda^\circ(j)\cQ_{\lambda^\tau}^\circ (j_{\tau})$ is negative.

Assume $\cP_\lambda(j)\cP_{\lambda^\tau}(j_{\tau})<0$; (the other case is similar and is discussed later). 
From Lemma~\ref{Ltau}, the points of the $j$-line corresponding to $C$ and $C^\tau$ are $j_t$ and $\xi_t = u^{-10}j_t^\tau$.
We deduce that $\cP_{\lambda^\tau}(j_{\tau})=u^{-10n}(\cP_{\lambda}(j))^{\tau}$, where $n=\deg \cP_\lambda$, which is odd by Lemma~\ref{Poddrealroot}. 
Hence $\left(\cP_{\lambda}(j)\right)\left(\cP_{\lambda}(j)\right)^\tau<0$. We choose $\epsilon \in \{u,u^\tau\}$ such that $\epsilon \cP_\lambda(j)$ is totally negative. 

As in Section \ref{denominator}, 
let $a_\lambda \in \cO_{F_0}$ be the
totally positive least common multiple  of the denominators of the coefficients of $\cP_\lambda(x) \in F_0[x]$.

Consider the value
$V:=\epsilon a_\lambda (j-c) \cP_\lambda(j)$ in $ {F_0}$. By construction, 
it is totally  positive.
By Corollary~\ref{coroSQ} and Lemma~\ref{quadrec5} combined,  
$V$ is either $0$ modulo $\lambda$ or not a square modulo $\lambda$.  
Note that we can reduce 
$V$ modulo $\lambda$ since $a_\lambda \cP_\lambda(x)\in \cO_{F_0}[x]$, 
and since 
$\mathrm{val}_{\lambda}(j)= \mathrm{val}_{\lambda}(j-c)=0$
(because $\lambda\not\in \mathcal{S}$).

If $V \equiv 0\bmod \lambda$ (meaning that $\mathrm{val}_\lambda(\epsilon a_\lambda (j-c) \cP_\lambda(j))>0$),
we have 
\[\mathrm{val}_\lambda( a_\lambda \cP_\lambda(j))=\mathrm{val}_\lambda\left(\prod_{\mathrm{val}_\lambda(\beta)\geq 0}(j-\beta) \prod_{\mathrm{val}_\lambda(\beta)=-1}(\varpi j- \varpi \beta)\right)>0,\] 
where $\beta$ runs through all roots of $\cP_\lambda(x)$.\footnote{By \Cref{L2}, $\mathrm{val}_\lambda(\beta)\geq -1$ for all $\beta$; we only use the
condition $\mathrm{val}_\lambda(\beta)=-1$ for ease of notation. In general, it is possible to multiply $\beta$ by a suitable power of $\varpi$ and the argument remains the same.} Note that $\mathrm{val}_\lambda( \prod_{\mathrm{val}_\lambda(\beta)=-1}(\varpi j- \varpi \beta))=0$. 

Thus there exists a principally polarized abelian variety $A$ (defined over $\overline{F}$)
having multiplication by $\mathcal{O}_F[\frac{1+\sqrt{-\lambda}}{2}]$,
such that the reductions at $\lambda$ of $A$ and $\mathrm{Jac}(C)$ are isomorphic.
By Proposition \ref{ST}, $\lambda$ is a prime of basic reduction of $A$, and hence of $\mathrm{Jac}(C)$, and by construction $\lambda\not \in \mathcal{S}$. 
 
If $V$ is not a square modulo $\lambda$, then 
there exists a place $v$ of $F_0$, which is not a square modulo $\lambda$, such that
$\mathrm{val}_v(V)$ is positive and odd. 
By Lemma \ref{QR}, if $v$ is not a square modulo $\lambda$, then $v$ is not split in $F_0(\sqrt{-\lambda})/F_0$, and  hence $v\not\in \mathcal{S}\setminus \{u\sqrt{5}\}$.  We deduce that either
 $v\neq u\sqrt{5}$ and $\mathrm{val}_v(j-c)=0$ or  $v= u\sqrt{5}$ and $\mathrm{val}_{u\sqrt{5}} (j-c)$ is even by Assumption~(2). 
 In both cases, $\mathrm{val}_v(j-c)$ is even. We conclude that $\mathrm{val}_v(a_\lambda \cP_\lambda(j))>0$. 
The same argument as above shows that roots $\beta$ with $\mathrm{val}_v(\beta)<0$ do not contribute to the positive valuation.

Thus, as in the other case, 
there exists a principally polarized abelian variety $A$ (defined over $\overline{F}$)
having multiplication by $\mathcal{O}_F[\frac{1+\sqrt{-\lambda}}{2}]$,
such that the reductions at $v$ of $A$ and $\mathrm{Jac}(C)$ are isomorphic.
By Proposition \ref{ST} and Lemma \ref{QR},  $v$ is a prime of basic reduction for $A$ and thus for $\mathrm{Jac}(C)$. Finally,  Proposition \ref{CMmod5} implies $v\neq u\sqrt{5}$ and hence $v\not\in \mathcal{S}$. 

Assume $\cQ_\lambda^\circ(j)\cQ_{\lambda^\tau}^\circ(j^{\tau})<0$. The argument in this case is similar, with $a_\lambda$ replaced by $d'_\lambda$ as defined in Proposition \ref{asquare2}. Note that since $\lambda, v\not\in \mathcal{S}\setminus \{u\sqrt{5}\}$, then $\lambda, v\neq 2$.

More precisely, as in Section~\ref{SswitchPR}, we use  $\circ$ to denote the unique automorphism of $\mathbb{P}^1$ which fixes $Q$ and switches $R$ and $P$. By \eqref{circfacts}, $j^\circ={cj}(j-c)^{-1}$ and 
$j^\circ-c=c^2(j-c)^{-1}$.  We deduce that $j^\circ \in F_0$, both $j^\circ$ and $j^\circ-c$ are totally negative, and $\mathrm{val}_{u\sqrt{5}}(j^\circ-c)\in 2\ZZ$. 
 Also,  $\mathrm{val}_v(j^\circ)=\mathrm{val}_v(j^\circ-c)=0$ for all primes  $v\not\in\mathcal{S}$.
The same argument for $\cP_\lambda$ above also shows that  we can choose $\epsilon \in \{u,u^\tau\}$ such that $\epsilon \cQ_\lambda^\circ (j)$ is totally negative. The rest of the argument holds verbatim. 
%Indeed, by Lemma \ref{Lcirc}, we have
%$\cP^\circ_\lambda(j^\circ)\cP_\lambda(j)=\frac{1}{\cP_\lambda (c)}(j^\circ-c)^n\cP_\lambda(j)^2$, where $n=\deg \cP_\lambda$ is odd
%and  $\cP_\lambda(c)$ is totally positive (Lemma \ref{Poddrealroot} and Proposition \ref{CMbyorder}). 
%Since $\epsilon \cP_\lambda(j)$ is totally negative, we deduce that
%By Lemma \ref{Poddrealroot} and Proposition \ref{CMbyorder},  $\cP_\lambda(c)$ is totally positive (recall $j_R=c$), and hence $\epsilon^\tau \cP^\circ_\lambda(j_\circ)$ is also totally negative. 
This completes the proof of \Cref{thm:M11}.
\end{proof}

\bibliographystyle{amsplain}
\bibliography{infbib}

\end{document}